\documentclass[12pt]{article}
\usepackage{geometry,amsmath,amssymb,amsthm,bbm, enumerate}
\newtheorem{corollary}{Corollary}
\newtheorem{lemma}{Lemma}
\newtheorem{proposition}{Proposition}
\newtheorem{theorem}{Theorem}
\newtheorem{example}{Example}

\newcommand{\tmtextit}[1]{{\itshape{#1}}}
\newcommand{\C}{{\mathbb{C}}}

\newcommand{\A}{{\mathbb{A}}}
\newcommand{\OO}{{\mathfrak{o}}}

\newcommand{\ord}{{\rm{ord}}}

\usepackage{mathdots,mathrsfs}
\DeclareMathAlphabet      {\mathbfit}{OML}{cmm}{b}{it}

\newtheorem{definition}{Definition}
\newtheorem{remark}{Remark}

\def\lam{{\lambda}}
\def\Lam{{\Lambda}}

\newcommand{\diag}{{\rm diag}}

\newcommand{\Sp}{{\rm Sp}}
\newcommand{\SO}{{\rm SO}}

\newcommand{\bis}{{\underline{\text{s}}}}
\newcommand{\CE}{{\mathcal {E}}}

\newcommand{\FJ}{{\mathfrak J}}

\newcommand{\FS}{{\mathfrak {S}}}
\newcommand{\Fc}{{\mathfrak c}}
\newcommand{\Fd}{{\mathfrak d}}
\newcommand{\Fo}{{\mathfrak o}}
\newcommand{\Ft}{{\mathfrak t}}
\newcommand{\FD}{{\mathfrak D}}
\newcommand{\FU}{{\mathfrak U}}

\newcommand{\Spin}{{\rm Spin}}
\newcommand{\pr}{{\rm pr}}

\newcommand{\Z}{{\mathbb Z}}

\newcommand{\ud}{\,\mathrm{d}}

\newcommand{\apair}[1]{\left\langle{#1}\right\rangle}

\newcommand{\cpair}[1]{\left\{{#1}\right\}}
\newcommand{\ppair}[1]{\left( {#1} \right)}
\usepackage{tikz}
\newcommand*\circled[1]{\tikz[baseline=(char.base)]{
           \node[shape=circle,draw,inner sep=2pt] (char) {#1};}}
           
\begin{document}

\title{Eisenstein Series on Covers of Odd Orthogonal Groups}
\author{Solomon Friedberg and Lei Zhang}\maketitle
\begin{abstract}We study the Whittaker coefficients of the minimal parabolic
Eisenstein series on the $n$-fold cover of the split odd orthogonal group $SO_{2r+1}$.
If the degree of the cover is odd, then Beineke, Brubaker and Frechette have conjectured
that the $p$-power contributions to the Whittaker coefficients may be computed using the theory of crystal graphs
of type C, by attaching to each path component a Gauss sum or a degenerate Gauss sum
depending on the fine structure of the path.
We establish their conjecture using a combination of automorphic and combinatorial-representation-theoretic
methods.   Surprisingly, we must make use of the type A theory, and the two different crystal graph
descriptions of Brubaker, Bump and Friedberg available for type A based on different factorizations of the long word into
simple reflections.  We also establish a formula for the
Whittaker coefficients in the even degree cover case, again based on crystal graphs of type C.  As a further consequence,
we establish a Lie-theoretic description of the coefficients for $n$ sufficiently large, thereby
confirming a conjecture of Brubaker, Bump and Friedberg. 
\end{abstract}


\medskip


\section{Introduction}\label{introduction}
Let $G$ be a reductive group defined over a number field
and let $P$ be a maximal parabolic subgroup of $G$ with Levi decomposition $P=MN$.
Given an automorphic representation $\pi$ of $M(\A)$, the adelic points of $M$, and a vector $f_\pi$ in the space of $\pi$,
one may construct an Eisenstein series $E(g;s,f_\pi)$ on $G(\A)$.
By the work of Langlands, this series, originally defined for $\Re(s)\gg 0$, has analytic continuation and functional equation.
If $\pi$ is generic then the Whittaker coefficients of $E$ may be expressed in terms of Langlands $L$-functions
for $\pi$, and the continuation and functional equation of these $L$-functions may be obtained from the study
of Eisenstein series---the Langlands-Shahidi method.  If the parabolic subgroup $P$ is not maximal then a similar statement
is true; the Eisenstein series in that case is a function of more than one complex variable.

Now suppose instead that $\widetilde{G}$ is a metaplectic cover of degree $n$ of $G(\A)$.  Such a cover exists
only when the base field contains enough roots of unity.  By the work of M\oe glin and Waldspurger~\cite{MW}, the Eisenstein
series on $\tilde{G}$ also have analytic continuation and functional equation.  Hence so do their Whittaker
coefficients.  However,  the Whittaker functional is not unique and the coefficients may not typically be expressed
in terms of $L$-functions, even in the simplest case of the Eisenstein series induced from the Borel subgroup
(``Borel Eisenstein series").  For example, for the $n$-fold cover of $GL(2)$, each Whittaker coefficient is an infinite
sum of $n$-th order Gauss sums.  These series, first studied by Kubota \cite{Ku},
thus have meromorphic continuation and
functional equation even though they are not (for $n>2$) Eulerian.  This continuation implies that the
arguments of $n$-th order Gauss sums at prime arguments are uniformly distributed on the unit circle
\cite{Pat}.

When the $n$-fold cover of $GL(2)$ is replaced by the $n$-fold cover of $GL(r)$ the Whittaker coefficients
of the Borel Eisenstein series again involve infinite sums of number theoretic
interest, but the structure of the coefficients is considerably more subtle \cite{BBF5}.  Each Whittaker
coefficient is a multiple Dirichlet series whose general coefficient
may be determined from the prime power coefficients by a `twisted' multiplicativity, similarly to the way that a
Gauss sum modulo a composite modulus may be expressed in terms of Gauss sums modulo prime powers.
Moreover, each coefficient at powers of $p$, any good prime, is a sum of products
of $n$-th order Gauss sums modulo powers of $p$ (sometimes degenerate),
which may be described by using {\sl crystal graphs}.  These graphs are attached to representations
of the quantum group $U_q(gl_{r}(\C))$.  The description using crystal
graphs is uniform in both $p$ and $n$.  We shall describe this in more detail below.
When $n$ is sufficiently large many of the Gauss sums become zero and an easier Lie theoretic
description that does not require crystal graphs
may also be given.  In that case the number of nonzero  terms at
powers of a fixed prime $p$ is exactly the order of the Weyl group, and there
is a natural correspondence between these terms and the elements of the Weyl group.
This description was first conjectured by Brubaker, Bump and Friedberg \cite{BBF1}; its
proof in this case follows from \cite{BBF5} using \cite{BBF2}.

It is natural to seek the Whittaker coefficients of Borel
Eisenstein series on covers of groups of other Cartan types, and to establish a relation between
these coefficients and crystal graphs.   This work gives the
first general result.   We treat the case of covers of split orthogonal groups
of Cartan type B.

For $n$, the degree of the cover, sufficiently large, a conjectural
Lie theoretic description of the Whittaker coefficients of Borel
Eisenstein series was presented for all Cartan types by Brubaker, Bump and
Friedberg \cite{BBF1, BBF2}.\footnote{The functional equations for the series were proved, but the
connection to Eisenstein series was only conjectured.}
A new feature arises when there are two root lengths:  there are
two different Gauss sums that appear.  When $n$ is odd these are Gauss sums
attached to two Galois-conjugate $n$-th order residue symbols, but when
$n$ is even these are Gauss sums of orders $n$ and $n/2$.  Once again for $n$ sufficiently
large the nonzero terms at powers of a fixed prime $p$
should correspond naturally to the elements of the Weyl group, and there is a conjectural Lie theoretic
description.  However, for lower degree covers additional terms should arise, which
are in the convex hull of the lattice points giving the powers of $p$ corresponding to the Weyl group elements.

For Eisenstein series on odd degree covers of a split orthogonal group of Cartan type B,
a conjectural description of these terms using crystal graphs was given by
Beineke, Brubaker and Frechette \cite{BeBrF1, BeBrF2}.  The quantum group that
appears should be regarded as attached to the quantized Lie algebra
of the $L$-group.  By a well-known principle going back to Savin, the analogue of the $L$-group for
odd degree covers matches that of the degree one cover, although for even degree covers 
the situation is more ambiguous (see Section 4.4 of \cite{BBCF}).
Hence the conjectural description of Beineke, Brubaker and Frechette involved crystal
graphs of type C.
Beineke, Brubaker and Frechette also proved that their
description was compatible with the Lie theoretic description in \cite{BBF1, BBF2}.
So for odd degree covers there is a conjectured description of the Whittaker coefficients that is uniform in $n$.
However, for $n$ even, no parallel description had been given, even conjecturally.

In this paper we establish the conjecture of Beineke, Brubaker and Frechette
for Borel Eisenstein series on odd degree covers of split orthogonal groups of type B.
That is, we show that their Whittaker coefficients may be described in terms of
multiple Dirichlet series whose $p$-parts are calculated (in a specific way)
using crystal graphs of type C.   In doing so,
we thus also confirm that the conjectured Lie theoretic description of Brubaker, Bump and Friedberg  \cite{BBF1, BBF2}
holds for $n$ sufficiently large, $n$ odd.  In addition,
we give a different formula for the case of even degree covers.  The coefficients are again expressed
in terms of the combinatorial objects associated to crystal graphs of type C, but with a
different assignment of number-theoretic data to these objects.  In fact, this second, new, description is
uniform in $n$.
We use this formula to confirm the conjectured Lie theoretic description for $n$ sufficiently large, $n$ even,
as well.

There is something surprising in our proof of the conjecture of Beineke, Brubaker and
Frechette, and to explain it we must go into more detail about how
crystal graphs are used in describing the Whittaker coefficients.
For type A, there are in fact two different descriptions of the $p$-part in terms of crystal graphs.
These descriptions are based on two different factorizations of the long element of the Weyl group
into simple reflections.  The factorization is used to determine a path from each vertex of the crystal
graph to the lowest weight vector (the edges of the graph indicate the action of the Kashiwara operators
attached to each simple root, and the factorization determines  which edges to use in forming the path).
Then the lengths of the path segments are used to determine a collection
of Gauss sums, which also depend on how the path fits in to the rest of the crystal graph.  See \cite{BBF4},
Chs.\ 2 and 3,
for more details.  For two factorizations which are as far apart as possible, the crystal graph description
is valid.  Moreover, the equality of the terms from the two different factorizations is not a formal equality of summands, but
rather much more intricate.  In fact, this equality is enough to imply the functional equation
for the multiple Dirichlet series \cite{BBF3}.

A direct proof of the equality of these two different type A crystal graph descriptions is given in \cite{BBF4}.
It turns out that many terms in the two descriptions
may be identified after applying the Schutzenberger involution.  However,
the difficult ones are not identified; in fact the number of summands on the two sides
may be different, and the equality requires identities among Gauss sums.    These terms are in a suitable
sense boundary terms for the polytope which is the convex hull of the lattice points describing
all possible paths.  Moreover, the equality
is true for all $n$, but with different reasons for different $n$.  To establish it,
one reduces inductively to
objects called {\it short patterns}.  These patterns are then classified. 
The most difficult terms are exactly the terms that come from the {\it totally resonant case}---see \cite{BBF4}, (6.14)
and following.  These terms must be handled by a subtle combinatorial argument. 

Returning to our situation, the Whittaker coefficients are computed inductively, viewing the Eisenstein
series on $SO_{2r+1}$ as induced from an Eisenstein series on $SO_{2r-1}$.  This yields a complicated inductive
formula for each Whittaker coefficient as an exponential sum involving lower rank coefficients
(Theorem~\ref{inductivesum}).  
There is also an inductive formula obtained from Beineke, Brubaker and Frechette's crystal graph description
(Lemma~\ref{ind-lemma}).
However, these formulae involve {\it different} number-theoretic weight factors, a situation
reminiscent of the two type A descriptions.  Remarkably, this is not
merely an analogy.  The equality of the two type C (on the $L$-group side) expressions is 
then established
{\it using} both the equality of the two different type A crystal graph descriptions and
a classification of type C combinatorial structures (which we again call short patterns) that
is analogous to the type A classification.  (In particular, there is a totally resonant
case that requires special effort.)  This is a new phenomenon.
In type A, the induction in stages methods matched the inductive crystal description exactly 
(\cite{BBF5}, Sect.~8), and there was no need to
establish an equality to bridge the Eisenstein series
and the crystal descriptions.  The need to use the two different type A
descriptions in studying other Cartan types
was also not expected.  

Our inductive computation of the Whittaker coefficients works for covers of all degrees.
However, the identities involving Gauss sums are different in the cases of even
and of odd degree covers, due to the differing
orders of the multiplicative character in
the Gauss sums that arise.  For $n$ odd, the conjecture of Beineke, Brubaker and Frechette prescribes
specific rules for attaching number-theoretic quantities to path lengths on the crystal graph.
These rules are in fact closely related to the type~A rules, and have a description in terms
of the maximality of the path segments with respect to the Kashiwara operators.
As they observe, these rules would not apply to the case that $n$ was even, and indeed our analysis establishing
their conjecture makes use of the assumption that $n$ is odd.  To address the
case $n$ even, we also present a {\it different}
set of rules for attaching number-theoretic quantities to path lengths that is uniform in the degree
of the cover, and we show that this different rule also matches the inductive formula from the Eisenstein series.
The rule we present is more combinatorially complicated, but does simplify for $n$ sufficiently large,
where it yields the conjectured Lie theoretic description of Brubaker, Bump and Friedberg.

Though our approach is global, working over a number field,
our results at primes $p$ may also be interpreted locally.  Indeed, working adelically,
the Whittaker integral of a Borel Eisenstein series may be unfolded to an integral over the adelic
points of the unipotent radical $U$ of the Borel subgroup; however the section
involved is not factorizable, but an infinite sum of factorizable sections.  
Following Brubaker, Friedberg and McNamara
(in preparation), one may use this to prove twisted multiplicativity, and to relate the global Whittaker
coefficients to local Whittaker integrals.  As McNamara has proved \cite{McN}, over a local field $F_v$ the unipotent
subgroup $U(F_v)$ may be decomposed
into cells on which the integrand in the Whittaker integral is constant, and these cells may be given
the combinatorial structure of the crystal $B(-\infty)$.  Thus the Whittaker integral may be evaluated
in terms of a sum over an (infinite) crystal.  However, our result does more: we reduce the support
from $B(-\infty)$ to the finite crystal attached to a representation of a specific highest weight, and we show that
values are number theoretic and may be computed for $n$ odd using the crystal graph recipe of
\cite{BeBrF1}.

In Section~\ref{sect2} we fix the notation, and review some facts about metaplectic groups.  The metaplectic
Eisenstein series are defined in Section~\ref{sect3}.  In Section~\ref{sect4} we 
compute the Whittaker coefficients of these Eisenstein series, using an induction-in-stages argument.
Each Whittaker coefficient is written as a Dirichlet series in several complex variables whose coefficients
satisfy an inductive relation giving the rank $r$ coefficients as sums of rank $r-1$ coefficients (Theorem~\ref{inductivesum}).  
The twisted multiplicativity of these coefficients is established in Section~\ref{sect5}; this allows one to
reduce to the study of the coefficients indexed by powers of a single prime.  
Section~\ref{sect6} has two parts. In Section~\ref{sect61} we
present a combinatorial version of the inductive formula of Section~\ref{sect4} for the $p$-power coefficients.
Then in Section~\ref{sect62} we review the conjectured description of Beineke, Brubaker and Frechette~\cite{BeBrF1,
BeBrF2} which uses crystal graphs, and we recast it in a similar inductive format.  In Section~\ref{sec:p} we establish
their conjecture (Theorem~\ref{thm:main}) by the argument described above.  And in Section~\ref{sec:even},
we establish a second inductive formula that applies in the even cover case as well.  
This formula is then used in Section~\ref{sec9} to prove the 
conjectured Lie theoretic description for $n$ sufficiently large.

The authors express their appreciation to Benjamin Brubaker and Daniel Bump for helpful discussions.
The first named author was supported by grants from the NSF (DMS-1001326) and NSA (H98230-10-1-0183).

\section{Metaplectic Groups}\label{sect2}

We begin by fixing the notation and constructing the $n$-fold cover
of $SO_{2r+1}$.  Fix an integer $n\geq 1$. Let $F$ be a number field such that the group $\mu_{4n}$ of $4n$-th roots
of unity in $F$ has order $4n$. 
(The requirement that $F$ contain $2n$ $2n$-th roots of unity is necessary; containing the 
full group of $4n$-th roots of unity is an assumption made for convenience. See McNamara~\cite{Mc}, Section~13.1,
for more information.)  
Fix a finite set of places $S$ of $F$ containing all
archimedean places, all places ramified over $\mathbbm{Q}$ (in particular
those dividing $n$), and which is sufficiently large that the ring $\mathfrak{o}_S$
of $S$-integers of $F$ is a principal ideal domain and the residue field has at least
4 elements for all $v\not\in S$. For each $v\in S$ let $F_v$ denote the completion of 
$F$ at $v$, and let $F_S = \prod_{v \in S} F_v$. 
 Let $(~,~)_v\in\mu_{2n}$ denote the $2n$-fold Hilbert symbol in $F_v$
and let $(~,~)_S\in\mu_{2n}$ denote the 
product over $S$ of these local Hilbert symbols: $(~,~)_S=\prod_{v\in S}(~,~)_v$.  

The metaplectic groups are constructed using a two-cocycle, as in Matsumoto \cite{Mat}.
In \cite{BBF5}, Brubaker, Bump and Friedberg, following earlier authors, described a specific cocycle 
$\sigma_0$ in $H^{2}({\rm GL}_{2r+1}(F_{S}),\mu_{2n})$. Its 
calculation involves the arithmetic of $F$ and is expressed in terms of products of Hilbert symbols
$(~,~)_S$.  We will work with this cocycle composed with an inner automorphism of ${\rm GL}_{2r+1}$:
$$
g\to wgw^{-1},\qquad
w=\begin{pmatrix}
I_{r+1}&&&\\&&&1\\&&\iddots&\\&1&&
\end{pmatrix}.
$$
Let $\sigma\colon {\rm GL}_{2r+1}(F_{S})\times {\rm GL}_{2r+1}(F_{S})\to \mu_{2n}$ be the two-cocycle given by
$$\sigma(g_1,g_2)=\sigma_0(wg_1w^{-1},wg_2w^{-1}).$$  For example, on the torus one has
\begin{equation}\label{torus-cocycle}
\sigma\ppair{w^{-1}\begin{pmatrix}
x_{1}&&\\&\ddots&\\&&x_{2r+1}
\end{pmatrix}w, w^{-1}\begin{pmatrix}
y_{1}&&\\&\ddots&\\&&y_{2r+1}
\end{pmatrix}w}
=\prod_{i>j}(x_{i},y_{j})_{S}.
\end{equation}

Let $\widetilde{GL}_{2r+1}(F_S)$ denote the $2n$-fold cover of the group $GL_{2r+1}(F_S)$ 
corresponding to $\sigma$.  
Thus the elements of $\widetilde{GL}_{2r+1}(F_S)$ are ordered pairs $(g,\zeta)$ with
$g\in GL_{2r+1}(F_S)$  and $\zeta\in \mu_{2n}$, and the multiplication in $\widetilde{GL}_{2r+1}(F_S)$ is given by
$$(g_1,\zeta_1)(g_2,\zeta_2)=(g_1g_2,\sigma(g_1,g_2)\zeta_1\zeta_2).$$
Let $p:\widetilde{GL}_{2r+1}(F_S)\to GL_{2r+1}(F_S)$ 
denote the projection $p((g,\zeta))=g$.

To work with the $n$-fold metaplectic cover of the split special orthogonal group we restrict the $2n$-fold cover of the general linear group.
This shift in the degree of the cover is due to the cover-doubling phenomenon described in Bump, Friedberg and Ginzburg \cite{BFG}, Section 2.
We shall see it reflected in Lemma~\ref{cover} below, in which all $2n$-fold residue symbols appear as squares.
Let $Q:F_S^{2r+1}\times F_S^{2r+1}\to F_S$ be the quadratic form
$$Q(x,y)=\sum_{i=1}^{2r+1}x_i\, y_{2r+2-i} -\frac{3}{2} x_{r+1}y_{r+1}.$$
Then $Q(x,y)$ is represented by the matrix
$$
J=\begin{pmatrix}&&&&&&1\\&&&&&\iddots&\\&&&&1&&\\&&&-1/2&&&\\&&1&&&&\\&\iddots&&&&&\\1&&&&&& \end{pmatrix}.
$$
Let $G$ or $SO_{2r+1}$ denote the split special orthogonal group which is 
 the stabilizer in $SL_{2r+1}$ of the quadratic form 
$Q(x,y)$.  Thus
$$G_{F_S}=\{g\in SL_{2r+1}(F_S)\mid~ ^t \!g J g=J\}.$$
 Let $\widetilde{G}_{F_S}=p^{-1}(G_{F_S})$.  
Similarly define $G_{F_v}$ and $\widetilde{G}_{F_v}$ for the completions
$F_v$ of $F$ at places $v$ of $F$.  Then $\widetilde{G}_{F_S}$ is isomorphic to the quotient of the direct product 
$\prod_{v\in S} \widetilde{G}_{F_v}$ by the equivalence relation 
identifying the copies of $\mu_{2n}$ at each $v\in S$.

\begin{remark}{\rm  The algebraic
group $SO_{2r+1}$ is not simply connected but has a two-fold covering, namely
the spin group $\rm{Spin}_{2r+1}$.  So 
there  are two two-fold covers naturally attached to $SO_{2r+1}(F_v)$, namely
$\rm{Spin}_{2r+1}(F_v)$ and $\widetilde{G}_{F_v}$.  These covers are not isomorphic.
Indeed,  the spin group is an algebraic group
while the cocycle used to define the two-fold covering $\widetilde{G}_{F_v}$ depends on the arithmetic of the
field $F_v$.   Moreover, though  $\rm{Spin}_{2r+1}$ is a cover of the algebraic group $SO_{2r+1}$,
this cover does not descend to the $F_v$ points, so that in fact $\rm{Spin}_{2r+1}(F_v)$
is only a cover of the subgroup of $SO_{2r+1}(F_v)$ consisting of elements whose spinor norm is a square.

Let $\underline{G}$ be a semisimple connected
simply-connected split algebraic group defined over $F_v$.  Then
Matsumoto \cite{Mat} defined an $n$-fold covering of $\underline{G}(F_v)$.  To explain the relation between Matsumoto's covers and the ones used here, fix $n$ as above.
On the one hand, Matsumoto's result gives
a two-cocycle $\sigma_{M}$ on ${\rm Spin}_{2r+1}(F_{v})$ with values in $\mu_n$.  On the other, 
let $\delta$ be the spinor norm from $G_{F_{v}}$ to $F_{v}^{\times}/(F_{v}^{\times})^{2}$ and denote by $G'_{F_{v}}$ the kernel of $\delta$. There is an exact sequence
$$
1\longrightarrow \{\pm 1\}\longrightarrow \Spin_{2r+1}(F_{v})\stackrel{\pr}{ \longrightarrow} G'_{F_{v}}\longrightarrow 1.
$$
Then restricting the cocycle $\sigma$ defined above to $G'_{F_v}$
and then pulling it back, one obtains a cocycle 
$\sigma_{res}$ on ${\rm Spin}_{2r+1}(F_{v})$: $\sigma_{res}(g,h):=\sigma({\rm pr}(g),{\rm pr}(h))$.  
Referring to Matsumoto's results and using the notation in \cite{Mat}, in order to specify $\sigma_{M}$, we 
choose $c_{\alpha}(s,t)=(s,t)^{-\|\alpha^{\vee}\|}_{n}$ where $\alpha$ is in the set $\Phi(\Spin_{2r+1})$ of roots of $\Spin_{2r+1}$,
the short coroots have length 1, and $c_{\alpha}$ is defined on \cite{Mat}, pg.\ 25.
Using the assumption that $-1\in (F^{\times}_{v})^{2n}$ and Matsumoto's relations on $\sigma_{M}$ in \cite{Mat}, it is 
not difficult to check that the two cocycles $\sigma_{M}$ and $\sigma_{res}$ are in fact equal. 
So the covers $\widetilde{G}_{F_v}$ obtained by embedding into a cover of a general linear group essentially
coincide with those constructed by Matsumoto.}
\end{remark}
\smallskip

If $c,d\in\mathfrak{o}_S$ are relatively prime, for $m$ a divisor of $2n$ let 
$\left(\frac{c}{d}\right)_{m}$ denote the $m$-th power residue symbol on $\mathfrak{o}_S$.  
Here the Hilbert and power residue symbols are normalized as in \cite{BBF5}, Sections 3 and 4.
(The $m$-th power residue symbols depend on the choice of $S$ but we suppress this from the notation.)
Then
we have the reciprocity law
\begin{equation}\label{reciprocity-law}
\left(\frac{c}{d}\right)_{2n}=(d,c)_S\left(\frac{d}{c}\right)_{2n}.
\end{equation}
Also, 
\begin{equation}\label{2n-squared}
\left(\frac{c}{d}\right)_{2n}^2=\left(\frac{c}{d}\right)_{n}.
\end{equation}

We embed $F$ (and in particular $\mathfrak{o}_S$)
into $F_S = \prod_{v \in S} F_v$ diagonally.  We similarly embed $SL_{2r+1}(F)$ into $SL_{2r+1}(F_S)$ diagonally.
Then there is an embedding  $g\mapsto (g,\kappa(g))$ of 
$SL_{2r+1}(\mathfrak{o}_S)$ into $\widetilde{SL}_{2r+1}(F_S):=p^{-1}(SL_{2r+1}(F_S))$
(see \cite{BBF5}, Section 4).  This embedding has the property that if $\beta$ is any positive
root of $SL_{2r+1}$ with respect to the standard Borel subgroup then 
\begin{equation}\label{kubota-hom}
\kappa\left(w^{-1}i_\beta(\begin{smallmatrix}a&b\\c&d\end{smallmatrix})w\right)=
\begin{cases}\left(\frac{d}{c}\right)_{2n}&\text{if $c\neq0$}\\1&\text{if $c=0$.}\end{cases}
\end{equation}
The embedding restricts to give an embedding of $G(\mathfrak{o}_S)$ into $\widetilde{G}_{F_S}$.
Let $\Phi$ denote the set of positive roots of $SO_{2r+1}$ with respect
to the standard Borel subgroup $B$ of upper triangular matrices of $SO_{2r+1}$.  In 
standard notation, $\Phi=\cpair{e_{i}\pm e_{j}, e_{i}\mid  i<j},$ where $e_i$ denotes
the $i$-th standard unit vector in ${\mathbb R}^r$.  For $\alpha\in \Phi$,
let $i_\alpha$ denote the canonical embedding of $SL_2$ into $SO_{2r+1}$
corresponding to $\alpha$.  Then the cover doubling phenomenon noted above is
reflected in the following Lemma.

\begin{lemma}\label{cover}
Let $\alpha\in\Phi$.  Then
\begin{equation}\label{cover-doubling}
\kappa\left(i_\alpha(\begin{smallmatrix}a&b\\c&d\end{smallmatrix})\right)=\begin{cases}
{\left(\frac{d}{c}\right)_{2n}\left(\frac{a}{b}\right)_{2n}}&\text{if $\alpha=e_{i}-e_{j}$ for $i<j$},   \medskip \\
\left(\frac{d}{c}\right)_{n}& \text{if $\alpha=e_{i}+e_{j}$ for $i<j$},  \medskip \\
\left(\frac{d}{c}\right)_n^2&\text{if $\alpha$ is short.}
\end{cases}
\end{equation}
\end{lemma}

Indeed, if $\alpha\in \Phi$ is a long root, then we may write $i_\alpha=i_\beta\,i_{\beta'}$ where $\beta,\beta'$ 
are the positive roots of $SL_{2r+1}$ which restrict to $\alpha$ on the diagonal torus of $SO_{2r+1}\subset SL_{2r+1}$.
Since $\kappa$ is a homomorphism, combining
(\ref{kubota-hom}) and (\ref{2n-squared}) we find that (\ref{cover-doubling}) holds.  If instead $\alpha$ is a short root, 
then $i_\alpha$ is obtained from the symmetric square map.  Indeed, for the short simple root $\alpha$ we have
$$i_\alpha(\gamma)=\begin{pmatrix}
I_{r-1}&&\\&M(\gamma)&\\&&I_{r-1}\end{pmatrix}\qquad \text{with}\qquad
M((\smallmatrix a&b\\c&d\endsmallmatrix))=
\begin{pmatrix} a^2&ab&b^{2}\\2ac&bc+ad&2bd\\c^2&cd&d^{2}\end{pmatrix};$$
for the other short roots the rows and columns are moved accordingly.
From the block-compatibility of the metaplectic cocycle established by Banks, Levy and Sepanski \cite{BLS} (or from the original construction of $\kappa$ 
by Bass, Milnor and Serre \cite{BMS}) it follows that $\kappa(i_\alpha(\gamma))=\kappa(M(\gamma))$.
It may be checked that
$$\kappa\left(M((\smallmatrix a&b\\c&d\endsmallmatrix))\right)=\left(\frac{d^2}{c^2}\right)_{2n}=\left(\frac{d}{c}\right)_n^2.$$
Indeed, this follows from \cite{BBFH1}, Eqn.\ (17), when $n=3$ by an easy computation (after taking into
account that the normalization of $\kappa$ there uses $\left(\tfrac{c}{d}\right)_n$ in place of
$\left(\tfrac{d}{c}\right)_n$), and a similar formula pertains in general
(see \cite{BBFH2}, Section 1).
Thus the Lemma holds in all cases.

\section{Eisenstein Series}\label{sect3}

In this Section we define the Eisenstein series on $\widetilde{G}_{F_S}$. 
Let $B=TU$ be the standard Borel subgroup of $G=SO_{2r+1}$; the torus $T$ is given by
$$T=\{\diag(t_{1},t_{2},\dots,t_{r},1,t^{-1}_{r},\dots,t_{2}^{-1},t_{1}^{-1})\}$$
and the subgroup $U$ consists of the upper triangular unipotent matrices in $G$.
The cocyle restricted to $T_{F_S}$ is given by (\ref{torus-cocycle}).  It is easy to check that  $\widetilde{T}_{F_S}$,
the pullback of $T_{F_S}$ to $\widetilde{G}_{F_S}$, is not abelian.  However, the subgroup
$\Omega= \Fo^{\times}_{S}(F^{\times}_{S})^{n}$ is a maximal isotropic subgroup of $F_S$ with respect to $(~,~)_S$,
and accordingly
the subgroup
$\widetilde{T}_{\Omega}$ consisting of $t$ as above such that $t_{i}\in \Omega$
for all $i$, $1\leq i\leq r$, is an abelian subgroup of $\widetilde{G}_{F_S}$ (and is in fact a maximal abelian
subgroup of $\widetilde{T}_{F_S}$).

Let $\mathbf s=(s_1,\dots,s_r)\in \C^r$.  For $x\in F_S$, let $|x|=\prod_{v\in S}|x_v|_v$.
For \break  $t=\diag(t_{1},t_{2},\dots,t_{r},1,t^{-1}_{r},\dots,t^{-1}_{2},t^{-1}_{1})\in T_{F_S}$,
let 
$$
\FJ(t)=(\prod^{r-1}_{i=1}\prod^{i}_{j=1} |t_{j}|^{2s_{i}})\prod^{r}_{i=1}|t_{i}|^{s_{r}}.
$$
Let $\underline{\text{s}}:G_{F_S}\to\widetilde{G}_{F_S}$ be the trivial section $\underline{\text{s}}(g)=(g,1)$. 
A function $f\colon \widetilde{G}_{F_S}\rightarrow \C$ is {\it genuine} if 
$f((g,\zeta))=\zeta f(\bis(g))$ for all $\zeta\in \mu_{2n}$.  
Let $I({\bf s})$ be the space consisting of all smooth genuine complex-valued functions $f$ on $\tilde{G}_{F_{S}}$ such that
$$
f(\bis(tu)g)=\FJ(t)f(g),
$$ 
for all $t\in \widetilde{T}_{\Omega}$, $u$ in $U(F_S)$ and $g\in \widetilde{G}_{F_S}$. 
Let  $\tilde{G}_{F_{S}}$ act on the space $I({\bf s})$ by right translation. Then $I({\bf s})$ is a representation of $\tilde{G}_{F_{S}}$.
By Lemma 2 in McNamara \cite{Mc}, $I({\bf s})$ has a
$G(\mathfrak{o}_{S})$-invariant function and this is unique up to a constant.  However, we shall consider the full 
induced space below and keep track of the dependence
on the inducing data.

Suppose $f\in I({\bf s})$ with $\Re(s_i)$ sufficiently large.  Since $f\in I({\bf s})$, $f(bg)=f(g)$ for all $b\in B(\mathfrak{o}_S)$, $g\in \widetilde{G}_{F_S}$.  Define the Eisenstein series 
$$E_f(g,\mathbf s)=\sum_{\gamma\in B(\mathfrak{o}_S)\backslash G(\mathfrak{o}_S)} f(\gamma g), \qquad g\in \widetilde{G}_{F_S}.$$
In this formula,  $\gamma\in G(\mathfrak{o}_S)$ is embedded in $\widetilde{G}_{F_S}$ via the map
$\gamma\mapsto (\gamma,\kappa(\gamma))$ as described
above.  Then this series converges absolutely and (for a given flat section)
has analytic continuation to all ${\bf s}\in\C^r$ (\cite{MW}).

\section{An Inductive Formula for the Whittaker Coefficients}\label{sect4}

The goal of this Section is to establish an inductive formula for the Whittaker coefficients of the Borel Eisenstein
series on the $n$-fold cover of $SO_{2r+1}$.  This formula expresses the coefficients in terms of similar coefficients
but with $r$ replaced by $r-1$.

Let $\omega_r$ be the $(2r+1)\times(2r+1)$ matrix with alternating $1$ and $-1$ on the anti-diagonal and zero elsewhere:
$$
\omega_r=\begin{pmatrix}
&&&&1\\&&&-1&\\&&1&&\\&\iddots&&\\
1&&&&&
\end{pmatrix}.
$$
Then $\omega_r$ is a representative for the longest Weyl element in $G_F$. 

The Whittaker coefficients of concern are described as follows. Let $\psi$ be a complex valued additive character 
on $F_S$ whose conductor is exactly $\mathfrak{o}_S$.
(For the existence of $\psi$, see Brubaker and Bump \cite{BB}, Lemma 1.)  Recall that $U$ denotes
the unipotent radical of the Borel subgroup $B$ of $G$; the group $U(F_S)$ embeds in $\widetilde{G}_{F_S}$ 
under the trivial section $\underline{\text{s}}(u)=(u,1)$.  Let $\mathbf m=(m_1,\dots,m_r)$ be an $r$-tuple of nonzero elements
of $\mathfrak{o}_S$, and let 
$$\psi_{\mathbf m}(u)=\psi\left(m_1u_{1,2}+m_2 u_{2,3}+\dots +m_{r-1} u_{r-1,r}+m_r u_{r,r+1}\right).$$
Then our goal is to compute the integral
\begin{equation}\label{Whit-coeff}
W_{\mathbf m}(f,\mathbf s)=\int_{U(\mathfrak{o}_S)\backslash U(F_S)} E_f(\underline{\text{s}}(u),\mathbf s)\,\psi_{\mathbf m}(u)\ud u.
\end{equation}
For later use, for $f\in I({\bf s})$ as above, we also introduce the Whittaker functionals
for sufficiently large $\Re(s_{i})$, 
\begin{align*}
\Lam_{{\bf m},t}(f)&=\FJ(t)^{-1}\int_{U(F_S)}f(\bis(t)\bis(\omega_{r})\bis(u))\,\psi_{\mathbf m}(u)\ud u.
\end{align*}

For $0\neq C\in \mathfrak{o}_S$, let $|C|$ denote the cardinality of $\mathfrak{o}_S/C\mathfrak{o}_S$. 
Also for $1\leq i\leq r$ let $\alpha_i$ denote the $i$-th simple root of (the connected component of) the Langlands
dual group $^LG$, with the long root being $\alpha_r$,
and let $\|\alpha_i\|$ denote the length of $\alpha_i$, normalized so that the short roots have length $1$.
Let $\varepsilon_i=\|\alpha_i\|^2$, so $\varepsilon_{i}=1$ for $1\leq i\leq r-1$ and $\varepsilon_{r}=2$.
We will prove the following result. 

\begin{theorem}\label{inductivesum}
$$W_{\mathbf m}(f,\mathbf s)=\sum_{0\neq C_i\in \mathfrak{o}_S/\mathfrak{o}_S^\times} 
\begin{frac}{H(C_1,\dots,C_r;\mathbf m)}
{|C_1|^{2s_1}\dots |C_r|^{2s_r}}\end{frac}
\Lambda_{\mathbf m, \mathbf C}(f),$$
where $\Lambda_{\mathbf m, \mathbf C}(f)$ is defined as above and 
$$C=\diag(C^{-1}_{1},C^{-1}_{2}C_{1},\dots,C^{-1}_{r-1}C_{r-2},C^{-2}_{r}C_{r-1},1,C^{2}_{r}C^{-1}_{r-1},\dots,C_{1}).$$ 
 The coefficients $H$ satisfy the inductive relation
\begin{align}\label{thesum}
&H(C_{1},\dots,C_{r};m_{1},\dots,m_{r})=\\
&\sum_{\substack{0\neq D_{i}\in \OO_{S}/\OO^{\times}_{S}\\0\neq d_{i}\in \OO_{S}/\OO^{\times}_{S}}}
\sum_{\substack{c_{i}\bmod{d_{i}}\\(c_i,d_i)=1}}\,\,\prod_{k=1}^r \left(\frac{c_k}{d_k}\right)_n
\left(\frac{c_{2r-k}}{d_{2r-k}}\right)_n\cdot
\prod^{r-1}_{i=1}\prod^{2r-1}_{j=2r-i}(d_{2r-i}/d_{i},d_{j})_{S}\cdot \prod^{r}_{i=2}(\Fd_{i-1}/\Fd_{i},D_{i})^{-\varepsilon_{i}}_{S}
\nonumber \\
&\times\psi\left(m_{1}\frac{c_{1}}{d_{1}}+
\sum_{j=1}^{r-1} m_{j+1}\left(\frac{u_{j}c_{j+1}}{d_{j+1}}+(-1)^{\varepsilon_{j+1}}\frac{d_{j}c_{2r-j}u_{2r-1-j}}{d_{j+1}d_{2r-j}}\right)\right) 
|d_{r}|^{2r-2}\prod_{j=1}^{r-1} \left|d^{j-1}_{j}d^{r+j-2}_{r+j}\right| 
\nonumber\\
&\times H\left(D_{2},\dots,D_{r};\frac{m_{2}d_{1}d_{2r-2}}{d_{2}d_{2r-1}},\frac{m_{3}d_{2}d_{2r-3}}{d_{3}d_{2r-2}},\dots,\frac{m_{r-1}d_{r-2}d_{r+1}}{d_{r-1}d_{r+2}},\frac{m_{r}d_{r-1}}{d_{r+1}}\right)\nonumber
\end{align}
where the outer sum is over $D_i$, $2\leq i\leq r$, and $d_i$, $1\leq i\leq 2r-1$, such that (modulo units)
$C_{i}=\Fd_{i}D_{i}$, where $D_{1}=1$ for convenience and 
\begin{equation}\label{eq:C}
\begin{array}{rl}
\Fd_{1}&=d_{1}d_{2}\cdots d_{r-1} d_{r}^{2}d_{r+1} \cdots d_{2r-1},\\
\Fd_j&=d_jd_{j+1}\cdots d_{r-1}d_r^2d_{r+1}\cdots d_{2r-j}d_{2r-j+1}^2\dots d_{2r-1}^2\qquad 2\leq j\leq r-1,\\
\Fd_{r}&=d_{r}d_{r+1}\cdots d_{2r-1},\\
\end{array}
\end{equation}
and such that the following divisibility conditions hold in $\mathfrak{o}_S$:
\begin{align}
d_{j+1}&\vert m_{j+1} d_{j}~\text{for}~1\leq j\leq r-1  \label{eq:div-d} \\
d_{j+1}d_{2r-j}&\vert m_{j+1} d_{j}d_{2r-j-1}~\text{for}~1\leq j\leq r-2,\,\,\ d_{r+1}\vert m_{r}d_{r-1}. \label{eq:div-d-2}
\end{align}
The integers $u_i$, $1\leq i\leq 2r-2$, are chosen so that $c_iu_i\equiv 1 \bmod d_i$ (and the
sum is independent of these choices due to the divisibility conditions).
\end{theorem}

We turn to the proof of the Theorem.
To begin, let us rewrite the Eisenstein series as a maximal parabolic Eisenstein series.
(Representation-theoretically, this step is transitivity of induction.)  Let $P$ be the maximal 
parabolic subgroup of $SO_{2r+1}$ consisting of matrices of the form
$$\begin{pmatrix} *&*&*\\&*&*\\&&*\end{pmatrix}$$
where the middle block is $2r-1$ by $2r-1$.
Then
$$E_f(g,s)=\sum_{\gamma\in P(\mathfrak{o}_S)\backslash G(\mathfrak{o}_S)} 
\Theta_f(\gamma g)\qquad g\in \widetilde{G}_{F_S}$$
where
$$\Theta_f(g)=\sum_{\gamma\in B(\mathfrak{o}_S)\backslash P(\mathfrak{o}_S)} f(\gamma g).$$
Notice that we may identify $B(\mathfrak{o}_S)\backslash P(\mathfrak{o}_S)$ with
$B'(\mathfrak{o}_S)\backslash G'(\mathfrak{o}_S)$, where the primes denote the corresponding groups of 
rank one lower.  Moreover, we may index the cosets in
$P(\mathfrak{o}_S)\backslash G(\mathfrak{o}_S)$ by the bottom rows of coset representatives, modulo units.

The Eisenstein series is given as a sum over $\gamma\in P(\mathfrak{o}_S) \backslash G(\mathfrak{o}_S)$.
We must find coset representatives and compute their contributions.
While it is not difficult to find coset representatives, 
the difficulty in the computation is that $G(\mathfrak{o}_S)$ embeds in $\widetilde{G}_{F_S}$ 
via the map $\kappa$, whose computation is not easy. 
To exhibit coset representatives for this quotient and compute $\kappa$ on them, 
we shall make use of the calculation for type A
by Brubaker, Bump and Friedberg \cite{BBF5}, Section 5.  We show that it is possible to construct
coset representatives from products of three matrices:
two embedded $SL_r$'s and one an embedded $SL_2$ corresponding
to a short root.  We then combine the calculations of \cite{BBF5} with various Bruhat decompositions
to obtain an integral expression
for the Whittaker coefficient.  This may be evaluated, and leads to the inductive formula of the Theorem.

We begin by exhibiting the coset representatives.  (We work in $G(\mathfrak{o}_S)$ and
then move to the cover later.) To parametrize the cosets, we use the map
$G(\mathfrak{o}_S)\to \mathfrak{o}_S^\times\backslash \mathfrak{o}_S^{2r+1}$ which takes
a matrix $\gamma\in G(\mathfrak{o}_S)$ to its last row modulo multiplication of the
entire row by units.  This map factors through $P(\mathfrak{o}_S)\backslash G(\mathfrak{o}_S)$,
a coset is uniquely determined by its image, and its image is contained in the subset of 
isotropic vectors with respect to the quadratic form
corresponding to $J^{-1}$ whose entries have gcd 1, again modulo units.  

In fact, the last row map establishes a one-to-one correspondence to this subset.  Indeed,
consider the three embeddings $i_1:SL_r\to G$, $i_2:SL_2\to G$, $i_3:SL_r\to G$
given as follows.  The embedding $i_1$ is described via blocks:
$$i_1:\begin{pmatrix}g&v\\w&a \end{pmatrix} \hookrightarrow
\begin{pmatrix}a'&&&w'&\\&g&&&v\\&&1&&\\v'&&&g'&\\&w&&&a \end{pmatrix},
$$
where $g, g'\in M_{(r-1)\times(r-1)}$, $v, v', w^{t}, {w'}^{t}\in M_{(r-1)\times 1}$, and $a, a'\in M_{1\times 1}$;
the primed entries are uniquely determined so that the matrix is in $G$.   The embedding $i_2$ is the
composition of the symmetric square map $M$ above with an embedding of $SO_3$ into $SO_{2r+1}$:
$$i_2:\begin{pmatrix}a&b\\c&d\end{pmatrix}\hookrightarrow
\begin{pmatrix} a^2&&ab&&b^2\\&I_{r-1}&&&\\2ac&&ad+bc&&2bd\\&&&I_{r-1}&\\c^2&&cd&&d^2\end{pmatrix}.
$$
The embedding $i_3$ maps to the Levi subgroup of the Siegel parabolic of $G$:
$$i_3: h\hookrightarrow \begin{pmatrix}h'&&\\&1&\\&&h\end{pmatrix},$$
where once again the primed matrix $h'$ is uniquely determined by the requirement that
the image of $i_3$ be in $G$.

\begin{lemma}Let $\mathbf L = (L_1,\dots,L_{2r+1})\in\mathfrak{o}_S^{2r+1}$ be an isotropic
vector with respect to the quadratic form corresponding to $J^{-1}$,
and suppose that $\gcd(L_1,\dots,L_{2r+1})=1$.
Then there exist $g_1,g_3\in SL_r(\mathfrak{o}_S)$
and $g_2\in SL_2(\mathfrak{o}_S)$ such that $i_1(g_1)\,i_2(g_2)\,i_3(g_3)$ has bottom row $\mathbf L$.
\end{lemma}

\begin{proof}
Given a matrix $\gamma\in G(\mathfrak{o}_S)$, there is a $g_3\in SL_r(\mathfrak{o}_S)$ such
that $\gamma\, i_3(g_3)$ has bottom row $\mathbf L$ with $L_{r+2}=\dots=L_{2r}=0$.  So it is sufficient to show that
every $\mathbf L$ such that $L_{r+1}=\dots=L_{2r}=0$ is the bottom row of a matrix of the form $i_1(g_1)\,i_2(g_2)$.
But the bottom row of $i_1\left(\left(\begin{smallmatrix}g&v\\w&a\end{smallmatrix}\right)\right)
i_2\left(\left(\begin{smallmatrix}a_2&b_2\\c_2&d_2\end{smallmatrix}\right)\right)$ (with the blocks for $i_1$ as above) is
$$\begin{pmatrix}ac_2^2&w&ac_2d_2&0_{1\times(r-1)}&ad_2^2\end{pmatrix}.$$
Let us choose
$$\begin{array}{cccc}
w=(L_2,\dots,L_{r}),&c_2=\frac{L_1}{\gcd(L_1,L_{r+1})},&d_2=\frac{L_{r+1}}{\gcd(L_1,L_{r+1})},&
a=\gcd(L_1,L_{2r+1}).
\end{array}$$
Then it is easy to check that the product has the desired last row. (Here the greatest common divisors
are determined up to units which are adjusted to obtain the exact equality of bottom rows.) Let us
remark that since the vector $\mathbf L$ is isotropic, it satisfies the equation 
$L_1L_{2r+1}=L_{r+1}^2$, and this implies, for example, that $c_2$ is also given by
$c_2=\frac{L_{r+1}}{\gcd(L_{r+1},L_{2r+1})}.$
\end{proof}

Let $P_r$ denote the standard parabolic subgroup of $SL_r$ of type $(r-1,1)$.  Then we have

\begin{lemma}  \label{lm:g1g3}
A complete set of coset representatives for $P(\mathfrak{o}_S)\backslash G(\mathfrak{o}_S)$
is given by $i_1(g_1)\,i_2(g_2)\,i_3(g_3)$ where $g_1,g_3\in P_r(\mathfrak{o}_S)\backslash SL_r(\mathfrak{o}_S)$
and $g_2\in P_2(\mathfrak{o}_S)\backslash SL_2(\mathfrak{o}_S)$.
\end{lemma}

Indeed, in the proof above, modulo units we see that the bottom rows of $\gamma_1$ and $\gamma_2$
are determined by $\mathbf L$.  Also right-multiplying $\gamma\in G(\mathfrak{o}_S)$
by $i_3(P_r(\mathfrak{o}_S))$ does not change its bottom row (modulo units).  Since the bottom row modulo units
parametrizes the cosets, we have constructed a full set of coset representatives.

Let $G_{\omega_{r}}$ denote the big Bruhat cell 
$G_{\omega_{r}} =B\omega_{r} B$ of $G$.  We will focus on $\gamma$ such that $P(\mathfrak{o}_S)\gamma \in 
P(\mathfrak{o}_S) \backslash (G(\mathfrak{o}_S)\cap G_{\omega_{r}}(F_S))$
since, in the standard way, the other cells do not contribute to the coefficients
attached to $\psi_{\mathbf m}$ as it is a generic additive character of the unipotent subgroup $U$ of $B$.
For each such coset, we will work with $\gamma'=\gamma\omega_r^{-1}$ and factor this.

We now proceed to use the type A computation of \cite{BBF5}, Section 5 (especially (25) there), and the Bruhat decomposition. 
The type A computation gives, by writing representatives as products of embedded $SL_2$'s
and using the Bruhat decomposition on each factor, factorizations for 
$i_1(g_1)$ and $i_3(g_3)$.  We may similarly carry out a Bruhat decomposition of $i_2(g_2)$.  
Each of these factorizations is indexed by parameters $d_i$, one for each simple root used
(these are the $d_i$ which appear in Theorem~\ref{inductivesum}).  We will denote the parameters
for $g_1$ by $d_{r+1},\dots,d_{2r-1}$, the parameter for $g_2$ by $d_r$
(this is simply the lower right entry of $g_2$), and the parameters for
$g_3$ by $d_1,\dots,d_{r-1}$.
Suppose $\gamma'=i_1(g_1)\,i_2(g_2)\,i_3(g_3)$. Then the above ingredients give a factorization
$$
\gamma' =\FU^{+}_{1}\FD_{1}\FU^{-}_{1}\FU^{+}_{2}\FD_{2}\FU^{-}_{2}\FU^{+}_{3}\FD_{3}\FU^{-}_{3}.
$$
Here the $\FU^{+}$ are upper triangular unipotent, the $\FU^{-}$ are lower triangular unipotent, and the
$\FD$ are diagonal.  We record
the matrices.  

First, we have the diagonal matrices $\FD_i$, $i=1,2,3$, as well as their product $\FD$ which we will require later.

\begin{align*}
\FD_{1}&=\diag\left((\prod^{r-1}_{i=1}d_{r+i})^{-1},d^{-1}_{2r-1},d^{-1}_{2r-2},\dots,d^{-1}_{r+1},1,d_{r+1},\dots, d_{2r-1},\prod^{r-1}_{i=1}d_{r+i}\right);\\
\FD_{2}&=\diag\left(d^{-2}_{r},I_{2r-1},d^{2}_{r}\right);\\
\FD_{3}&=\diag\left((\prod^{r-1}_{i=1}d_{i})^{-1},d_{1},d_{2},\dots,d_{r-1},1,d^{-1}_{r-1},\dots,d^{-1}_{1},\prod^{r-1}_{i=1}d_{i}\right);\\
\FD&=\FD_{1}\FD_{2}\FD_{3}\\&=
\diag\left((d_{r}\prod^{2r-1}_{i=1}d_{i})^{-1}, \frac{d_{1}}{d_{2r-1}},\frac{d_{2}}{d_{2r-2}},\dots,\frac{d_{r-1}}{d_{r+1}},1,\frac{d_{r+1}}{d_{r-1}},\dots,\frac{d_{2r-1}}{d_{1}},d_{r}\prod^{2r-1}_{i=1}d_{i}\right).
\end{align*}

Second, we record the lower triangular unipotents:

$$\begin{array}{ccc}
&\FU^{-}_{1}=\begin{pmatrix} 1&&&&\\&V_{1}&&&\\&&1&&\\w'&&&V'_{1}&\\&w&&&1 \end{pmatrix},
&
\FU^{-}_{2}=\begin{pmatrix}1&&&&\\ &I_{r-1}&&&\\\frac{2c_{r}}{d_{r}}&&1&&\\&&&I_{r-1}&\\ \frac{c^{2}_{r}}{d^{2}_{r}}&&\frac{c_{r}}{d_{r}}&&1 \end{pmatrix}, \cr
\end{array}$$
and
$$
\FU^{-}_{3}=\begin{pmatrix}V'_{3}&&\\&1&\\&&V_{3} \end{pmatrix}
~\text{with}~
V_{3}=\begin{pmatrix}1&&&&&&\\
\frac{u_{r-2}c_{r-1}}{d_{r-1}}&1&&&&&\\
*&\frac{u_{r-3}c_{r-2}}{d_{r-2}}&1&&&&\\
\vdots&\vdots&\vdots&\ddots&&&\\
*&*&*&\dots&1&&\\
*&*&*&\dots&\frac{u_{1}c_{2}}{d_{2}}&1&\\
*&*&*&\dots&*&\frac{c_{1}}{d_{1}}&1 \end{pmatrix}.
$$
In these matrices, the $c_i$ and $d_i$ are coprime and $c_iu_i\equiv 1\bmod d_i$.  The formula for $V_1$ is similar
to that for $V_3$ but with all indices shifted up by $r$.  The matrices $V'_i$ and $w'$ are determined so that the
unipotents are in $G$.

Third, we have the following upper triangular unipotents:

$$
\FU^{+}_{2}=\begin{pmatrix}1&&-\frac{u_{r}}{d_{r}}&&\frac{u^{2}_{r}}{d^{2}_{r}}\\&I_{r-1}&&&\\&&1&&-\frac{2u_{r}}{d_{r}}\\&&&I_{r-1}&\\&&&&1 \end{pmatrix}
$$
$$
\FU^{+}_{3}=\begin{pmatrix}U'_{3}&&\\&1&\\&&U_{3} \end{pmatrix}
~~\text{where}~~
U_{3}=\begin{pmatrix} 1&&&&-\frac{u_{r-1}}{d_{r-1}}\\&1&&&-\frac{u_{r-2}}{d_{r-1}d_{r-2}}\\&&\ddots&&\vdots\\&&&1&-\frac{u_{1}}{\prod^{r-1}_{i=1}d_{i}}\\&&&&1\end{pmatrix}.
$$
We will not need $\FU^+_1$ (it is in the unipotent radical of $P$).

After multiplying and rearranging, one finds that
$
\gamma'=\FU^{+}\FD\FU^{-},
$
where $\FU^{+}$ is in the unipotent radical of $P$ and
$$\psi_{\mathbf m}(\omega_{r}^{-1} \FU^- \omega_{r})=
\psi\left(m_{1}\frac{c_{1}}{d_{1}}+
\sum_{j=1}^{r-1} m_{j+1}\left(\frac{u_{j}c_{j+1}}{d_{j+1}}+(-1)^{\varepsilon_{j+1}}\frac{d_{j}c_{2r-j}u_{2r-1-j}}{d_{j+1}d_{2r-j}}\right)\right).
$$

We now proceed to compute the Whittaker coefficients (compare with \cite{BBF5}, Section 5).
The above factorization lifts from $G$ to $\widetilde{G}_{F_S}$ but at the expense of some power residue
and Hilbert symbols.
Indeed, computations similar to \cite{BBF5}, pg.\ 1098 and using Lemma~\ref{cover} show that in $\widetilde{G}_{F_S}$
\begin{align*}
&(\gamma_{1},\kappa(\gamma_1)))=\prod^{2r-1}_{k=r+1} \ppair{\frac{c_{k}}{d_{k}}}_{n}\,\,
\prod^{2r-1}_{i=r+2}\prod^{i}_{j=r+1}(d_{i},d_{j})^{-1}_{S}\bis(\FU^{+}_{1})\bis(\FD_{1})\bis(\FU^{-}_{1});\\
&(\gamma_2,\kappa(\gamma_2))=\left(\frac{d_r}{c_r}\right)^2_n(d_{r},c_{r})_S^{2}\,\bis(\FU^{+}_{2})\bis(\FD_{2})\bis(\FU^{-}_{2});\\
&(\gamma_{3},\kappa(\gamma_3))=\prod^{r-1}_{k=1}\left(\frac{d_k}{c_k}\right)_n(d_{k},c_{k})_{S}\,\bis(\FU^{+}_{3})\bis(\FD_{3})\bis(\FU^{-}_{3}).
\end{align*}
Indeed, the factor $\left(\tfrac{a}{b}\right)_{2n}$ of Lemma~\ref{cover} is multiplied by the Hilbert symbol $(b,d)_S$ in carrying
out this decomposition.  But we have 
$$\left(\frac{a}{b}\right)_{2n}(b,d)_S=\left(\frac{d}{b}\right)_{2n}^{-1}(b,d)_S=\left(\frac{b}{d}\right)_{2n}^{-1}
=\left(\frac{c}{d}\right)_{2n},$$
where these equalities are obtained by applying basic properties of the power-residue symbol along with the reciprocity law
(\ref{reciprocity-law}).  Thus we obtain two factors of $\left(\frac{c_k}{d_k}\right)_{2n}$ for each $k$ which give
the factors of $\left(\frac{c_k}{d_k}\right)_{n}$.
Also a computation using (\ref{torus-cocycle}) yields
$$
\prod^{2r-1}_{i=r+2}\prod^{i}_{j=r+1}(d_{i},d_{j})^{-1}_{S}\prod^{3}_{i=1}\bis(\FU^{+}_{i})\bis(\FD_{i})\bis(\FU^{-}_{i})=\prod^{r-1}_{i=1}\prod^{2r-1}_{j=2r-i}(d_{2r-i}/d_{i},d_{j})_{S}\, \bis(\FU^{+})\bis(\FD)\bis(\FU^{-}).
$$

Let $U_P$ denote the unipotent radical of the parabolic subgroup $P$.  This is the group of unipotent matrices with all non-diagonal entries outside of the first row and column equal to zero.  Then we have a decomposition $U=U_PU_P'$
where $U_P'$ is the complementary subgroup. Now it is easy to check that $U_P(\mathfrak{o}_S)$
acts properly on the right on the set of cosets 
in $P(\mathfrak{o}_S)\backslash (G(\mathfrak{o}_S)\cap G_{\omega_{r}}(F_S))$.
We use this to
replace the integration in (\ref{Whit-coeff}) with a double integration over $U_P(F_S)$ and $U_P'(\mathfrak{o}_S)
\backslash U_P'(F_S)$:
\begin{multline*}
\sum_{\gamma \in 
P(\mathfrak{o}_S) \backslash (G(\mathfrak{o}_S)\cap G_{\omega_{r}}(F_S))}\int_{U(\mathfrak{o}_S)\backslash U(F_S)}
=\\
\sum_{\gamma \in 
P(\mathfrak{o}_S) \backslash (G(\mathfrak{o}_S)\cap G_{\omega_{r}}(F_S))/ U_P(\mathfrak{o}_S)}
\int_{U_P(F_S)}\int_{U_P'(\mathfrak{o}_S)\backslash U_P'(F_S)}.
\end{multline*}

To carry this out with the cosets parametrized above, note that if $\gamma=\gamma'\omega_r$ where $\gamma'$
has bottom row $\mathbf{L}$, then modding out by $U_P(\mathfrak{o}_S)$ on the right corresponds to
taking all $L_k$ modulo $L_{2r+1}$ (which is nonzero since $\gamma$ is in the big cell).  With the
parametrization above, $L_{2r+1}=d_r\prod_{j=1}^{2r-1}d_j$, and this corresponds to taking
$c_k$ modulo $\tilde{d}_k$ and prime to $d_k$, where 
$\tilde{d}_k=\prod^{k}_{j=1}d_{j}$ for $k\leq r$ and $\tilde{d}_k=d_{r}d^{-1}_{2r-k}\prod^{k}_{j=1}d_{j}$
for $r<k\leq 2r-1$.  

Putting all this together, we have
\begin{align}
W_{\mathbf m}&(f,\mathbf s) =\sum_{d_{k}} 
\prod^{r-1}_{i=1}\prod^{2r-1}_{j=2r-i}(d_{2r-i}d_{i}^{-1},d_{j})_{S}\nonumber \\
&\sum_{\substack{c_{k}\bmod \tilde{d}_{k}\\ \gcd(c_k,d_k)=1}}\ppair{\frac{c_{k}}{d_{k}}}_n^{\varepsilon_k} 
\psi\left(m_{1}\frac{c_{1}}{d_{1}}+
\sum_{j=1}^{r-1} m_{j+1}\left(\frac{u_{j}c_{j+1}}{d_{j+1}}+(-1)^{\varepsilon_{j+1}}\frac{d_{j}c_{2r-j}u_{2r-1-j}}{d_{j+1}d_{2r-j}}\right)\right)
\nonumber \\
&\int_{U_P(F_S)}\int_{U_P'(\mathfrak{o}_S)\backslash U_P'(F_S)}
\Theta\left(\bis(\FD)\,\bis(\omega_{r})\,\bis(u_P')\,\bis(u_P)\right)\,\psi\left(\sum^{r}_{i=1} m_{i}u_{i,i+1}\right)\ud u_P'\ud u_P. \label{eq:W1}
\end{align}
Here the $u_k$ satisfy $c_ku_k\equiv 1 \bmod {d}_k$.  

Now the function $\Theta$ is invariant under the group of lower-triangular unipotent 
matrices in $P(\mathfrak{o}_S)$.  Putting such matrices
into the integral in \eqref{eq:W1}, moving them rightwards and changing variables, one sees that the integral vanishes unless the divisibility
condition \eqref{eq:div-d-2} holds.  Moreover, replacing  $u_j$ by $u_{j}+t_jd_{j}$, $1\leq j\leq r-1$, 
and summing over $t_j$ modulo $d_j$, it follows that the expression is zero unless
the divisibility condition \eqref{eq:div-d} holds as well.

Since $\kappa(\gamma)$ depends only on $c_{k}$ modulo $d_{k}$ for $1\leq k\leq 2r-1$, if we sum over $c_{k}\bmod d_{k}$ and  then multiply the result by $|\prod^{k-1}_{j=1}d_{j}|$ when $k\leq r$ and by $|d_{r}d^{-1}_{2r-k}\prod^{k-1}_{j=1}d_{j}|$ when $r<k$, this has the same result as summing over $L_{i}\mod L_{2r-1}$. In doing so
we obtain the factor $|d_{r}|^{2r-2}\prod^{r-1}_{k=1}|d_{k}|^{2r-2-k}\, |d_{r+k}|^{r-k-1}.$

Next we make use of the factorization 
$$\omega_r=\begin{pmatrix}1&&\\&\omega_{r-1}&\\&&1
\end{pmatrix}\begin{pmatrix} &&1\\&-I_{2r-1}&\\1&& \end{pmatrix}.$$
To do so, for $g'\in \widetilde{\SO}_{2r-1}(F_{S})$, denote
$$
f'_{r-1}(g')=\int_{U_P(F_S)}f\left(i(g'))\FD' \bis\begin{pmatrix} &&1\\&-I_{2r-1}&\\1&& \end{pmatrix} \bis(u_P)\right)\psi(m_{1}{(u_P)}_{1,2})\ud u_P,
$$
where $i\colon \widetilde{\SO}_{2r-1}\to \widetilde{\SO}_{2r+1}$ is the embedding in the Levi subgroup of $P$ and
\begin{align*}
\FD'&=\begin{pmatrix}1&&\\&\omega_{r-1}&\\&&1\end{pmatrix}^{-1}\FD\begin{pmatrix}1&&\\&\omega_{r-1}&\\&&1\end{pmatrix}
\\
&=\diag\left((d_{r}\!\prod^{2r-1}_{i=1} d_{i})^{-1}, \frac{d_{2r-1}}{d_{1}},\frac{d_{2r-2}}{d_{2}},\dots, \frac{d_{r+1}}{d_{r-1}},1,\frac{d_{r-1}}{d_{r+1}},\dots, \frac{d_{2}}{d_{2r-2}},\frac{d_{1}}{d_{2r-1}}, d_{r}\!\prod^{2r-1}_{i=1} d_{i}\right).
\end{align*}
(The function $f'_{r-1}$ of course depends on $\FD'$ but we suppress this dependence for notational convenience.)
The function $f'_{r-1}$ is in $I(\mathbf{s}')$ where $\mathbf{s}'=(s_2,\dots,s_r)$.

Then a straightforward calculation shows that 
\begin{multline}\label{mess1}
\int_{U_P(F_S)}\!\!\Theta\left(\bis(\FD) \bis(\omega_{r})\bis(u'_P)\bis(u_P)\right)\,
\psi(m_{1}{(u_P)}_{1,2})\,du_P \\
=\sigma(i(\omega_{r-1}),\FD')^{-1} \!E_{f'_{r-1}}\left({\bis(\omega_{r-1})\bis(u')},\mathbf{s}'\right)
\end{multline}
where $u'$ is the unipotent element in ${SO}_{2r-1}$ such that $i(u')=\FD'u\FD'^{-1}$.
We may drop the $\bis(\omega_{r-1})$ on the right hand side of (\ref{mess1}) since $E_{f'_{r-1}}$ is automorphic.

The double integral in (\ref{eq:W1}) may then be simplified by changing $u_P'$ to $\FD'u_P'\FD'^{-1}$.
This variable change has no effect on the measure of compact quotient
$U_P'(\mathfrak{o}_S)\backslash U_P'(F_S)$. However, it changes the locations that support the character (corresponding to the simple roots), sending
\begin{multline*}
\cpair{u_{2,3},u_{3,4},\dots,u_{r-1,r},u_{r,r+1}}
\mapsto\\ \cpair{\frac{d_{2r-1}d_{2}}{d_{1}d_{2r-2}}u_{2,3},\frac{d_{2r-2}d_{3}}{d_{2}d_{2r-3}}u_{3,4},\dots,\frac{d_{r+2}d_{r-1}}{d_{r-2}d_{r+1}}u_{r-1,r},\frac{d_{r+1}}{d_{r-1}}u_{r,r+1}}.
\end{multline*}

These steps give the following evaluation of the double integral in \eqref{eq:W1}:
\begin{align}
&\sigma(i(\omega_{r-1}),\FD')^{-1} \int_{U'(\mathfrak{o}_S)\backslash U'(F_S)}E_{f'_{r-1}}(\bis(u),\mathbf{s}')
\nonumber\\
&\psi\ppair{\frac{m_{2}d_{1}d_{2r-2}}{d_{2r-1}d_{2}}x_{2,3}+\frac{m_{3}d_{2}d_{2r-3}}{d_{2r-2}d_{3}}x_{3,4}+\cdots+\frac{m_{r-1}d_{r-2}d_{r+1}}{d_{r+2}d_{r-1}}x_{r-1,r}+\frac{m_{r}d_{r-1}}{d_{r+1}}x_{r,r+1}}\,du.\label{eq:cocycle-d}
\end{align}
Here $U'$ is the unipotent radical of $SO_{2r-1}$; for convenience we write $u\in U'$ as $u=(x_{i,j})$ with $2\leq i,j\leq 2r$.

Now we may use the induction hypothesis and write the integral in (\ref{eq:cocycle-d}) as
\begin{align*}
&\sum_{0\neq D_{i}\in \Fo_{S}/\Fo^{\times}_{S}}H(D_{2},\dots,D_{r};\frac{m_{2}d_{1}d_{2r-2}}{d_{2r-1}d_{2}},\dots,\frac{m_{r}d_{r-1}}{d_{r+1}})\prod^{r}_{i=2}|D_{i}|^{-2s_{i}} \Lam_{{\bf m}',D}(f'_{r-1}),
\end{align*}
where 
$$D=\diag(D^{-1}_{2},D^{-1}_{3}D_{2},\dots,D^{-1}_{r-1}D_{r-2},D^{-2}_{r}D_{r-1},1,D^{2}_{r}D^{-1}_{r-1},\dots,D_{2})
$$ 
and ${\bf m}'=(\frac{m_{2}d_{1}d_{2r-2}}{d_{2r-1}d_{2}},\dots,\frac{m_{r}d_{r-1}}{d_{r+1}})$. 

Let $C_{i}=\Fd_{i}D_{i}$, where the $\Fd_{i}$ are defined in \eqref{eq:C} and $D_{1}=1$.
Then substituting in the definitions and comparing, one sees that 
\begin{align}
\FJ(D)\Lam_{{\bf m'},D}(f'_{r-1})&
=\sigma(i(\omega_{r-1}),\FD')\prod^{r}_{i=2}(\Fd_{i-1}/\Fd_{i},D_{i})^{-\varepsilon_i}_{S}
\nonumber\\
&\left(\left\vert \frac{d_{1}}{d_{2r-1}}\right\vert^{2r-3}\left\vert\frac{d_{2}}{d_{2r-2}}\right\vert^{2r-5}\cdots
\left\vert\frac{d_{r-2}}{d_{r+2}}\right\vert^{3}\left\vert\frac{d_{r-1}}{d_{r+1}}\right\vert\right)^{-1}
\FJ(C)\Lam_{{\bf m},C}(f).\label{eq:cocycle-D}
\end{align}
Combining Equations (\ref{eq:W1}), (\ref{mess1}), (\ref{eq:cocycle-d}) and (\ref{eq:cocycle-D}), Theorem~\ref{inductivesum} follows.
\qed

\smallskip

The functionals $\Lambda_{\mathbf m, \mathbf C}(f)$ may be analyzed as in \cite{BBF5},
Section 6.  Since the arguments there apply with only minimal changes, we do not carry this out here.

\section{Twisted Multiplicativity}\label{sect5}

In this Section we establish the twisted multiplicativity properties of the coefficients $H$.  This allows us to recover
all coefficients from those indexed by parameters which are all powers of a single prime.
See also \cite{BBF1}.  Recall $\varepsilon_i=\|\alpha_i\|^2$.

\begin{theorem}\label{thm:twist-m}
If vectors ${\bf m}$, ${\bf m'}$ and ${\bf C}\in (\Fo_{S}-\{0\})^{r}$ satisfy $\gcd(m'_{1}\cdots m'_{r}, C_{1}\cdots C_{r})=1$, then
\begin{equation}\label{eq:twist-m}
H({\bf C}; m_{1}m'_{1},\dots,m_{r}m'_{r})=\prod^{r}_{i=1}\ppair{\frac{m'_{i}}{C_{i}}}_n^{-\varepsilon_i}H({\bf C}; {\bf m}).
\end{equation}
\end{theorem}
\begin{proof}
The proof is by induction on $r$. For $r=1$, $H(C_{1};m_{1}m'_{1})=g_{2}(m_{1}m'_{1},C_{1})$, where $g_2$ is the Gauss
sum defined in Sect.~\ref{sect6} below.  Equation~\eqref{eq:twist-m} follows by the usual properties of Gauss sums.

For general $r$, since all $d_i$ divide $C_1$, we have $\gcd(d_i,m'_j)=1$ for all $i,j$. Thus, the divisibility conditions~\eqref{eq:div-d}
for $(m_1m_1',\dots,m_r m'_r)$, 
namely $d_{j+1}\vert m_{j+1}m'_{j+1} d_{j}$ for $1\leq j\leq r-1$, hold if and only if
$d_{j+1}\vert m_{j+1} d_{j}~\text{for}~1\leq j\leq r-1$.  Similarly, the divisibility conditions~\eqref{eq:div-d-2} for $(m_1m_1',\dots,m_r m'_r)$
hold if and only if $d_{j+1}d_{2r-j}\vert m_{j+1} d_{j}d_{2r-j-1}$ for $1\leq j\leq r-2$ and $d_{r+1}\vert m_{r} d_{r-1}$.
In the inner sum in (\ref{thesum}), make the variable change $c_{i}\rightarrow  (\prod_{j=1}^{i}m'_{i})^{-1}c_{i}$ if $1\leq i\leq r$, and $c_{i}\rightarrow  (m'_{r}\prod_{j=1}^{i-1}m'_{i})^{-1}c_{i}$ for $r< i\leq 2r-1$.
This variable change removes all $m'_{i}$'s from the sum and contributes the factor 
\begin{equation}\label{eq:factor1}
\prod^{r-1}_{i=1}\ppair{\frac{\prod^{i}_{j=1}m'_{j}}{d_{i}}}_n^{-1}
\ppair{\frac{\prod^{r}_{i=1}m'_{i}}{d_{r}}}_n^{-2}
\prod^{r}_{i=2}\ppair{\frac{\prod^{r}_{\ell=1}m'_{\ell}\prod^{r}_{j=i}m'_{j}}{d_{2r+1-i}}}_n^{-1}.
\end{equation}
In addition, by induction (note $\gcd(m'_{2}\cdots m'_{r},D_{2}\cdots D_{r})=1$), the $H$ on the right of (\ref{thesum}) contributes a factor of 
\begin{equation}\label{eq:factor2}
\prod^{r}_{i=2}\ppair{\frac{m'_{i}}{D_{i}}}_n^{-\varepsilon_i}.
\end{equation}
Multiplying \eqref{eq:factor1} and \eqref{eq:factor2} and making use of \eqref{eq:C}, one obtains Theorem~\ref{thm:twist-m}.
\end{proof}

Let $\apair{\cdot,\cdot}$ be the standard Euclidean inner product, normalized so that the short roots have length $1$.

\begin{theorem}\label{twmu2}
If vectors ${\bf C}$ and ${\bf C'}$ in $(\Fo_{S}-\{0\})^{r}$ satisfy $\gcd(C_{1}\cdots C_{r}, C'_{1}\dots C'_{r})=1$, then
\begin{equation}\label{eq:twist-c}
H(C_{1}C'_{1},\dots,C_{r}C'_{r};{\bf m})=\mu({\bf C},{\bf C'})H({\bf C}; {\bf m}) H({\bf C'}; {\bf m}),
\end{equation}
where $\mu({\bf C},{\bf C'})$ is an $n$-th root of unity depending ${\bf C}$, ${\bf C'}$, given by
$$
\mu({\bf C},{\bf C'})=\prod^{r}_{i=1}\ppair{\frac{C_{i}}{C'_{i}}}_n^{\varepsilon_i}\ppair{\frac{C'_{i}}{C_{i}}}_n^{\varepsilon_i}\,\prod_{j<i}\ppair{\frac{C_{i}}{C'_{j}}}_n^{2\apair{\alpha_{i},\alpha_{j}}}
\ppair{\frac{C'_{i}}{C_{j}}}_n^{2\apair{\alpha_{i},\alpha_{j}}}.
$$
\end{theorem}
\begin{proof}
Since $\gcd(C_{1}\cdots C_{r}, C'_{1}\dots C'_{r})=1$, there is a one-to-one correspondence between the $d_{i}$ satisfying 
\begin{equation*}
\begin{array}{ll}
d_{r}\prod^{2r-1}_{i=1}d_{i}=C'_{1}C_{1};&\\
(d_{r}\prod^{2r-1}_{j=i}d_{j}\prod^{2r-1}_{j=2r+1-i}d_{j})\vert C_{i}C_i'&\text{for $1< i<r$};\\
(\prod^{2r-1}_{j=r}d_{j})\vert C_{r}C_r';&
\end{array}
\end{equation*}
and the divisibility conditions~\eqref{eq:div-d},  \eqref{eq:div-d-2}, and the integer pairs $e_{i}$, $e'_{i}$ such that
$d_{i}=e_{i}e'_{i}$ for all $i$, satisfying
\begin{equation}
\begin{array}{l}
e_{r}\prod^{2r-1}_{i=1}e_{i}=C_{1}\text{ and }
e'_{r}\prod^{2r-1}_{i=1}e'_{i}=C'_{1};\\
(e_{r}\prod^{2r-1}_{j=i}e_{j}\prod^{2r-1}_{j=2r+1-i}e_{j})\vert C_{i} \text{ for } 1<i<r \text{ and }
(\prod^{2r-1}_{j=r}e_{j})\vert C_{r};\\
(e'_{r}\prod^{2r-1}_{j=i}e'_{j}\prod^{2r-1}_{j=2r+1-i}e'_{j})\vert C'_{i} \text{ for } 1<i<r\text{ and }
(\prod^{2r-1}_{j=r}e'_{j})\vert C'_{r};\\
e_{j+1}\vert m_{j+1} e_{j}~\text{for}~1\leq j\leq r-1;\\
e'_{j+1}\vert m_{j+1} e'_{j}~\text{for}~1\leq j\leq r-1;\\
e_{j+1}e_{2r-j}\vert m_{j+1} e_{j}e_{2r-j-1} \text{ for } 1\leq j\leq r-2 \text{ and } e_{r+1}\vert m_{r}e_{r-1};\\
e'_{j+1}e'_{2r-j}\vert m_{j+1} e'_{j}e'_{2r-j-1} \text{ for } 1\leq j\leq r-2 \text{ and }e'_{r+1}\vert m_{r}e'_{r-1}.\\
\end{array}
\end{equation} 
Thus, we can split the sum over the $d_{i}$ into sums over $e_{i}$ and $e'_{i}$.

Let
\begin{equation}
c_i=\begin{cases}
x'_{i}\prod^{i}_{j=1}e_{j}+x_{i}\prod^{i}_{j=1} e'_{j}  &\text{ for } 1\leq i\leq r,\\
x'_{i}e_{r}e^{-1}_{2r-i}\prod^{i}_{j=1}e_{j}+x_{i}e'_{r}e'^{-1}_{2r-i}\prod^{i}_{j=1}e'_{j} &\text{ for } r+1\leq i\leq 2r-1.
\end{cases}
 \end{equation}
Since $\prod^{i-1}_{j=1}e_{j}$ is a unit modulo $e'_{i}$ and $\prod^{i-1}_{j=1}e'_{j}$ is a unit modulo $e_{i}$, as $x'_{i}$ varies modulo $e'_{i}$ and $x_{i}$ varies modulo $e_{i}$, $c_{i}$ for $1\leq i\leq r$ varies modulo $d_{i}$. This is also true for $r<i\leq 2r-1$. With this parametrization, the $c_{i}$ that are invertible modulo $e_{i}$ are those with $\gcd(x_{i},e_{i})=\gcd(x'_{i},e'_{i})=1$, and for such $c_{i}$, $u_{i}$ is determined by the equations 
$$
\begin{array}{l}
x_{i}u_{i}\prod^{i}_{j=1}e'_{j}\equiv 1\bmod e_{i}
\text{ and } x'_{i}u_{i}\prod^{i}_{j=1}e_{j}\equiv 1\bmod e'_{i} \text{ for } 1\leq i\leq r;\\
x_{i}u_{i}e'_{r}e'^{-1}_{2r-i}\prod^{i}_{j=1}e'_{j}\equiv 1\bmod e_{i}\text{ and }
x'_{i}u_{i}e_{r}e^{-1}_{2r-i}\prod^{i}_{j=1}e_{j}\equiv 1\bmod e'_{i}\text{ for } r< i< 2r.
\end{array}
$$
Let $v_{i}$ modulo $e_{i}$ (resp. $v'_{i}$ modulo $e'_{i}$) satisfy $v_{i}x_{i}\equiv 1\bmod e_{i}$ (resp. $v'_{i}x'_{i}\equiv 1\bmod e'_{i}$).
Let $u_0=v_0=v_0'=1$.  Substituting in and simplifying, one sees that
\begin{equation*}
\begin{array}{lll}
\psi\left(\frac{m_{i}u_{i-1}c_{i}}{d_{i}}\right)&=\psi\left(\frac{m_{i}v'_{i-1}x'_{i}}{e'_{i}}\right)
\psi\left(\frac{m_{i}v_{i-1}x_{i}}{e_{i}}\right)&\text{for $1\leq i\leq r$}\\
\psi\ppair{\frac{m_{i}d_{i-1}c_{2r+1-i}u_{2r-i}}{d_{i}d_{2r+1-i}}}&=
\psi\ppair{\frac{m_{i}e'_{i-1}x'_{2r+1-i}v'_{2r-i}}{e'_{i}e'_{2r+1-i}}}\psi\ppair{\frac{m_{i}e_{i-1}x_{2r+1-i}v_{2r-i}}{e_{i}e_{2r+1-i}}}
&\text{for $2\leq i\leq r$.}
\end{array}
\end{equation*}
Therefore, the exponential sum in Theorem~\ref{inductivesum} factors into two sums with similar divisibility conditions and similar exponentials. 

To carry out the proof by induction, we must analyze the power residues symbols.
Below, we will work with pairs of numbers of the form $A$, $A'$ and $B$, $B'$ such that $\gcd(A,A')=\gcd(B,B')=1$.
For convenience, let $h(A,B)=\ppair{\frac{A}{B}}_n\ppair{\frac{A'}{B'}}_n$. Then
$h(A_{1}A_{2},B)=h(A_{1},B)h(A_{2},B),$ $h(A^{-1},B)=h(A,B)^{-1}$ and, by reciprocity, 
$(A,B')_{S}(A',B)_{S}h(A,B)=h(B,A)$.  
With this notation, we have 
\begin{equation}
\ppair{\frac{c_{i}}{d_{i}}}_n=\ppair{\frac{x'_{i}}{e'_{i}}}_n\ppair{\frac{x_{i}}{e_{i}}}_n\cdot
\begin{cases}
h(e_{1}\cdots e_{i},e_{i})& \text{ if }1\leq i\leq r,\\
h(e_{r}e^{-1}_{2r-i}\prod^{i}_{j=1}e_{j},e_{i})& \text{ if } r<i\leq 2r-1.
\end{cases}\label{eq:twist-c/d}
\end{equation}

The factorization $d_i=e_ie_i'$ leads to factorizations of the $D_i$ in Equation~\eqref{eq:C} for $2\leq i\leq r$.
Rather than introducing a new letter, we denote the factors $D_i$, $D_i'$; we have
$\gcd(D_{2}\cdots D_{r},D'_{2}\cdots D'_{r})=1$. 
Then by induction we have
\begin{align}
&H(D_{2}D'_{2},\dots,D_{r}D'_{r};\frac{m_{2}d_{1}d_{2r-2}}{d_{2}d_{2r-1}},\dots,\frac{m_{r-1}d_{r-2}d_{r+1}}{d_{r-1}d_{r+2}},\frac{m_{r}d_{r-1}}{d_{r+1}})
\nonumber\\
=&\prod^{r}_{i=2}h(D_{i},D_{i})^{\varepsilon_i}\prod_{i<j}h(D_{j},D_{i})^{2\apair{\alpha_{i},\alpha_{j}}}H(D_{2},\dots,D_{r},\frac{m_{2}d_{1}d_{2r-2}}{d_{2}d_{2r-1}},\dots,\frac{m_{r}d_{r-1}}{d_{r+1}})
\nonumber\\
&\times H(D'_{2},\dots,D'_{r},\frac{m_{2}d_{1}d_{2r-2}}{d_{2}d_{2r-1}},\dots,\frac{m_{r}d_{r-1}}{d_{r+1}}). \label{eq:twist-D}
\end{align}

Let $m_{i}=n_{i}n'_{i}$ such that $n_{i}e_{i-1}e_{2r-i}/e_{i}e_{2r-i+1}$ and 
$n'_{i}e'_{i-1}e'_{2r-i}/e'_{i}e'_{2r-i+1}$ are integral and $\gcd(n_{i},C'_{1})=\gcd(n'_{i},C_{1})=1$ for all $i$. 
Then $\gcd(n_{i},D'_{j})=\gcd(n'_{i},d_{j})=1$ for all $i$, $j$. By Theorem~\ref{thm:twist-m},
\begin{align*}
&H(D_{2},\dots,D_{r},\frac{m_{2}d_{1}d_{2r-2}}{d_{2}d_{2r-1}},\dots,\frac{m_{r}d_{r-1}}{d_{r+1}})\\
=&\prod^{r}_{i=2}\ppair{\frac{n'_{i}e'_{i-1}e'_{2r-i}/(e'_{i}e'_{2r-i+1})}{D_{i}}}_n^{-\varepsilon_i}
H(D_{2},\dots,D_{r},\frac{n_{2}e_{1}e_{2r-2}}{e_{2}e_{2r-1}},\dots,\frac{n_{r}e_{r-1}}{e_{r+1}}).
\end{align*}
Note that $\frac{n'_{i}e'_{i-1}e'_{2r-i}}{e'_{i}e_{2r-i+1}}=n'_{r}e'_{r-1}/e_{r+1}$ when $i=r$.
Moreover, for $2\leq i\leq r$,
$$
\ppair{\frac{n'_{i}e'_{i-1}e'_{2r-i}/e'_{i}e'_{2r-i+1}}{D_{i}}}_n^{-1}=\ppair{\frac{n'_{i}}{D_{i}}}_n^{-1}
\ppair{\frac{e'_{i-1}/e'_{2r-i+1}}{D_{i}}}_n^{-1}
\ppair{\frac{e'_{i}/e'_{2r-i}}{D_{i}}}_n.
$$
Since $\gcd(n'_{i},D_{i})=1$ for all $i$, we may then put the $n_{i}$ back into the coefficients $H$
by Theorem~\ref{thm:twist-m}.  Doing so, and using a similar argument with the last factor in~\eqref{eq:twist-D}, we obtain
\begin{multline}
H(D_{2}D'_{2},\dots,D_{r}D'_{r};\frac{m_{2}d_{1}d_{2r-2}}{d_{2}d_{2r-1}},\dots,\frac{m_{r-1}d_{r-2}d_{r+1}}{d_{r-1}d_{r+2}},\frac{m_{r}d_{r-1}}{d_{r+1}})\\
=\ppair{\prod^{r}_{i=2}h(D_{i},D_{i})
h(e_{i-1}e^{-1}_{2r-i+1},D_{i})^{-1}h(e_{i}e^{-1}_{2r-i},D_{i})}^{\varepsilon_i}
\prod^{r-1}_{i=1}h(D_{i+1},D_{i})^{-\varepsilon_{i+1}}\\
H(D_{2},\dots,D_{r},\frac{m_{2}e_{1}e_{2r-2}}{e_{2}e_{2r-1}},\dots,\frac{m_{r}e_{r-1}}{e_{r+1}})
 H(D'_{2},\dots,D'_{r},\frac{m_{2}e'_{1}e'_{2r-2}}{e'_{2}e'_{2r-1}},\dots,\frac{m_{r}e'_{r-1}}{e'_{r+1}}). \label{eq:twist-DD'}
\end{multline}

By contrast,
$$
\mu({\bf C},{\bf C'})=\prod^{r}_{i=1}h(C_{i},C_{i})^{\varepsilon_i}\prod^{r-1}_{i=1}h(C_{i+1},C_{i})^{-\varepsilon_{i+1}}.
$$
Substituting in the expressions from (\ref{eq:C}) and using the properties of the function $h(A,B)$ noted above, one finds that
\begin{align}
\mu({\bf C},{\bf C'})&=
\prod^{r}_{i=2}h(D_{i},D_{i})^{\varepsilon_i}h(e_{i}e^{-1}_{2r-i},D_{i})\prod^{r-1}_{i=1}h(D_{i+1},D_{i})^{-\varepsilon_{i+1}}
\nonumber\\
&\prod^{r}_{i=2}(\Fd_{i-1}\Fd^{-1}_{i},D'_{i})_S^{-\varepsilon_i}(\Fd'_{i-1}{\Fd'_{i}}^{-1},D_{i})_S^{-\varepsilon_i}
h(\Fd_{i-1}\Fd^{-\varepsilon_i}_{i},D_i^{\varepsilon_i})^{-1}
\nonumber\\
&\prod^{r}_{i=1}\prod^{i}_{j=1}h(e_{j},e_{i})^{\varepsilon_i}\prod^{2r-1}_{i=r+1}\prod^{i}_{j=1}h(e_{r}e^{-1}_{2r-i}e_{j},e_{i})
\nonumber\\
&\prod^{r-1}_{i=1}\prod^{2r-1}_{j=2r-i}(e_{2r-i}e^{-1}_{i},e'_{j})_{S}\prod^{r-1}_{i=1}\prod^{2r-1}_{j=2r-i}(e'_{2r-i}{e'_{i}}^{-1},e_{j})_{S}. \label{eq:twist-d}
\end{align}

Finally, we collect the residue symbols that arise in~\eqref{eq:twist-c/d} and~\eqref{eq:twist-DD'}, and Hilbert symbols that arise in~\eqref{eq:cocycle-D} and~\eqref{eq:cocycle-d}. These match~\eqref{eq:twist-d}, and the theorem is proved.
\end{proof}

The proof above is similar to the proof for type A in \cite{BBF5}, Theorem 3.  However, the formulae for 
twisted multiplicativity in this Section involve the root system of the Langlands dual group 
$^LG$, a phenomenon predicted in \cite{BBF1}.  
Theorem~\ref{twmu2} agrees with  \cite{BeBrF1} after noting that in 
formula (20) of that work the order of
the simple roots $\alpha_i$ is the reverse of ours.

\section{Combinatorial Descriptions of the Coefficients at Powers of $p$}\label{sect6}

By twisted multiplicativity, the coefficients $H({\bf C};{\bf m})$ for all $\bf C$ and $\bf m$ (with $m_i\neq0$ for all $i$)
are determined from the coefficients in which $C_i$ and $m_i$ are powers of a prime.  Let $p$ be a fixed prime.  
In the rest of this paper, we will study
these $p$-power coefficients, and we shall show that they may be evaluated using crystal graphs
of type C.  Such a graph is attached
to a given dominant weight $\eta$ for the Langlands dual group $^LG$
and describes a representation of the quantized universal enveloping algebra with
$\eta$ as highest weight.  
See for example \cite{BBF4}, Chapter 2, for a brief summary and references. 
For the $\bf m$-th Whittaker coefficient, $\eta$ will be related to the $\ord_p(m_i)$.

Littelmann~\cite{Litt}
described the set of path lengths attached to such a crystal graph as one traverses
from each vertex to the lowest weight vector,
following a prescribed order (coming from a certain factorization of the long
element of the Weyl group into simple reflections).   Assembling these path lengths into a vector,
one obtains the lattice points of a certain polytope.  
Each lattice point is called a Berenstein-Zelevinsky-Littelmann pattern (or BZL-pattern for short).

A type C BZL-pattern may be rewritten as a triangular array of non-negative integers $c_{i,j}$, with $1\leq i\leq r$
and $i\leq j\leq 2r-i$, satisfying certain inequalities.  See \cite{Litt}, Theorem 6.1 and Corollary 6.1 or
 \cite{BeBrF1}, Section 2.2.
Following Littelmann, write $\overline{c}_{i,j}=c_{i,2r-j}$ for $i\leq j\leq r$.
In particular, the top row of this array is given by
$$\left(\begin{matrix}
c_{1,1}&c_{1,2}&\dots&c_{1,r}&\overline{c}_{1,r-1}&\dots&\overline{c}_{1,1}
\end{matrix}\right).$$

In this Section we present two combinatorial descriptions related to crystal graphs.  First, in Subsection~\ref{sect61}
 we shall rewrite the inductive formula of Theorem~\ref{inductivesum} in the $p$-power case as a sum over
such top rows.  Second, in Subsection~\ref{sect62} we shall recall the conjectural formula of Beineke,
Brubaker and Frechette \cite{BeBrF1}, which constructs certain coefficients $H_{BZL}$ directly from
the crystal graph.  In Section~\ref{sec:p} we shall show that these coefficients in fact match the coefficients
$H$ of the Eisenstein series, as conjectured.

\subsection{A Combinatorial Version of the Inductive Formula}\label{sect61} 

Let ${\bf m}$ be an $r$-tuple of non-negative integers (its components will be the logarithm base $p$ of the components
of $\bf m$
in the prior sections, but for convenience we keep the same letter), 
and $\mu$ be the $r$-tuple of positive integers given by $\mu_{j}=m_{r+1-j}+1$.
Define $CQ_{1}(\mu)$ to be the set of $(2r-1)$-tuples $(d_{1},\dots,d_{2r-1})$ of non-negative integers satisfying the inequalities
\begin{equation}\label{ineq:CQ-d}
\begin{cases}
d_{j}\leq \mu_{r+1-j} & 1\leq j\leq r,\\
d_{j+1}+d_{2r-j}\leq \mu_{r-j}+d_{j}& 1\leq j\leq r-1.
\end{cases}
\end{equation}
Let $\Ft=(d_{1},d_{2},\dots,d_{2r-1})$ be a $(2r-1)$-tuple of non-negative integers. A {\it weight vector} $k(\Ft)$ of $\Ft$ is
a vector of the form
$$
k(\Ft)=(k_{1}(\Ft),k_{2}(\Ft),\dots,k_{r}(\Ft))
$$
with
\begin{equation}\label{eq:k}
k_{r}(\Ft)=\sum^{r}_{j=1} d_{2r-j} \text{ and }
k_{i}(\Ft)=\sum^{2r-1}_{j=i}d_{j}+d_{r}+\sum^{i-1}_{j=1}d_{2r-j}, \text{ for } 1\leq i<r.
\end{equation}
(In Sect.~\ref{sec:p} and \ref{sec:even} we will also apply similar formulae to other vectors of odd length $2r'-1$ for various $r'$, replacing
$r$ by $r'$.)
When the choice of $\Ft$ is clear, we write $k_{i}$ in place of $k_{i}(\Ft)$.
If $\Ft\in CQ_{1}(\mu)$, the weight $k(\Ft)$ is called {\it strict} if 
\begin{equation}\label{strict}
k_{i+1}-k_{i+2}< \mu_{r-i}+k_{i}-k_{i+1}\text{ for $1\leq i\leq r-2$, and } 0<\mu_{1}+k_{r-1}-2k_{r}.
\end{equation}
These inequalities are equivalent to 
\begin{equation}\label{strictness-ineqs}
d_{i+1}<\mu_{r-i}+d_{i}-d_{2r-i}+d_{2r-i-1} \text{ for $1\leq i\leq r-2$, and } 0<\mu_{1}+d_{r-1}-d_{r+1}.
\end{equation}

Analogous to the top row of the $BZL$-patterns,  define
\begin{equation}\label{eq:Fc}
\bar{\Fc}_{1,j}=\sum^{j}_{i=1}d_{2r-i},~\Fc_{1,r}=\sum^{r}_{i=1}d_{2r-i}+d_{r}, \text{ and } \Fc_{1,j}=\Fc_{1,r}+\sum^{r-1}_{i=j}d_{i}, \text{ for } 1\leq j<r.
\end{equation}
Then we have
\begin{equation}\label{eq:delta-circle}
\Fc_{1,j}-\Fc_{1,j+1}=d_{j}~\text{for $1\leq j\leq 2r-1$ and $j\neq r$,
~and~}\Fc_{1,r}-\bar{\Fc}_{1,r-1}=2d_{r}.
\end{equation}
For convenience in the discussion below, we will write $\Fc_{1,j}=\bar{\Fc}_{1,2r-j}$ for $1\leq j\leq 2r-1$.

We need some facts about Gauss sums.  Let $m,c\in \mathfrak{o}_S$, $c\neq0$, and $t\in\Z$.  Let
$$g_{t}(m,c)=\sum_{\substack{d\bmod c\\ \gcd(c,d)=1}}\ppair{\frac{d}{c}}^{t}_{n}\psi\ppair{\frac{md}{c}}.$$
If $t=1$ we write simply $g(m,c)$.  We recall the following properties of these Gauss sums at prime powers.   Let $q=|p|$.
\begin{itemize}
\item Given a fixed prime $p$, for any integer $t\not\equiv 0\pmod n$ one has
$$
g_{t}(p^{k},p^{l})=\begin{cases}
q^{k}g_{tl}(1,p) &\text{ if $l=k+1$,}\\
\phi(p^{l})&\text{ if $l\leq k$ and $n\mid lt$}\\
0&\text{ otherwise,}
\end{cases}
$$
where $\phi(p^{l})$ denotes Euler's totient function for $\Fo_{S}/p^{l}\Fo_{S}$.
\item For any integer $t\not\equiv0\pmod n$, $g_{t}(1,p)g_{n-t}(1,p)=q$.
\end{itemize}

The Whittaker coefficients will be described below using {\it decorated} BZL-patterns.  For type A,
see \cite{BBF5}, pg.\ 1087; for type C see \cite{BeBrF1}, pg.\ 50. The decorations correspond
to certain conditions on the root strings with regard to the Kashiwara (or root) operators;
see \cite{BBF5}, Sect.~2 and \cite{BBF4}, Chs.~2 and 3.  
We rewrite the inductive formula using a similar idea.
 Given $\Ft\in CQ_{1}(\mu)$, define an array $\Delta(\Ft)=(\Fc_{1,1},\Fc_{1,2},\dots,\bar{\Fc}_{1,1})$.
We decorate $\Delta(\Ft)$ as follows:
\begin{enumerate}
\item  The entry $\Fc_{1,j+1}$ is circled if $\Fc_{1,j}=\Fc_{1,j+1}$ or $\Fc_{1,j}=0$. 
\item The entry $\Fc_{1,j}$ is boxed if equality holds in the upper-bound inequality~\eqref{ineq:CQ-d}.
\end{enumerate}
See also Section~\ref{sec:p} below.

To each entry $c=\Fc_{1,j}$ in a decorated array $\Delta(\Ft)$, we associate the quantity
\begin{equation}\label{dec:Delta}
\gamma_{\Delta}(c)=\begin{cases}
q^{c} & \text{ if the entry $c$ is circled but not boxed,}\\
g(p^{c-1},p^{c}) &  \text{ if the entry $c$ is boxed but not circled,}\\                    
q^c(1-q^{-1}) &  \text{ if the entry $c$ is neither circled nor boxed and $n\mid c$,}\\
0&  \text{ otherwise.}
\end{cases}
\end{equation}
Note that in particular $\gamma_\Delta(c)=0$ if the entry $c$ is both circled and boxed.
Compare \cite{BBF5}, Eqn.~(7).
Define
$G_{\Delta}(\Ft)=q^{-\sum^{r}_{i=1}d_{2r-i}}\prod^{2r-1}_{j=1}\gamma_{\Delta}(\Fc_{1,j}).$

We continue to use multi-index notation.  In particular,
if ${\bf k}=(k_1,\dots,k_r)$ and ${\bf m}=(m_1,\dots,m_r)$ are $r$-tuples of non-negative integers, then
we write $H(p^{\bf k};p^{\bf m})$ for $H(p^{k_{1}},\dots,p^{k_{r}};p^{m_{1}},\dots,p^{m_{r}})$.

A straightforward calculation using Theorem~\ref{inductivesum} establishes
the following Corollary.
\begin{corollary}\label{thm:H} Let $\bf k$ and $\bf m$ be $r$-tuples of non-negative integers.  Then
\begin{equation}\label{ind-equation-Delta}
H(p^{\bf k};p^{\bf m})=\sum_{\substack{{\bf k}', {\bf k}''\\ {\bf k}'+(0,{\bf k}'')={\bf k}}}\,\,
\sum_{\substack{\Ft\in CQ_{1}(\mu)\\ k(\Ft)={\bf k}'}}G_{\Delta}(\Ft)
H(p^{{\bf k}''};p^{{\bf m}'})
\end{equation}
where the outer sum is over tuples ${\bf k'}=(k_1',\dots,k_r')$ and ${\bf k''}=(k''_1,\dots,k_{r-1}'')$ of non-negative integers 
such that ${\bf k}'+(0,{\bf k}'')={\bf k}$ and where 
$${\bf m}'=(m_{2}+k'_{1}+k'_{3}-2k'_{2},\dots,m_{r-1}+k'_{r-2}+2k'_{r}-2k'_{r-1},m_{r}+k'_{r-1}-2k'_{r}).$$
\end{corollary}

\subsection{Coefficients Constructed From Type C Crystal Graphs}\label{sect62}

We now recall the definition for the $p$-power coefficients 
$H_{BZL}(p^{\kappa};p^{\ell})$ defined by Beineke, Brubaker and Frechette in 
\cite{BeBrF1}.  Here $\kappa=(\kappa_{1},\dots,\kappa_{r})$ and ${\ell}=(l_{1},\dots,l_{r})$ are $r$-tuples of non-negative integers. 
These coefficients will be given as weighted sums over certain $BZL$-patterns. 
We will write these coefficients in an inductive form similar to that given in Corollary~\ref{thm:H} above.
(The difference in our choice of the ordering of the simple roots $\alpha_i$ as compared to \cite{BeBrF1}
is reflected in the indexing below.)

Let $\epsilon_{i}$ be the fundamental dominant weights of the dual group $\Sp_{2r}(\C)$, and let $\rho=\sum_i \epsilon_i$
be the Weyl vector.  Given $\ell$, let $\lam=\sum^{r}_{i=1}l_{i}\epsilon_{i}$ and let $\mu=\lambda+\rho$.
We will work on the crystal graph of highest weight $\mu$.
Recall that each BZL-pattern has a weight coming from the projection map from the vertices of the crystal graph to the weight lattice.
The contributions to $H_{BZL}(p^{\kappa};p^{\ell})$ come from BZL-patterns in
a single weight space determined by $\kappa$ in the crystal graph of highest weight $\mu$.
Denote the set of all BZL-patterns for this crystal graph as $BZL(\mu)$.

To give this precisely, let $\Gamma=\Gamma(c_{i,j})\in BZL(\mu)$ be a BZL-pattern (the notation
follows \cite{BeBrF1}).  As explained above,
the entries $c_{i,j}$ may be put into a triangular array; then the entries in the rows are non-negative and weakly decreasing, and the entries $c_{i,j}$ satisfy certain inequalities.  Specifically 
the entries $c_{1,j}$ in the top row satisfy the {\it upper-bound inequalities} 
\begin{equation}\label{ineq:c}
\begin{cases}
\bar{c}_{1,j}\leq \mu_{r-j+1}+\bar{c}_{1,j-1}, &\\
c_{1,j}\leq \mu_{r-j+1}+\bar{c}_{1,j-1}-2\bar{c}_{1,j}+c_{1,j+1}+\bar{c}_{1,j+1},&
\end{cases}
\end{equation}
for all $1\leq j\leq r$, where $\bar{c}_{1,0}=0$.
Define the weight vector $\kappa(\Gamma)$ attached to the BZL-pattern $\Gamma$ by
$$
\kappa(\Gamma)=(\kappa_{1}(\Gamma),\kappa_{2}(\Gamma),\dots,\kappa_{r}(\Gamma))
$$
with
$$
\kappa_{1}(\Gamma)=\sum^{r}_{i=1}c_{i,r} \text{ and } \kappa_{j}(\Gamma)=\sum^{r+1-j}_{i=1} (c_{i,r+1-j}+\bar{c}_{i,r+1-j}), \text{ for } 1<j\leq r.
$$

Recall the {\it decoration rules} for the $BZL$-patterns given in \cite{BeBrF1}:
\begin{enumerate}
\item The entry $c_{i,j}$, $j<2r-i$, is circled if $c_{i,j}=c_{i,j+1}$. The right-most entry in a row, $c_{i,2r-i}$, is circled if it equals 0.
\item  The entry $c_{i,j}$ is boxed if equality holds in the upper-bound inequalities. (See \cite{Litt}, Theorem 6.1 and Corollary 6.1 or
 \cite{BeBrF1}, Section 2.2.)
\end{enumerate}
To each entry $c_{i,j}$ in a decorated BZL-pattern $\Gamma$, we associate the quantity
\begin{equation}\label{dec:Gamma}
\gamma_{\Gamma}(c_{i,j})=\begin{cases}
q^{c_{i,j}} & \text{ if $c_{i,j}$ is circled but not boxed,}\\
g(p^{c_{i,j}-1},p^{c_{i,j}})& \text{ if $c_{i,j}$ is boxed but not circled, and $j\neq r$},\\
g_{2}(p^{c_{i,j}-1},p^{c_{i,j}})& \text{ if $c_{i,j}$ is boxed but not circled, and $j=r$,}\\
q^{c_{i,j}}(1-q^{-1})
& \text{ if $c_{i,j}$ is neither circled nor boxed and $n\mid c_{i,j}$,}\\
0& \text{ otherwise.}
\end{cases}
\end{equation}
(The above definition corrects a misprint in \cite{BeBrF1} by adding the condition $n\mid c_{i,j}$ in
the case that $c_{i,j}$ is undecorated.)
In particular, $\gamma_\Gamma(c_{i,j})=0$ 
if $c_{i,j}$ is both circled and boxed.
For $\Gamma=\Gamma(c_{i,j})$ a decorated BZL-pattern, let $G(\Gamma)=\prod_{i,j}\gamma_\Gamma(c_{i,j})$.
Then the definition of the $p$-coefficients constructed in \cite{BeBrF1} by crystal graphs is:
\begin{equation}\label{crystal-def}
H_{BZL}(p^{\kappa};p^{\ell})=\sum_{\substack{\Gamma\in BZL(\mu)\\ \kappa(\Gamma)=\kappa}}G(\Gamma).
\end{equation}

We rewrite this definition inductively.
Let $BZL_{1}(\mu)$ denote the set of $(2r-1)$-tuples $\Ft=(d_{1},\dots,d_{2r-1})\in \Z^{2r-1}_{\geq 0}$
which satisfy the inequalities
\begin{equation}\label{ineq:BZL-d}
\begin{cases}
d'_{j}\leq \mu_{r-j+1}+d'_{j-1}-d_{2r-j+1} &\text{for $1\leq j\leq r$,}\\
d_{j}+d'_{j}\leq \mu_{r-j+1}+d'_{j-1}-d_{2r-j+1}+d_{2r-j} &\text{for $1\leq j\leq r-1$,}\\
\end{cases}
\end{equation}
where $d_{0}=d_{2r}=0$ and
$d'_{i}=\min\left(d_{i},d_{2r-i}\right)$.  Let
\begin{equation}\label{eq:c}
\bar{c}_{1,j}= \sum^{j-1}_{i=1} d_{2r-i}+d'_{j},~c_{1,r}=\sum^{r}_{i=1}d_{2r-i}, \text{ and }c_{1,j}=k_{j}(\Ft)-\bar{c}_{1,j} \text{~for } 1\leq j<r.
\end{equation}
Here $k_j(\Ft)$ is given by (\ref{eq:k}).  
Note that for $2\leq j\leq r$,
\begin{equation}\label{eq:circle}
\bar{c}_{1,j}-\bar{c}_{1,j-1}=d'_{j}+d_{2r-j+1}-d'_{j-1}  \text{ and }
c_{1,j-1}-c_{1,j}=d'_{j}+d_{j-1}-d'_{j-1}.
\end{equation}
The vector $(c_{1,1},c_{1,2},\dots,c_{1,2r-1})$ is weakly decreasing.

If $\Gamma(\Ft)=(c_{1,1},c_{1,2},\dots,c_{1,2r-1})$ is a one-row array decorated using the rules above,
set $G_{\Gamma}(\Ft)=\prod_{j}\gamma_{\Gamma}(c_{1,j})$.
Then the following Lemma is a direct consequence of (\ref{crystal-def}).
\begin{lemma}\label{ind-lemma} Let $\kappa$ and $\ell$ be $r$-tuples of non-negative integers.  Then
\begin{equation}\label{ind-equation-Gamma}
H_{BZL}(p^{\kappa};p^{\ell})=\sum_{\substack{\kappa',\kappa''\\ \kappa'+(\kappa'',0)=\kappa}}\,\,\sum_{\substack{\Ft\in BZL_{1}(\lam+\rho)\\ k_{i}(\Ft)=\kappa'_{r+1-i} \forall i}}G_{\Gamma}(\Ft)\,\,H_{BZL}(p^{\kappa''};p^{{\ell}'}),
\end{equation}
 where the outer sum is over tuples ${\kappa'}=(\kappa_1',\dots,\kappa_r')$ and 
 ${\kappa''}=(\kappa''_1,\dots,\kappa_{r-1}'')$ of non-negative integers 
such that ${\kappa}'+({\kappa}'',0)={\kappa}$ and where
$${\ell}'=(l_{1}+\kappa'_{2}-2\kappa_{1},l_{2}+\kappa'_{3},\dots,l_{r-1}+\kappa'_{r}+\kappa'_{r-2}-2\kappa'_{r-1}).$$ 
 \end{lemma}

\section{Evaluation of the $p$-parts and $BZL$ descriptions}\label{sec:p}
In this Section, we will prove the conjectural description of the Eisenstein series coefficients
constructed using crystal graphs by
Beineke, Brubaker and Frechette \cite{BeBrF1, BeBrF2}. We establish the following identity.
\begin{theorem}\label{thm:main}
Let $\kappa=(\kappa_1,\dots,\kappa_r)$ and $\ell=(l_1,\dots,l_r)$ be $r$-tuples of non-negative integers.  Then
\begin{equation}\label{eq:H-CQ=BZL}
H(p^{\bf k};p^{\bf m})=H_{BZL}(p^{\bf \kappa};p^{\ell}),
\end{equation}
where $m_{i}=l_{r+1-i}$ and $k_{i}=\kappa_{r+1-i}$ for all $1\leq i\leq r$.
\end{theorem}
We prove this theorem by induction. In the rest of this paper, we assume that $m_{i}=l_{r+1-i}$ for all $i$. First, we consider the top rows in the definitions of these two terms.   Let $\mu$ be as in Section~\ref{sect6}.
Let 
$$
H_{\Gamma}(p^{\kappa}; p^{\ell})=\sum_{\substack{\Ft\in BZL_{1}(\mu)\\ k_{i}(\Ft)=\kappa_{r+1-i}\, \forall i}}G_{\Gamma}(\Ft),
\qquad
H_{\Delta}(p^{\bf k}; p^{\bf m})=\sum_{\substack{\Ft\in CQ_{1}(\mu)\\ k(\Ft)={\bf k}}}G_{\Delta}(\Ft).
$$

{\bf Statement A.} {\it Given a fixed strict weight vector ${\bf k}=k(\Ft)$, let $\kappa$ be the vector with components $\kappa_{i}=k_{r+1-i}$. Then}
\begin{equation}\label{eq:H-gamma-delta}
H_{\Gamma}(p^{\kappa};p^{\ell})=H_{\Delta}(p^{\bf k};p^{\bf m}).
\end{equation}

By Corollary~\ref{thm:H} and Lemma~\ref{ind-lemma}, Statement~A implies Theorem~\ref{thm:main}.  Indeed, 
if the weight $k(\Ft)$ is non-strict, then in (\ref{ind-equation-Gamma}) $H_{BZL}(p^{\kappa''};p^{\ell'})=0$ since some entry is both 
boxed and circled, while in (\ref{ind-equation-Delta}) $G_\Delta(\Ft)=0$ as some entry in $\Ft$ is both boxed and circled.
Hence both contributions are zero. The remaining contributions are matched by Statement A together
with the inductive hypothesis.  The rest of this paper will be occupied with the proof of Statement~A. 

Let us compare this to the analogous step for type A.  In type A, at this point the top row equality follows by a term-by-term matching
of the two sides.  (See \cite{BBF5}, Sect.\ 8.)  However, that is not the case here.  Indeed, Statement~A
is comparable to Statement~B of \cite{BBF4}, Ch.\ 6.  In \cite{BBF4}, the authors compare two different descriptions of
the coefficients
(labelled by $\Gamma$ and $\Delta$) corresponding to {\it two different factorizations of the long element into
simple reflections}.  The comparison is highly non-trivial and is a mix of combinatorial geometry and number-theoretic
facts about Gauss sums.  Once we recognize that Statement A of this paper presents an analogous combinatorial 
problem, it is natural to establish 
the comparison by adapting the methods of \cite{BBF4} to the present case.  We shall do this below.

We now give an equivalent description of the
decoration rules for $\Gamma(\Ft)$  and $\Delta(\Ft)$ in Sections~\ref{sect62} (resp.\ \ref{sect61})
for $\Ft$ as in Statement A, and 
then introduce two new 
decorated arrays $\Gamma^{\iota}(\Ft)$ and $\Delta^{\iota}(\Ft)$.

\begin{lemma}  
\begin{enumerate}
\item
Let $\Ft\in BZL_1(\mu)$.  Then 
in $\Gamma(\Ft)$, the entries are decorated as follows:
\begin{enumerate}
\item The entry $c_{1,j}$ for $j\leq r$ is boxed if $d_{j}+d'_{j}=\mu_{r+1-j}+d'_{j-1}-d_{2r+1-j}+d_{2r-j}$. The entry $\bar{c}_{1,j}$ for $j<r$ is boxed if $d'_{j}=\mu_{r+1-j}+d'_{j-1}-d_{2r+1-j}$.
\item The entry $c_{1,j}$ for $j< r$ is circled if $d_{j}=d'_{j}$ and $d'_{j+1}=0$. The entry $\bar{c}_{1,j}$ for $j\leq r$ is circled if $d'_{j}=0$ and $d_{2r+1-j}=d'_{j-1}$.
\end{enumerate}

\item
Let $\Ft\in CQ_1(\mu)$.  Then in $\Delta(\Ft)$, the entries are decorated as follows:
\begin{enumerate}
\item The entry $\Fc_{1,j}$ for $j\leq r$ is boxed if $d_{j}=\mu_{r+1-j}$. The entry $\bar{\Fc}_{1,j}$ for $j<r$ is boxed if $d_{j+1}=\mu_{r-j}+d_{j}-d_{2r-j}$.
\item The entry $\Fc_{1,j}$ is circled if $d_{j-1}=0$ for $j>1$ or $\Fc_{1,j}=0$.
\end{enumerate}
\end{enumerate}
\end{lemma}

This Lemma is immediate from the decoration rules in Section~\ref{sect6} in view of
Eqns.\ (\ref{eq:c}) and (\ref{eq:Fc}).

\begin{definition}
\begin{enumerate}
\item
Let $\iota(x)$ be the function $\iota(x)=N-k_{1}(\Ft)+x$ where $N$ is a multiple of $n$ and $N>k_{1}(\Ft)$.
Let $\Gamma^{\iota}$ be the decorated array with entries $(\iota(c_{1,1}),\iota(c_{1,2}),\dots,\iota(c_{1,2r-1}))$.
The entry $\iota(c_{1,j})$ is decorated as follows:
\begin{enumerate}
\item The entry $\iota(c_{1,j})$ for $j\leq r$ is boxed if $d_{j}+d'_{j}=\mu_{r+1-j}+d'_{j-1}-d_{2r+1-j}+d_{2r-j}$. The entry $\iota(\bar{c}_{1,j})$ for $j<r$ is boxed if $d'_{j}=\mu_{r+1-j}+d'_{j-1}-d_{2r+1-j}$.
\item The entry $\iota(c_{1,j})$ for $j\leq r$ is circled if $d_{j-1}=d'_{j-1}$ and $d'_{j}=0$. The entry $\iota(\bar{c}_{1,j})$ for $j<r$ is circled if $d'_{j+1}=0$ and $d_{2r-j}=d'_{j}$.
\end{enumerate}
(These decoration rules are independent of the choice of $N$.)
Let $G_{\Gamma^{\iota}}(\Ft)=\prod^{2r-1}_{j=1}\gamma_{\Gamma}(\iota(c_{1,j}))$, where $\gamma_\Gamma$ is given
by (\ref{dec:Gamma}).

\item
Let $\Delta^{\iota}$ be the decorated array with entries
$(\Fc(\Delta^{\iota})_{1,1},\dots,\Fc(\Delta^{\iota})_{1,2r-1})$, where $\Fc(\Delta^{\iota})_{1,j}=\iota(\Fc_{1,j+1})$ for $1\leq j\leq 2r-2$ and $\Fc(\Delta^{\iota})_{1,2r-1}=N-k_{1}$. 
The entry $\Fc(\Delta^{\iota})_{1,j}$ is decorated as follows:
\begin{enumerate}
\item The entry $\Fc(\Delta^{\iota})_{1,j}$ for $j<r$ is boxed if $d_{j+1}=\mu_{r-j}$. The entry $\bar{\Fc}(\Delta^{\iota})_{1,j}$ for $j\leq r$ is boxed if $d_{j}=\mu_{r+1-j}+d_{j-1}-d_{2r-j+1}$.
\item The entry $\Fc(\Delta^{\iota})_{1,j}$ for $1\leq j\leq 2r-1$ is circled if $d_{j}=0$.
\end{enumerate}
Let $G_{\Delta^{\iota}}(\Ft)=q^{\sum^{r}_{i=1}d_{i}}\prod^{2r-1}_{j=1}\gamma_{\Delta}(\Fc(\Delta^{\iota})_{1,j})$,
where $\gamma_\Delta$ is given by (\ref{dec:Delta}).
\end{enumerate}
\end{definition}

We give one example.

\begin{example}\label{ex:r=1}\rm
Suppose $r=1$.
 Let $d_{1}\leq \mu$. In $\Gamma(\Ft)$ and $\Delta(\Ft)$, the entries $c_{1,1}=d_{1}$ and $\Fc_{1,1}=2d_{1}$ are circled if $d_{1}=0$ and boxed if $d_{1}=\mu$. Hence in all cases (using the properties of Gauss sums and
 that $n$ is odd) we have
$$G_{\Gamma}(\Ft)=g_{2}(p^{\mu-1},p^{d_{1}})=p^{-d_{1}}g(p^{\mu+d_1-1},p^{2d_{1}})=G_{\Delta}(\Ft).$$
In $\Gamma^{\iota}(\Ft)$ and $\Delta^{\iota}(\Ft)$, the entries $\iota(c_{1,1})=N-d_{1}$ and $\Fc(\Delta^{\iota})_{1,1}=N-2d_{1}$ are circled if $d_{1}=0$ and boxed if $d_{1}=\mu$. Hence in all cases
$$
G_{\Gamma^{\iota}}(\Ft)=\gamma_{\Gamma}(N-d_{1})=q^{d_{1}}\gamma_{\Delta}(N-2d_{1})=G_{\Delta^{\iota}}(\Ft).
$$
Note that the quantities $G_{\Gamma^{\iota}(\Ft)}$, $G_{\Delta^{\iota}}(\Ft)$ depend on the choice of $N$ as above
but they are equal for any such choice.
\end{example}

\begin{lemma}\label{lm:CQ=BZL}
Given $\mu=(\mu_{1},\dots,\mu_{r})$ in $\Z^{r}_{\geq 0}$,
$
CQ_{1}(\mu)=BZL_{1}(\mu).
$
\end{lemma}
\begin{proof}

It is sufficient to show that the inequalities~\eqref{ineq:BZL-d} are equivalent to the inequalities~\eqref{ineq:CQ-d}.  
When $r=1$ this is trivial.  We prove this Lemma for $r\geq 2$ by induction.

First, when $r=2$, the inequalities~\eqref{ineq:BZL-d} and \eqref{ineq:CQ-d} become
$$
\begin{cases}
d'_{1}\leq \mu_{2}&\\
d_{2}+d_{3}\leq \mu_{1}+d'_{1}&\\
d_{1}+d'_{1}\leq \mu_{2}+d_{3},&
\end{cases}
\text{ and }
\begin{cases}
d_{1}\leq \mu_{2}\\
d_{2}\leq \mu_{1}\\
d_{2}+d_{3}\leq \mu_{1}+d_{1}.
\end{cases}
$$
Since $d_1'=\min(d_1,d_{3})$,
we have $d_{1}\leq \mu_{2}$ if and only if $d'_{1}\leq \mu_{2}$ and $d_{1}+d'_{1}\leq \mu_{2}+d_{3}$.
In addition, $d_{2}\leq \mu_{1}$ and $d_{2}+d_{3}\leq \mu_{1}+d_{1}$ are equivalent to $d_{2}+d_{3}\leq \mu_{1}+d'_{1} $. Therefore, we have $CQ_{1}(\mu)=BZL_{1}(\mu)$.

Similarly, for general $r$ the inequalities $d'_{1}\leq \mu_{r}$ and $d_{1}+d'_{1}\leq \mu_{r}+d_{2r-1}$ are equivalent to $d_{1}\leq \mu_{r}$.
Thus the inequalities \eqref{ineq:BZL-d} are equivalent to
\begin{equation}\label{ineq:d1}
\begin{cases}
d_{1}\leq \mu_{r} &\\
d'_{2}\leq \mu_{r-1}+d'_{1}-d_{2r-1} &\\
d'_{j}\leq \mu_{r-j+1}+d'_{j-1}-d_{2r-j+1}&\text{for $2<j\leq r$}\\
d_{j}+d'_{j}\leq \mu_{r-j+1}+d'_{j-1}-d_{2r-j+1}+d_{2r-j}&\text{for $2<j<r$}\\
d_{2}+d'_{2}\leq \mu_{r-1}+d'_{1}-d_{2r-1}+d_{2r-2}.
\end{cases}
\end{equation}
Let
$$\nu=(\mu_{2},\dots,\mu_{r-2},\mu_{r-1}+d'_{1}-d_{2r-1})\in \Z^{r-1}_{\geq 0}.$$
Then
$(d_{2},\dots,d_{2r-2})$ is in $BZL_{1}(\nu)$.
 By induction, we have $BZL_{1}(\nu)=CQ_{1}(\nu)$, and  the inequalities~\eqref{ineq:d1} are equivalent to
\begin{equation} \label{ineq:d2}
\begin{cases}
d_{1}\leq \mu_{r} &\\
d_{2}\leq  \mu_{r-1}+d'_{1}-d_{2r-1} &\\
d_{j}\leq \mu_{r+1-j}&\text{for $2<j\leq r$}\\
d_{j+1}+d_{2r-j}\leq \mu_{r-j}+d_{j}&\text{for $1<j\leq r-1$.}
\end{cases}
\end{equation}
In addition, $d_{2}\leq  \mu_{r-1}+d'_{1}-d_{2r-1}$ is equivalent to the inequalities $d_{2}\leq \mu_{r-1}$ and $d_{2}+d_{2r-1}\leq \mu_{r-1}+d_{1}$.
Thus the inequalities \eqref{ineq:d2} are equivalent to the inequalities \eqref{ineq:CQ-d}. Therefore, we have $BZL_{1}(\mu)=CQ_{1}(\mu)$ for $\mu\in \Z^{r}_{\geq 0}$.
\end{proof}

Let $\mu\in\Z^{r}_{\geq1}$.  By Lemma~\ref{lm:CQ=BZL}, $CQ_{1}(\mu)$ and $BZL_{1}(\mu)$ define the same set. 
 A sequence $\Ft$ in $CQ_{1}(\mu)$ or $BZL_{1}(\mu)$ is called a {\it short pattern}.  
 (The usage of this term here is not the same as its usage for the type A case in
 \cite{BBF4}.  However, it is justified as Statement B there, an equality of two sums over
 short pattern prototypes, is comparable to our Statement A.)  
   
A decorated array is {\it strict} if there is no entry which is 
both boxed and circled. 
Recall that $\varepsilon_2=1$ if $r>2$ and $\varepsilon_2=2$ if $r=2$, 
and $k_{1}- \varepsilon_{2}k_{2}\leq \mu_{r}$ by~\eqref{ineq:CQ-d}.

\begin{lemma}\label{lm:strictness}  Let $r\geq2$ and let
$\Ft$ be a short pattern such that $k(\Ft)$ is strict and $k_{1}-\varepsilon_{2}k_{2}<\mu_{r}$.   Then 
the decorated array $\Gamma(\Ft)$ is strict  if and only if $\Gamma^{\iota}(\Ft)$ is strict. 
\end{lemma}
\begin{proof}
First, in $\Gamma(\Ft)$ the entry $\bar{c}_{1,j}$ for $j\leq r$ is both boxed and circled if and only if $d'_{j}=0$, $d_{2r+1-j}=d'_{j-1}$ and $\mu_{r+1-j}=0$. Since $\mu\in\Z^{r}_{\geq 1}$,  there is no entry $\bar{c}_{1,j}$ for $j\leq r$ which is 
both boxed and circled. The entry $c_{1,j}$ for $j<r$ is boxed and circled if and only if $d'_{j+1}=0$,
\begin{equation}\label{eq:lm:gamma-strict}
d_{j}=d'_{j}
\text{ and }
2d_{j}=\mu_{r+1-j}+d'_{j-1}-d_{2r+1-j}+d_{2r-j}.
\end{equation}
By the upper bounds~\eqref{ineq:CQ-d}, $d_{j}\leq \mu_{r-j+1}+d_{j-1}-d_{2r-j+1}$ and $d_{j}\leq \mu_{r-j+1}$. 
Since $d_{j-1}'=\min(d_{j-1},d_{2r-j+1})$, Eqn.~\eqref{eq:lm:gamma-strict} is equivalent to $d_{j}=d_{2r-j}=\mu_{r-j+1}+d'_{j-1}-d_{2r-j+1}$. 
The decorated array $\Gamma(\Ft)$ is non-strict if and only if $d_{j}=d_{2r-j}=\mu_{r-j+1}+d'_{j-1}-d_{2r-j+1}$ and $d'_{j+1}=0$ for some $j\leq r-1$.

Next, in $\Gamma^{\iota}(\Ft)$, the entry $\iota(c_{1,j})$ for $j\leq r$ is not boxed and circled  simultaneously,
since that would imply the equalities 
$$d_{j}'=0,~d_{j-1}=d'_{j-1}, \text{ and }
d_j=\mu_{r+1-j}+d_{j-1}-d_{2r-j+1}+d_{2r-j}.
$$
By~(\ref{strictness-ineqs}), this contradicts the strictness of $k(\Ft)$ and $k_{1}-\varepsilon_{2}k_{2}<\mu_{r}$. 
The entry $\iota(\bar{c}_{1,j})$ for $j<r$ is boxed and circled if and only if  
\begin{equation}\label{eq:lm:iota-strict}
d'_{j+1}=0,~d_{2r-j}=d'_{j} \text{ and }
d_{2r-j}=\mu_{r+1-j}+d'_{j-1}-d_{2r-j+1}.
\end{equation}
By the same upper bounds for $d_{j}$, the entry $\iota(c_{1,j})$ for $j<r$ is both boxed and circled if and only if 
$d_{j}=d_{2r-j}=\mu_{r-j+1}+d'_{j-1}-d_{2r-j+1}$ and $d'_{j+1}=0$. The Lemma follows.
\end{proof}

For each short pattern $\Ft=(d_1,\dots,d_{2r-1})\in BZL_1(\mu)$, 
let $i_{0}(\Ft)$ or simply $i_0$ denote the quantity
$\min\cpair{i\mid 1\leq i<r, d_{i}\neq d_{2r-i}}$ provided this set is non-empty.
As $d_{i}-d_{2r-i}=k_{i}-k_{i+1}$ for $i<r-1$ and $d_{r-1}-d_{r+1}=k_{r-1}-2k_{r}$, the index $i_{0}$ is uniquely determined by $k(\Ft)$.

To establish Statement A, we will consider the following three cases: 
\begin{enumerate}
\item The set $\cpair{i\mid 1\leq i<r, d_{i}\neq d_{2r-i}}$ is empty. Such a short pattern will be,
by analogy with \cite{BBF4}, called {\it totally resonant}. 
\item The index $i_{0}$ exists and $d_{i_{0}}>d_{2r-i_{0}}$. We say that  such a short pattern is in {\it Class I}. 
\item The index $i_{0}$ exists and $d_{i_{0}}<d_{2r-i_{0}}$. We say that such a short pattern is in {\it Class II}.
\end{enumerate}
Since this classification of patterns is determined by the weight $k(\Ft)$, we shall similarly call a given weight $\bf k$
totally resonant, Class I, or Class II.

We begin by establishing some results for short patterns which are not totally resonant.  From now on, we fix $\mu$ and
suppose that $\Ft\in BZL_1(\mu)$ unless otherwise indicated.

\begin{lemma}\label{lm:gamma-circle}
Let $\Ft$ be a short pattern that is not totally resonant such that $k(\Ft)$ is strict and such that $\Gamma(\Ft)$ is strict.
Let $c$ be the entry $\bar{c}_{1,j}$ or $c_{1,j}$ of $\Gamma(\Ft)$ (resp.\ $\iota(\bar{c}_{1,j})$ or $\iota(c_{1,j})$ of $\Gamma^{\iota}(\Ft)$) where $1\leq j\leq i_{0}$. 
If  $c$ is not boxed,  then $G_{\Gamma}(\Ft)$ (resp.\ $G_{\Gamma^{\iota}}(\Ft)$) vanishes unless $n$ divides $c$. 
\end{lemma}

\begin{proof}
First, suppose that the entry $c=\bar{c}_{1,j}$ of $\Gamma(\Ft)$ is not boxed.
If it is also not circled, then $G_{\Gamma}(\Ft)$ vanishes unless $\bar{c}_{1,j}$ is divisible by $n$,
by definition.  If it is circled,  then $\bar{c}_{1,j}=\bar{c}_{1,j+1}$,
so we may continue to the right of $\bar{c}_{1,j}$ until we come to the first uncircled entry,
and the same argument applies. 
This can only fail if we come to the edge of the pattern. If this happens, then $\bar{c}_{1,j}$ equals 0 and is divisible by $n$. Let $\bar{c}_{1,j_{0}}$ be the uncircled one for some $1\leq j_{0}<j$. Since  $\bar{c}_{1,j_{0}+1}$ is circled,  we have $d'_{j_{0}+1}=0$ and the entry $c_{1,j_{0}}$ is circled, by $c_{1,j_{0}}-c_{1,j_{0}+1}=d'_{j_{0}+1}$. In addition, since $\Gamma(\Ft)$ is strict, the entry $c_{1,j_{0}}$ is not boxed and $d_{j_{0}}\neq \mu_{r+1-j_{0}}$. Hence, $\bar{c}_{1,j_{0}}$ is neither boxed nor circled. $G_{\Gamma}(\Ft)$ vanishes unless $\bar{c}_{1,j_{0}}=\bar{c}_{1,j}$ is divisible by $n$.

Next, suppose that the entry $c=c_{1,j}$ of $\Gamma(\Ft)$ is not boxed. If it is also not circled, then again
$G_{\Gamma}(\Ft)$ vanishes unless $c_{1,j}$ is divisible by $n$, by definition. If $c_{1,j}$ is circled, 
we continue to the right of $c_{1,j}$ until we come to the first entry $c_{1,j_0}$ that is not circled.  
It is sufficient to show that $c_{1,j_{0}}$ is not boxed. To see this, if $j_{0}\leq i_{0}$, then the entry $c_{1,j_{0}-1}$ is circled and  $d_{j_{0}}=0$. Since $d_{j_{0}}\neq \mu_{r-j_{0}+1}$, the entry $c_{1,j_{0}}$ is also not boxed.
If $j_{0}>i_{0}$, then $c_{1,i_{0}}$ is circled and $d'_{i_{0}+1}+d_{i_{0}}-d'_{i_{0}}=0$. Thus, 
$d_{i_{0}}<d_{2r-i_{0}}$ and $\bar{c}_{1,i_{0}+1}$ is not circled since 
$\bar{c}_{1,i_{0}+1}-\bar{c}_{1,i_{0}}=d_{2r-i_{0}}-d_{i_{0}}>0$. 
We only need to consider the case $i_{0}<j_{0}<2r-i_{0}$. When $j_{0}\leq r$, since  $c_{1,j_{0}-1}$ is circled, $c_{1,j_{0}-1}=c_{1,j_{0}}$, which implies $d'_{j_{0}}=0$ and $d_{j_{0}-1}=d'_{j_{0}-1}$. 
In addition, since $k(\Ft)$ is strict,  we have $d_{j_{0}}+d'_{j_{0}}<\mu_{r-j_{0}+1}+d_{j_{0}-1}-d_{2r-j_{0}+1}+d_{2r-j_{0}}$ and hence $c_{1,j_{0}}$ is not boxed. 
When $r<j_{0}\leq 2r-i_{0}-1$, the entries  $c_{1,\ell}$ for all $i_{0}\leq \ell<j_{0}$ are circled. 
It follows that $d_{\ell}=0$ for all $i_{0}\leq \ell\leq j_{0}$. Since $k(\Ft)$ is strict, this
implies that the entry $c_{1,j_{0}}$ is not boxed.

The proof for $\Gamma^{\iota}(\Ft)$ is similar. For this case, however,
we consider the first uncircled
entry $\iota(c_{1,j_0})$ to the {\sl left} of $c$.  We omit the details.
\end{proof}

A similar result holds for the  quantities $G_{\Delta}(\Ft)$ and $G_{\Delta^{\iota}}(\Ft)$.

\begin{lemma}\label{lm:delta-circle}
Let $\Ft$ be a short pattern that is not totally resonant and such that  $\Delta(\Ft)$ (resp.\ $\Delta^{\iota}(\Ft)$) is strict.
Let $c$ be the entry $\bar{\Fc}_{1,j}$ or $\Fc_{1,j}$ of $\Delta(\Ft)$ 
(resp.\ $\bar{\Fc}(\Delta^{\iota})_{1,j}$ or $\Fc(\Delta^{\iota})_{1,j}$ of $\Delta^{\iota}(\Ft)$) where $1\leq j\leq i_{0}$. 
If  $c$ is not boxed,  then $G_{\Delta}(\Ft)$ (resp.\ $G_{\Delta^{\iota}}(\Ft)$) vanishes unless $n$ divides $c$. 
\end{lemma}

\begin{proof}
Similarly to Lemma~\ref{lm:gamma-circle}, we show that there exists an index $j_0$ such that the 
entry $\Fc_{1,j_{0}}=c$ is neither boxed nor circled.   

First, suppose that the entry $c=\Fc_{1,j}$ of $\Delta(\Ft)$ is not boxed. If it is not circled, then using
the definition of $G_{\Delta}(\Ft)$ the result holds. 
If it is circled, we continue to the left of $\Fc_{1,j}$ until we come to the first uncircled entry, which we denote
$\Fc_{1,j_0}$. 
Since $\Fc_{1,1}$ is never circled in a non-totally-resonant pattern, such an entry must exist.
Since $\Fc_{1,j_{0}+1}$ is circled, we have $d_{j_{0}}=0<\mu_{r+1-j_{0}}$. Hence, $\Fc_{1,j_{0}}$ is not boxed. 

Next suppose that the entry $c=\bar{\Fc}_{1,j}$ of $\Delta(\Ft)$ is not boxed. 
If $\bar{\Fc}_{1,j}=0$, then it is divisible by $n$, so we may assume that $\bar{\Fc}_{1,j}>0$. 
If $c$ is not circled, we are done. If it is circled, we continue to the left of $\bar{\Fc}_{1,j}$ until come to the first entry
which is not circled. Let $c'$ be this entry.
If $c'=\Fc_{1,j_{0}}$ for some $j_{0}\leq r$, then $\Fc_{1,j_{0}+1}$ is circled. It follows that $d_{j_{0}}=0$ and $\Fc_{1,j_{0}}$ is not boxed. If $c'=\bar{\Fc}_{1,j_{0}}$ for some $j<j_{0}<r$, then the entry $\bar{\Fc}_{1,j_{0}-1}$ is circled and $d_{2r-j_{0}}=0$. When $d_{j_{0}}>0$, $d_{j_{0}+1}<\mu_{r-j_{0}}+d_{j_{0}}$ and $\bar{\Fc}_{1,j_{0}}$ is not boxed. When $d_{j_{0}}=0$, the entry $\Fc_{1,j_{0}+1}$ is circled. Since  $\Delta(\Ft)$ is strict, $\Fc_{1,j_{0}+1}$ is not boxed and $d_{j_{0}+1}<\mu_{r-j_{0}}$. Hence, $\bar{\Fc}_{1,j_{0}}$ is not boxed.

The proof for $\Delta^{\iota}(\Ft)$ is similar, and we omit the details.
\end{proof}

If $\Ft$ is not totally resonant, recall that $i_0=\min\cpair{i\mid 1\leq i<r, d_{i}\neq d_{2r-i}}$.
Define  
\begin{align*}
\Ft^{*}&=(d_1, d_2,\dots,d_{i_0-1},d'_{i_0},d_{2r-i_0+1},\dots,d_{2r-2},d_{2r-1})\\
\Ft^{\sharp}&=(d_{i_{0}+1},d_{i_{0}+2},\dots,d_{2r-i_{0}-1}).
\end{align*} 
Let $a=|d_{i_{0}}-d_{2r-i_{0}}|$,
and define $\mu^{*}=(\mu_{r+1-i_{0}}-a,\mu_{r+2-i_{0}},\dots,\mu_{r})$ and $\mu^{\sharp}=(\mu_{1},\mu_{2},\dots,\mu_{r-i_{0}})$ if $\Ft$ is in Class I,  and $\mu^{*}=(\mu_{r+1-i_{0}},\mu_{r+2-i_{0}},\dots,\mu_{r})$ and $\mu^{\sharp}=(\mu_{1},\mu_{2},\dots,\mu_{r-i_{0}}-a)$ if $\Ft$ is in Class II.
Here in Class I if $i_{0}=1$ then $\mu^{*}=\mu_{r}-a$, and in Class II if $i_{0}=r-1$ then $\mu^{\sharp}=\mu_{1}-a$. 
By the bounds~\eqref{ineq:BZL-d},  $\Ft^{*}$ and $\Ft^{\sharp}$ are in $BZL_{1}(\mu^{*})$ and $BZL_{1}(\mu^{\sharp})$ respectively. 
Set $k^{*}=k(\Ft^{*})$ and $k^{\sharp}=k(\Ft^{\sharp})$.
By the definition of weight  in \eqref{eq:k} (with $r$ in \eqref{eq:k} replaced by $i_0$, resp.\ $r-i_0$),  we have
\begin{align}\label{star-sharp1}k^{*}_{i_{0}}&=k^{*}_{j}/2=\sum^{i_{0}-1}_{i=1}d_{i}+d'_{i_{0}}\text{ for $j\leq i_{0}-1$,}\\
\label{star-sharp2}k^{\sharp}_{r-i_{0}}&=\sum^{r-i_{0}}_{i=1}d_{i+i_{0}},\\
\label{star-sharp3}k^{\sharp}_{i}&=\sum^{2(r-i_{0})-1}_{j=i}d_{i_{0}+j}+d_{r}+\sum^{i-1}_{j=1}d_{2r-i_{0}-j}, \text{ for } 1\leq i<r-i_{0}.
\end{align}

Notice that even if $k(\Ft)$ is strict we could we have $\mu_{r+1-i_{0}}=a$ in the Class I case or $\mu_{r-i_{0}}=a$ in 
the Class II case. If $k(\Ft)$ is strict, $\mu_{r+1-i_{0}}=a$ in the Class I case may occur only when $i_{0}=1$, and 
$\mu_{r-i_{0}}=a$ in the Class II case may occur only when $i_{0}\leq r-2$.

A short pattern $\Ft$ in $BZL_{1}(\mu)$ is called {\it maximal} if $d_{i}=d_{2r-i}=\mu_{r+1-i}$ for all $i\leq r$,
and {\it non-maximal} otherwise.  We have the following criterion for non-vanishing.

\begin{lemma}\label{lm:divisibility}
Suppose that $k(\Ft)$ is strict.
\begin{enumerate}
\item
If $\Ft$ is in Class I, then $G_{\Gamma}(\Ft)$ and $G_{\Delta}(\Ft)$ vanish unless $n$ divides $k^{*}_{i_{0}}$,
and $G_{\Gamma^{\iota}}(\Ft)$ and $G_{\Delta^{\iota}}(\Ft)$ vanish unless $n$ divides $k_{1}-k^{*}_{i_{0}}$.
\item
If $\Ft$ is in Class II, then $G_{\Gamma}(\Ft)$ and $G_{\Delta}(\Ft)$ vanish unless $n$ divides $k_{1}-k^{*}_{i_{0}}$,
and $G_{\Gamma^{\iota}}(\Ft)$ and $G_{\Delta^{\iota}}(\Ft)$ vanish unless $n$ divides $k^{*}_{i_{0}}$.
\item
If $\Ft$ is in Class I or Class II (that is, $\Ft$ is not totally resonant) and $\Ft^{*}\in BZL_1(\mu^*)$ is non-maximal, 
then $G_{\Gamma}(\Ft)$, $G_{\Delta}(\Ft)$, $G_{\Gamma^{\iota}}(\Ft)$, and $G_{\Delta^{\iota}}(\Ft)$ vanish unless $n$ divides $k_{1}$.
\end{enumerate}
\end{lemma}
\begin{proof}
Suppose that $\Ft$ is in Class I. 
In $\Gamma(\Ft)$, $\bar{c}_{1,i_{0}}=k^{*}_{i_{0}}$ is not boxed. Indeed, it would be boxed if and only if
$d_{i_0}'=\mu_{r+1-i_0}$, but the inequality (\ref{ineq:c}) rules this out.  Similarly,
in $\Delta(\Ft)$, $\bar{\Fc}_{1,i_{0}}=k^{*}_{i_{0}}$ is not boxed. 
In $\Gamma^{\iota}(\Ft)$, $\iota(\bar{c}_{1,i_{0}})=N-k_{1}+k^{*}_{i_{0}}$ is not boxed.
In $\Delta^{\iota}(\Ft)$, $\bar{\Fc}(\Delta^{\iota})_{1,i_{0}+1}=\iota(\bar{\Fc}_{1,i_{0}})=N-k_{1}+k^{*}_{i_{0}}$ is not boxed. 
Hence the desired divisibility properties follow from Lemmas~\ref{lm:gamma-circle} and \ref{lm:delta-circle}.

Suppose instead that $\Ft$ is in Class II. 
In $\Gamma(\Ft)$, $c_{1,i_{0}}=k_{1}-k^{*}_{i_{0}}$ is not boxed.
In $\Delta(\Ft)$, $\Fc_{1,i_{0}+1}=k_{1}-k^{*}_{i_{0}}$ is not boxed.  
In $\Gamma^{\iota}(\Ft)$, $\iota(c_{1,i_{0}})=N-k^{*}_{i_{0}}$ is not boxed.
In $\Delta^{\iota}(\Ft)$, $\Fc(\Delta^{\iota})_{1,i_{0}}=\iota(\Fc_{1,i_{0}+1})=N-k^{*}_{i_{0}}$ is not boxed.
Again the desired divisibility properties follow from Lemmas~\ref{lm:gamma-circle} and \ref{lm:delta-circle}.

Last, suppose that $\Ft$ is not totally resonant and $\Ft^{*}\in BZL_1(\mu^*)$ is non-maximal.  
Then there exists an index $j\leq i_0$ such that  $d_{j}<\mu^{*}_{i_{0}+1-j}$. 
In $\Gamma(\Ft)$, resp.\ $\Gamma^{\iota}(\Ft)$, the entries
$c_{1,j}$, $\bar{c}_{1,j}$, resp.\ $\iota(c_{1,j})$, $\iota(\bar{c}_{1,j})$, are not boxed. 
In $\Delta(\Ft)$, resp.\ $\Delta^\iota(\Ft)$,  the entries $\Fc_{1,j}$,
$\bar{\Fc}_{1,j-1}$, resp.\ $\Fc(\Delta^{\iota})_{1,j-1}=\iota(\Fc_{1,j})$, $\bar{\Fc}(\Delta^\iota)_{1,j}=\iota(\bar{\Fc}_{1,j-1})$, are not boxed. By Lemmas~\ref{lm:gamma-circle} and \ref{lm:delta-circle},  $G_{\Gamma}(\Ft)$
and $G_{\Delta}(\Ft)$, resp.\ $G_{\Gamma^{\iota}}(\Ft)$ and $G_{\Delta^{\iota}}(\Ft)$, vanish unless $n$ divides $c_{1,j}+\bar{c}_{1,j}=\Fc_{1,j}+\bar{\Fc}_{1,j-1}=k_{1}$,
resp.\  $\iota(c_{1,j})+\iota(\bar{c}_{1,j})=\iota(\Fc_{1,j})+\iota(\bar{\Fc}_{1,j-1})=2N-k_1$. Thus the Lemma holds. 
\end{proof}

Given a fixed weight ${\bf k}$ which is strict with respect to $\mu$, i.e.\ one satisfying (\ref{strict}), 
 let $\FS_{\bf k}(\mu)$ be the set of all short patterns $\Ft$ with $k(\Ft)={\bf k}$. The set $\FS_{\bf k}(\mu)$ depends on ${\bf k}$ and on $\mu$ but we also write $\FS_{\bf k}$ or even $\FS$ for convenience.

For later use, we state the following result, which will allow us to carry out an inductive argument.
\begin{lemma}\label{split-reduce} Suppose that $\bf k$ is in Class I or Class II and has index $i_0$.  Then 
the map $\Ft\to (\Ft^{*},\Ft^{\sharp})$ gives a bijection from
$
\FS_{\bf k}(\mu)$
to
$$ \bigcup_{{\bf k^{*}}, {\bf k^{\sharp}}}\FS_{\bf k^{*}}(\mu^*)\times \FS_{\bf k^{\sharp}}(\mu^{\#}),
$$
where ${\bf k^{*}}$ runs over the totally resonant weight vectors of length $i_0$, ${\bf k^\#}$ runs over 
the weight vectors of length
$r-i_0$,  and the union is over all pairs of weights satisfying 
Eqn.~\eqref{eq:I-k} below in Case I and Eqn.~\eqref{eq:II-k} below in Case II.
\end{lemma}
\begin{proof}
This follows directly from the definitions.
\end{proof}

We now proceed to prove Statement A in each of the three cases enumerated above.

\subsection{The Totally Resonant Case}\label{Tot-Res-Case} \label{sec:resonant}

In this subsection, we consider a short pattern $\Ft$ that is totally resonant. Then  for $1\leq j\leq r-1$, we have
\begin{equation}\label{eq:resonant-c}
k_{j}=2k_{r}=2c_{1,r}=\Fc_{1,1},~
\bar{c}_{1,j}=\bar{\Fc}_{1,j}=\sum^{j}_{i=1}d_{i},
\text{ and }
\Fc_{1,j+1}=c_{1,j}=k_{r}+\sum^{r}_{i=j+1}d_{i}.
\end{equation}
As above, we write $\bar{c}_{1,j}=c_{1,2r-j}$ and $\bar{\Fc}_{1,j}=\Fc_{1,2r-j}$ for $1\leq j\leq r$, for convenience.

We will apply the results concerning totally resonant short pattern prototypes of type A
in \cite{BBF4} to establish Eqn.~\eqref{eq:H-gamma-delta}.
To do so, we assign decorated {\it two-row} arrays, similar to those in \cite{BBF4} Ch.~6 ff., to the short pattern $\Ft$.

Let $\Gamma'(\Ft)$ be the array of nonnegative integers
\begin{equation*}
\Gamma'(\Ft)=\cpair{
\begin{matrix}
\bar c_{1,r}&&\bar{c}_{1,r-1}&&\bar{c}_{1,r-2}&&\cdots&&\bar{c}_{1,1}\\
&c_{1,r-1}&&c_{1,r-2}&&\cdots&&c_{1,1}&
\end{matrix}
},
\end{equation*}
where the entries $c_{1,j}$ are defined by Eqn.~\eqref{eq:c} above (with $\bar c_{1,r}=c_{1,r}$),
and decorated as follows. In $\Gamma'(\Ft)$, $\bar{c}_{1,j}$ is circled if $\bar{c}_{1,j}=\bar{c}_{1,j-1}$ and $\bar{c}_{1,j}$ is boxed if $\bar{c}_{1,j}-\bar{c}_{1,j-1}=\mu_{r-j+1}$, where $1\leq j\leq r$. In the bottom row, $c_{1,j}$ is circled if and only if the $\bar{c}_{1,j+1}$ is circled, and $c_{1,j}$ is boxed if and only if the $\bar{c}_{1,j}$ is boxed.
The array $\Gamma'(\Ft)$ has the property that
each sum of left diagonals $c_{1,i}+\bar{c}_{1,i}$ equals $2k_r$ for $1\leq i<r$; however $\bar{c}_{1,r}=k_r$ and not $2k_r$.
Define  $G_{\Gamma'}(\Ft)=\prod^{2r-1}_{j=1}\gamma_{\Gamma}(c_{1,j})$.  Since the rules
are the same, it is immediate that $G_{\Gamma'}(\Ft)=G_{\Gamma}(\Ft)$.

Also let $\Delta'(\Ft)$ be the array
$$
\Delta'(\Ft)=\cpair{
\begin{matrix}
\Fc_{1,r}&&\Fc_{1,r-1}&&\Fc_{1,r-2}&&\cdots&&\Fc_{1,1}\\
&\bar{\Fc}_{1,r-1}&&\bar{\Fc}_{1,r-2}&&\cdots&&\bar{\Fc}_{1,1}&
\end{matrix}
},
$$
where the entries $\Fc_{1,j}$ are defined in Eqn.~\eqref{eq:Fc} and decorated as follows.
In $\Delta'(\Ft)$, if $1\leq j<r$, then $\Fc_{1,j}$ is circled if $\Fc_{1,j}=\Fc_{1,j+1}$, and $\Fc_{1,j}$ is boxed if $\Fc_{1,j}-\Fc_{1,j+1}=\mu_{r-j+1}$. The entry $\Fc_{1,r}$ is circled if $\Fc_{1,r}=\bar{\Fc}_{1,r-1}$
and it is boxed if $\Fc_{1,r}-\Fc_{1,r+1}=2\mu_{1}$. In the bottom row,  
$\bar{\Fc}_{1,j}$ is circled if and only if the $\Fc_{1,j}$ is circled, and the $\bar{\Fc}_{1,j}$ is boxed if and 
only if  the $\Fc_{1,j+1}$ is boxed. 
The array $\Delta'(\Ft)$ is a $\Delta$-accordion of weight $2k_r$ in the sense of \cite{BBF4}, pg.~43:
each sum of right diagonals $\Fc_{1,j}+\bar{\Fc}_{1,j-1}$ equals $2k_r$, as does ${\Fc}_{1,1}$.
Define  $G_{\Delta'}(\Ft)=q^{-\sum^{r}_{i=1}d_{i}}\prod^{2r-1}_{j=1}\gamma_{\Delta}({\Fc}_{1,j})$.
Notice that the array $\Delta'(\Ft)$ does {\sl not} have the same circling rule as that for $\Delta(\Ft)$.
We will compare $G_{\Delta}(\Ft)$ and $G_{\Delta'}(\Ft)$ below.

\begin{remark}{\rm
The decoration rules for the arrays
$\Gamma'(\Ft)$ and $\Delta'(\Ft)$ above are the same as the decoration rules for the type A two-row 
``accordion" arrays $\Gamma(\Ft)$ and $\Delta(\Ft)$ in the type A totally resonant case
as defined in \cite{BBF4},
Ch.~6.}
\end{remark}

\begin{lemma}\label{lm:Delta=Delta'}
Let $\Ft$ be a totally resonant short pattern. Then
$G_{\Delta}(\Ft)=G_{\Delta'}(\Ft)$.
\end{lemma}
\begin{proof}
First, it easy to check that the following statements are equivalent to each other: (1) $\Delta(\Ft)$ is not   strict; (2) $\Delta'(\Ft)$ is not strict; (3) $d_{j-1}=0$ and $d_{j}=\mu_{r+1-j}$ for some $j$,  $2\leq j\leq r$.
If $\Delta(\Ft)$ is not strict, both sides are zero, and the Lemma is trivial.

Next, assume that $\Delta(\Ft)$ is strict. We must keep track of the circling rules
for the two arrays as they are different.  Consider any maximal string
of consecutive zeros $\{d_{i}\}_{j_{0}\leq i\leq j_{1}}$ in $\Ft$.
That is, $d_{i}=0$ for all $j_{0}\leq i\leq j_{1}$, and $d_{j_{0}-1}\neq 0$,  $d_{j_{1}+1}\neq 0$ provided
these quantities are defined.
We analyze the following 4 cases. 

Case (1): $j_{0}=1$ and $j_{1}=r$. Then $G_{\Delta}(\Ft)=G_{\Delta'}(\Ft)=1$.

Case (2): $j_{0}=1$ and $j_{1}<r$.  We have
$$
\Fc_{1,1}=\Fc_{1,2}=\cdots=\Fc_{1,j_{1}+1} \text{ and }
\bar{\Fc}_{1,j_{1}}=\bar{\Fc}_{1,j_{1}-1}=\cdots=\bar{\Fc}_{1,1}=0.
$$
In $\Delta(\Ft)$, the entries $\Fc_{1,i}$ for $2\leq i\leq j_{1}+1$ and $\bar{\Fc}_{1,i}$ for $1\leq i\leq j_{1}$ are circled. Since $d_{1}\neq \mu_{r}$, the entry $\Fc_{1,1}$ is neither boxed nor circled. Since $\Delta(\Ft)$ is strict, 
keeping track of the circled entries we find that
$\prod^{j_{1}+1}_{i=1}\gamma_{\Delta}(\Fc_{1,i})\prod^{j_{1}}_{i=1}\gamma_{\Delta}(\bar{\Fc}_{1,i})=q^{j_{1}\Fc_{1,1}}\gamma_{\Delta}(\Fc_{1,1})$.
In $\Delta'(\Ft)$, the entries $\Fc_{1,i}$ for $1\leq i\leq j_{1}$ and $\bar\Fc_{1,i}$ for $1\leq i\leq j_{1}$ are circled. Since $0<d_{j_{1}+1}<\mu_{r-j_{1}}$, the entry $\Fc_{1,j_{1}+1}$ is neither boxed nor circled.
 Since $\Delta(\Ft)$ is strict, we arrive at the equality
 $$\prod^{j_{1}+1}_{i=1}\gamma_{\Delta}(\Fc_{1,i})\prod^{j_{1}}_{i=1}\gamma_{\Delta}(\bar{\Fc}_{1,i})=\prod^{j_{1}+1}_{i=1}\gamma_{\Delta'}(\Fc_{1,i})\prod^{j_{1}}_{i=1}\gamma_{\Delta'}(\bar{\Fc}_{1,i}).$$

Case (3): $j_{0}>1$ and $j_{1}=r$. We have
$$
\Fc_{1,j_{0}}=\Fc_{1,j_{0}+1}=\cdots=\Fc_{1,r}=\bar{\Fc}_{1,r-1}=\cdots=\bar{\Fc}_{1,j_{0}-1}.
$$

In $\Delta(\Ft)$, the entries $\Fc_{1,i}$ for $j_{0}+1\leq i\leq 2r+1-j_{0}$ are circled. Since $d_{j_{0}-1}\neq 0$ and $d_{j_{0}}\neq \mu_{r+1-j_{0}}$, the entry $\Fc_{1,j_{0}}$ is neither boxed nor circled. 
Since $\Delta(\Ft)$ is strict, again keeping track of circled entries gives
$\prod^{2r+1-j_{0}}_{i=j_{0}}\gamma_{\Delta}(\Fc_{1,i})=q^{(2(r-j_{0})+1)\Fc_{1,j_0}}\gamma_{\Delta}(\Fc_{1,j_{0}})$.
In $\Delta'(\Ft)$, the entries $\Fc_{1,i}$ for $j_{0}\leq i\leq 2r-j_{0}$ are circled. Since $d_{j_{0}-1}\neq 0$ and $d_{j_{0}}\neq \mu_{r+1-j_{0}}$, the entry $\bar{\Fc}_{1,j_{0}-1}$ is neither boxed nor circled. Since $\Delta(\Ft)$ is strict, we arrive
at the equality
$$\prod^{2r+1-j_{0}}_{i=j_{0}}\gamma_{\Delta}(\Fc_{1,i})=\prod^{2r+1-j_{0}}_{i=j_{0}}\gamma_{\Delta'}(\Fc_{1,i}).$$

Case (4):  $1<j_{0},j_1<r$.   We have
\begin{equation}\label{eq:lm-circle}
\Fc_{1,j_{0}}=\Fc_{1,j_{0}+1}=\cdots=\Fc_{1,j_{1}+1} \text{ and }
\bar{\Fc}_{1,j_{1}}=\bar{\Fc}_{1,j_{1}-1}=\cdots=\bar{\Fc}_{1,j_{0}-1}.
\end{equation}
In $\Delta(\Ft)$, the entries $\Fc_{1,i}$ and $\bar{\Fc}_{1,i-2}$ for $j_{0}+1\leq i\leq j_{1}+1$ are circled. Since $d_{j_{0}}\neq \mu_{r+1-j_{0}}$ and $d_{j_{0}-1}\neq 0$, the entry $\Fc_{1,j_{0}}$ is neither boxed nor circled. 
Since $d_{j_{1}+1}\neq 0$ and $d_{j_{1}+1}\neq \mu_{r-j_{1}}$, the entry $\bar{\Fc}_{1,j_{1}}$ is neither boxed nor circled. Since $\Delta(\Ft)$ is strict, we obtain
$$\prod^{j_{1}+1}_{i=j_{0}}\gamma_{\Delta}(\Fc_{1,i})\prod^{j_{1}}_{i=j_{0}-1}\gamma_{\Delta}(\bar{\Fc}_{1,i})=q^{(j_{1}-j_{0}+1)\Fc_{1,j_{0}}}\gamma_{\Delta}(\Fc_{1,j_{0}})q^{(j_{1}-j_{0}+1)\bar{\Fc}_{1,j_{1}}}\gamma_{\Delta}(\bar{\Fc}_{1,j_{1}}).$$

In $\Delta'(\Ft)$, the entries $\Fc_{1,i-1}$ and $\bar\Fc_{1,i-1}$ for $j_{0}+1\leq i\leq j_{1}+1$ are circled. 
Since $0<d_{j_{1}+1}<\mu_{r-j_{1}}$, the entry $\Fc_{1,j_{1}+1}$ is neither boxed nor circled. Since $d_{j_{0}-1}\neq 0$ and $d_{j_{0}}\neq \mu_{r+1-j_{0}}$, the entry $\bar{\Fc}_{1,j_{0}-1}$ is neither boxed and circled. So we find 
$$
\prod^{j_{1}+1}_{i=j_{0}}\gamma_{\Delta}(\Fc_{1,i})\prod^{j_{1}}_{i=j_{0}-1}\gamma_{\Delta}(\bar{\Fc}_{1,i})=\prod^{j_{1}+1}_{i=j_{0}}\gamma_{\Delta'}(\Fc_{1,i})\prod^{j_{1}}_{i=j_{0}-1}\gamma_{\Delta'}(\bar{\Fc}_{1,i}).
$$

The entries not involving strings of zeroes are identical (including decorations) in the two arrays.  Hence the desired equality holds.
\end{proof}

If $\Ft$ is maximal, then all entries $c_{1,j}$ and $\Fc_{1,j}$ in $\Gamma(\Ft)$ and $\Delta(\Ft)$ are boxed. Since $\Fc_{1,1}=2c_{1,r}$, $\gamma_{\Gamma}(c_{1,r})=q^{-k_{r}}\gamma_{\Delta}(\Fc_{1,1})$. Hence, if $\Ft$ is maximal, Statement~A is true. 

\begin{lemma}\label{lm:know-c}
Assume that $n\nmid 2k_{r}$. Then either $G_{\Gamma'}(\Ft)=G_{\Delta'}(\Ft)=0$ or $\Ft$ is maximal.
\end{lemma}

\begin{proof}
The decoration rules of $\Gamma'(\Ft)$ and $\Delta'(\Ft)$ are the same as the decoration rules of $\Gamma(\Ft)$ and $\Delta(\Ft)$ defined in \cite{BBF4}, Ch.~6.
Thus the proof of Proposition 11.1 in \cite{BBF4} also applies in our situation, and gives this result.
\end{proof}

To prove Statement A for totally resonant short patterns $\Ft$, it remains to handle the 
case that $\Ft$ is non-maximal and $n\mid 2 k_r$.
{\sl Since $n$ is odd}, we reduce to the case that $n \mid k_{r}$,
and we assume this henceforth. 
In order to apply the results in \cite{BBF4}, 
we shift the above arrays
by introducing the decorated arrays of non-negative integers
\begin{equation}\label{Gamma-flat}
\Gamma^{\flat}(\Ft)=\cpair{
\begin{matrix}
\bar{c}_{1,r}&&\bar{c}_{1,r-1}&&\bar{c}_{1,r-2}&&\cdots&&\bar{c}_{1,1}\\
&c_{1,r-1}-k_{r}&&c_{1,r-2}-k_{r}&&\cdots&&c_{1,1}-k_{r}&
\end{matrix}
},
\end{equation}
and 
\begin{equation}\label{Delta-flat}
\Delta^{\flat}(\Ft)=\cpair{
\begin{matrix}
\Fc_{1,r}-k_{r}&&\Fc_{1,r-1}-k_{r}&&\Fc_{1,r-2}-k_{r}&&\cdots&&\Fc_{1,1}-k_{r}\\
&\bar{\Fc}_{1,r-1}&&\bar{\Fc}_{1,r-2}&&\cdots&&\bar{\Fc}_{1,1}&
\end{matrix}
},
\end{equation}
with boxing and circling rules as above.  Observe that $\Gamma^{\flat}(\Ft)$ is a $\Gamma$-accordion of
weight $k_r$ and $\Delta^{\flat}(\Ft)$ is a $\Delta$-accordion of weight $k_r$ in the sense of 
\cite{BBF4}, pgs.\ 42-43.  Passing to these arrays will allow us to use the results of \cite{BBF4} below.
 Define $G_{\Gamma^{\flat}}(\Ft)=\prod^{2r-1}_{j=1}\gamma_{\Delta}(c(\Gamma^{\flat})_{1,j})$ and $G_{\Delta^{\flat}}(\Ft)=\prod^{2r-1}_{j=1}\gamma_{\Delta}(\Fc(\Delta^{\flat})_{1,j})$.
 
\begin{proposition}\label{lm:resonant}
Suppose that the weight $\bf k$ is totally resonant.  Then
$$
\sum_{\Ft\in\FS_{\bf k}}G_{\Gamma}(\Ft)=q^{(r-1)k_{r}}\sum_{\Ft\in\FS_{\bf k}}G_{\Gamma^{\flat}}(\Ft)=q^{(r-1)k_{r}}\sum_{\Ft\in\FS_{\bf k}}G_{\Delta^{\flat}}(\Ft)=\sum_{\Ft\in\FS_{\bf k}}G_{\Delta}(\Ft).
$$
\end{proposition}

\begin{proof}
By the above discussion, we only need to consider the case of $\bf k$ such that  $n$ divides $k_{r}$,
and we may limit the sums to be over $\Ft$ non-maximal.  Since $n\mid  k_r$, the Gauss sums
modulo $c+k_r$ and modulo $c$ are related.  This implies the equalities
$\gamma_{\Gamma}(c+k_{r})=q^{k_{r}}\gamma_{\Gamma}(c)$ and $\gamma_{\Delta}(c+k_{r})=q^{k_{r}}\gamma_{\Delta}(c)$
and so we obtain the equalities
$$\sum_{\Ft\in\FS_{\bf k}}G_{\Gamma}(\Ft)=q^{(r-1)k_{r}}\sum_{\Ft\in\FS_{\bf k}}G_{\Gamma^{\flat}}(\Ft)$$
$$\sum_{\Ft\in\FS_{\bf k}}G_{\Delta}(\Ft)=q^{(r-1)k_{r}}\sum_{\Ft\in\FS_{\bf k}}G_{\Delta^{\flat}}(\Ft).$$
We may now apply Statement C of  \cite{BBF4} to conclude 
that  $\sum_{\Ft\in\FS_{\bf k}}G_{\Gamma^{\flat}}(\Ft)=\sum_{\Ft\in\FS_{\bf k}}G_{\Delta^{\flat}}(\Ft)$. 
(Note that the involution $\Ft\mapsto \Ft'$ of that result is built into the notation
(\ref{Gamma-flat}) and (\ref{Delta-flat}) above.) This completes
the proof.
\end{proof}

We conclude this subsection by establishing a similar result using the arrays $\Gamma^{\iota}(\Ft)$
and $\Delta^{\iota}(\Ft)$.  This will be required to treat the Class I and Class II situations below.

\begin{proposition}\label{pro:resonant}
Suppose that the weight $\bf k$ is totally resonant.  Then
$$
\sum_{\Ft\in \FS_{\bf k}}G_{\Gamma}(\Ft)=q^{(2k_{r}-N)(2r-1)}
\sum_{\Ft\in \FS_{\bf k}} G_{\Gamma^{\iota}}(\Ft)=
q^{(2k_{r}-N)(2r-1)}\sum_{\Ft\in\FS_{\bf k}} G_{\Delta^{\iota}}(\Ft)=\sum_{\Ft\in\FS_{\bf k}}G_{\Delta}(\Ft).
$$
\end{proposition}
\begin{proof}
Since $\Ft$ is totally resonant, $c_{1,j}$ and $\iota(c_{1,2r-j})$, $\Fc(\Delta')_{1,j}$ 
and $\Fc(\Delta^{\iota})_{1,2r-j}$ have the same decorations. In addition,
we have $c_{1,j}+\iota(c_{1,2r-j})=\Fc(\Delta')_{1,j}+\Fc(\Delta^{\iota})_{1,2r-j}=N$ 
for all $1\leq j\leq 2r-1$. It follows that for $1\leq j\leq r-1$,
$$
\gamma_{\Gamma}(c_{1,j})\gamma_{\Gamma}(\bar{c}_{1,j})=
q^{4k_{r}-2N}\gamma_{\Gamma}(\iota(c_{1,j}))\gamma_{\Gamma}(\iota(\bar{c}_{1,j}))
\text{ and }
\gamma_{\Gamma}(c_{1,r})=q^{2k_{r}-N}\gamma_{\Gamma}(\iota(c_{1,r}),
$$
and similar relations hold for $\Fc(\Delta')_{1,j}$ and $\bar\Fc(\Delta^{\iota})_{1,j}$. The result then
follows from Proposition~\ref{lm:resonant}.
\end{proof}

\subsection{The Class I Case}
In this subsection, we consider the case when $\bf k$ is in Class I with fixed index $i_0$. 
Our strategy is as follows:  given a corresponding array
\begin{equation}\label{one-row-array}\left(\begin{matrix}
c_{1,1}&c_{1,2}&\dots&c_{1,r}&\overline{c}_{1,r-1}&\dots&\overline{c}_{1,1}
\end{matrix}\right),
\end{equation}
we break this up into the ``totally resonant" piece
$$\left(\begin{matrix}
c_{1,1}&c_{1,2}&\dots&c_{1,i_0}&\overline{c}_{1,i_0-1}&\dots&\overline{c}_{1,1}
\end{matrix}\right),$$
the lower rank piece
$$\left(\begin{matrix}
c_{1,i_0+1}&\dots&c_{1,r}&\dots&\overline{c}_{1,i_0+1}
\end{matrix}\right),$$
and the singleton $\overline{c}_{1,i_0}$.
The totally resonant piece indeed corresponds to a totally resonant short pattern of lower rank,
and may be treated using the results in Subsection~\ref{Tot-Res-Case}.  The lower rank piece is handled by
induction, and the singleton will match on both sides of the desired equalities.

Now suppose $\Ft$ is a short pattern in Class I with associated array (\ref{one-row-array}).  In this case, we have
\begin{equation}
\begin{array}{ll}
k_{i}(\Ft)=2k^{*}_{i_{0}}+a+k^{\sharp}_{1}
&\text{ for }1\leq i\leq i_{0},\\
k_{i}(\Ft)=2k^{*}_{i_{0}}+k^{\sharp}_{i-i_{0}} &\text{ for } i_{0}+1\leq i<r,\\
k_{r}(\Ft)=k^{*}_{i_{0}}+k^{\sharp}_{r-i_{0}},&
\end{array}\label{eq:I-k}
\end{equation}
where the notation is as in (\ref{star-sharp1}), (\ref{star-sharp2}), and (\ref{star-sharp3})
and the immediately preceding paragraph.
The entries of the decorated arrays $\Gamma(\Ft)$ and $\Gamma^{\iota}(\Ft)$ are given as follows.
\begin{equation} \label{eq:c-*-sharp}
\begin{array}{lll}
\bar{c}_{1,j}=\bar{c}^{*}_{1,j} &\iota(\bar{c}_{1,j})=N-k_{1}+\bar{c}^{*}_{1,j}&\text{ for } 1\leq j\leq i_{0},\\
\bar{c}_{1,j}=k^{*}_{i_{0}}+\bar{c}^{\sharp}_{1,j-i_{0}} &\iota(\bar{c}_{1,j})=N-k_{1}+k^{*}_{i_{0}}+\bar{c}^{\sharp}_{1,j-i_{0}} &\text{ for } i_{0}+1\leq j<r,\\
c_{1,j}=k^{*}_{i_{0}}+c^{\sharp}_{1,j-i_{0}} &\iota(c_{1,j})=N-k_{1}+k^{*}_{i_{0}}+c^{\sharp}_{1,j-i_{0}} &\text{ for } i_{0}+1\leq j\leq r,\\
c_{1,j}=c^{*}_{1,j}+a+k^{\sharp}_{1} &\iota(c_{1,j})=N-2k^{*}_{i_{0}}+c^{*}_{1,j}&\text{ for } 1\leq j\leq i_{0}.
\end{array}
\end{equation}

We define new decorated arrays $\Gamma'(\Ft)$ and $\Gamma'^{\iota}(\Ft)$. 
The entries of $\Gamma'(\Ft)$ and $\Gamma'^{\iota}(\Ft)$ 
are the same as the entries of $\Gamma(\Ft)$ and $\Gamma^{\iota}(\Ft)$, resp.
If $k_{1}-\varepsilon_{2}k_{2}=\mu_{r}$, then the decoration rules for $\Gamma'(\Ft)$ are the same 
as those for $\Gamma(\Ft)$.  In this case, $\Gamma'^{\iota}(\Ft)$ is by definition decorated to be non-strict
(and so will contribute zero to the coefficients).
(Notice that when $k(\Ft)$ is strict, $k_{1}-\varepsilon_{2}k_{2}=\mu_{r}$ if and only if $\mu_{r+1-i_{0}}=a$.)
If $k_{1}-\varepsilon_{2}k_{2}\ne \mu_{r}$, then the decoration rules 
for $\Gamma'(\Ft)$ and $\Gamma'^{\iota}(\Ft)$ are modified from those for 
$\Gamma(\Ft)$ and $\Gamma^{\iota}(\Ft)$, resp., as follows.
In $\Gamma'(\Ft)$, the entry $\bar{c}_{1,i_{0}}$ is neither boxed nor circled. 
The entry $c_{1,i_{0}}$ is boxed if $d_{2r-i_{0}}=\mu_{r+1-i_{0}}-a$, and circled if $d_{2r-i_{0}}=0$.
The rest of the entries in $\Gamma'(\Ft)$ are given the same decorations as $\Gamma(\Ft)$.
In $\Gamma'^{\iota}(\Ft)$, the entry $\iota(\bar{c}_{1,i_{0}})$ is neither boxed nor circled. The entry $\iota(c_{1,j})$ for $i_{0}<j<2r-i_{0}$ is circled if $\iota(c_{1,j})=\iota(c_{1,j+1})$.
The rest of the entries in $\Gamma'^{\iota}(\Ft)$ are given the same decorations as $\Gamma^{\iota}(\Ft)$.
Note that for $j$ in the interval $i_{0}<j<2r-i_{0}$ the decoration rules for $\Gamma'^{\iota}(\Ft)$ are the same as the decoration rules for $\Gamma(\Ft)$.  Also note that 
until now, all entries of zero occurring in decorated arrays were always circled, and so
(when not boxed) received weight $q^0=1$.
With this modification of the decoration rules,
we have the possibility of an undecorated zero, which contributes $1-q^{-1}$ since $n$ always divides $0$.

For the convenience of the reader, we give an example.

\begin{example}\label{r=3 example}\rm
Suppose $r=3$, and let ${\bf k}$ be a weight vector with $i_{0}=2$ and $a=k_{2}-2k_{3}>0$. Fix $\mu$. We have
$$
\FS_{\bf k}(\mu)=\cpair{\Ft=(d_{1},\dots,d_{5})\in \Z^{5}_{\geq 0}\mid d_{1}=d_{5}\leq \mu_{3}, d_{4}\leq \mu_{2}-a, d_{3}\leq \mu_{1},k(\Ft)={\bf k}}.
$$
Also, $\Ft^{*}=(d_{1},d_{4})$ and $\Ft^{\sharp}=(d_{3})$. Then $k^{*}_{1}=2k^{*}_{2}=2(d_{1}+d_{4})$,  $k^{\sharp}_{1}=d_{3}$, $k_{1}=k_{2}=2k^{*}_{2}+a+2k^{\sharp}$ and $k_{3}=k^{\sharp}_{1}+k^{*}_{2}$. The array $\Gamma(\Ft)$ is 
\begin{align*}
\Gamma(\Ft)=&(d_{1}+d_{2}+2d_{3}+d_{4}, d_{2}+2d_{3}+d_{1},d_{3}+d_{4}+d_{1},d_{1}+d_{4},d_{1})\\
=&(k^{*}_{2}+d_{4}+a+2k^{\sharp}_{1},k^{*}_{2}+a+2k^{\sharp}_{1},k^{\sharp}_{1}+k^{*}_{2},k^{*}_{2},d_{1}).
\end{align*}
The decoration rules for $\Gamma(\Ft)$ and $\Gamma'(\Ft)$ are given as follows:
$$
\Gamma(\Ft): \begin{pmatrix}
k^{*}_{2}+d_{4}+a+2k^{\sharp}_{1}&k^{*}_{2}+a+2k^{\sharp}_{1}&k^{\sharp}_{1}+k^{*}_{2}&k^{*}_{2}&d_{1}\\
\fbox{$d_{1}$}\circled{$d_{4}$}&\fbox{$d_{2}$}&\fbox{$d_{3}$}\circled{$d_{3}$}&\circled{$d_{4}$}&\fbox{$d_{1}$}\circled{$d_{1}$}
\end{pmatrix}
$$
$$
\Gamma'(\Ft): \begin{pmatrix}
k^{*}_{2}+d_{4}+a+2k^{\sharp}_{1}&k^{*}_{2}+a+2k^{\sharp}_{1}&k^{\sharp}_{1}+k^{*}_{2}&k^{*}_{2}&d_{1}\\
\fbox{$d_{1}$}\circled{$d_{4}$}&\fbox{$d_{2}$}\circled{$d_{4}$}&\fbox{$d_{3}$}\circled{$d_{3}$}&&\fbox{$d_{1}$}\circled{$d_{1}$}
\end{pmatrix}.
$$
Here the first rows are the arrays, and the second rows are the decoration rules for each array. The symbol
$\fbox{$d_{i}$}$ means the entry above it is boxed if $d_{i}=\mu_{r+1-i}$. The symbol
$\circled{$d_{i}$}$ means the entry above it is circled if $d_{i}=0$.   
\end{example}

Similarly to Lemma~\ref{lm:Delta=Delta'}, 
these modifications in the decorations do not change the associated products
of Gauss sums.  We have:

\begin{lemma}\label{I:gamma=gamma'}
If $\Ft$ is a short pattern in Class I and $k(\Ft)$ is strict, then 
$G_{\Gamma}(\Ft)=G_{\Gamma'}(\Ft)$ and $G_{\Gamma^{\iota}}(\Ft)=G_{\Gamma'^{\iota}}(\Ft)$.
\end{lemma}

\begin{proof}
If $k_{1}-\varepsilon_2 k_{2}=\mu_r$, then
the desired equalities follow immediately from the definitions. Otherwise, since $\Ft\in BZL_1(\mu)$, necessarily
$k_{1}-\varepsilon_2 k_{2}<\mu_r$

Suppose thus that $k_{1}-\varepsilon_2 k_{2}<\mu_r$.
First, we will show that $G_{\Gamma'}(\Ft)=G_{\Gamma}(\Ft)$. It is sufficient to show that 
\begin{equation}\label{eq:lm:c-0}
\gamma_{\Gamma}(c_{1,i_{0}})\gamma_{\Gamma}(\bar{c}_{1,i_{0}})=\gamma_{\Gamma'}(c_{1,i_{0}})\gamma_{\Gamma'}(\bar{c}_{1,i_{0}}).
\end{equation}
If $d_{2r-i_{0}}\neq 0$, these two entries have the same decorations in $\Gamma(\Ft)$ and $\Gamma'(\Ft)$,
so Eqn.~\eqref{eq:lm:c-0} is true. If $d_{2r-i_{0}}=0$, then  $\Ft^{*}$ is not maximal.
Thus, $G_{\Gamma}(\Ft)=G_{\Gamma'}(\Ft)=0$ unless $n$ divides $k_{1}$ and $k^{*}_{i_{0}}$ by Lemma~\ref{lm:divisibility}.
If so, $n$ divides $\overline{c}_{1,i_0}=k_{i_0}^*$ and $n$ divides $c_{1,i_0}=k_1-\overline{c}_{1,i_0}$ (as $d_{2r-i_{0}}=0$).
Thus in this case,
$\gamma_{\Gamma}(c_{1,i_{0}})\gamma_{\Gamma}(\bar{c}_{1,i_{0}})=q^{c_{1,i_{0}}}(1-q^{-1})q^{\bar{c}_{1,i_{0}}}$
and $\gamma_{\Gamma'}(c_{1,i_{0}})\gamma_{\Gamma'}(\bar{c}_{1,i_{0}})=q^{\bar{c}_{1,i_{0}}}(1-q^{-1})q^{c_{1,i_{0}}}.$
Eqn.~\eqref{eq:lm:c-0} follows. 

Second, we show that $G_{\Gamma^{\iota}}(\Ft)=G_{\Gamma'^{\iota}}(\Ft)$.
We only need to consider the entries $\iota(c_{1,j})$ for $i_{0}< j\leq 2r-i_{0}$. 
We may suppose that no entry $c_{1,j}$ for $i_{0}<j\leq 2r-i_{0}$ is both boxed and circled in  $\Gamma^{\iota}(\Ft)$,
since if one were, by Lemma~\ref{lm:strictness}, an entry in $\Gamma'^{\iota}(\Ft)$ would also have
this property and both sides would be zero. 

In $\Gamma^{\iota}(\Ft)$, let $\iota(c_{1,j})$ for $m_{1}\leq j\leq m_{2}$ be a sequence of consecutive circled entries, where $m_{1}> i_{0}$ and $m_{2}\leq 2r-i_{0}$, which is maximal,
i.e.\ such that the entry $\iota(c_{1,m_{1}-1})$ is not circled
and either $\iota(c_{1,m_{2}+1})$ is not circled or $m_{2}=2r-i_{0}$. By the assumption, the entries $\iota(c_{1,j})$ for $m_{1}\leq j\leq m_{2}$ are not boxed. Since $\Gamma(\Ft)$ is strict, the entry $\iota(c_{1,m_{1}-1})$ is not boxed (note $m_{1}>i_{0}+1$ as $\Ft$ is in Class I).

On the other hand, in $\Gamma'^{\iota}(\Ft)$, the entries $\iota(c_{1,j})$ for $m_{1}-1\leq j\leq m_{2}-1$ are circled and the entry $\iota(c_{1,m_{2}})$ is  unboxed and uncircled. Therefore,  $G_{\Gamma^{\iota}}(\Ft)=G_{\Gamma'^{\iota}}(\Ft)$.
\end{proof}

We now carry out similar constructions for the $\Delta$ arrays.
In $\Delta(\Ft)$, we have
\begin{equation}\label{eq:fc-*-sharp}
\begin{array}{lll}
\bar{\Fc}_{1,j}=\bar{\Fc}^{*}_{1,j}=\bar{c}_{1,j} &\iota(\bar{\Fc}_{1,j})=N-k_{1}+\bar{\Fc}^{*}_{1,j}& \text{ for } 1\leq j\leq i_{0},\\
\bar{\Fc}_{1,j}=k^{*}_{i_{0}}+\bar{\Fc}^{\sharp}_{1,j-i_{0}} &\iota(\bar{\Fc}_{1,j})=N-k_{1}+k^{*}_{i_{0}}+\bar{\Fc}^{\sharp}_{1,j-i_{0}}& \text{ for } i_{0}+1\leq j<r,\\
\Fc_{1,j}=k^{*}_{i_{0}}+\Fc^{\sharp}_{1,j-i_{0}}  &\iota(\Fc_{1,j})=N-k_{1}+k^{*}_{i_{0}}+\Fc^{\sharp}_{1,j-i_{0}} &\text{ for } i_{0}+1\leq j\leq r,\\
\Fc_{1,j}=\Fc^{*}_{1,j}+a+k^{\sharp}_{1} &\iota(\Fc_{1,j})=N-2k^{*}_{i_{0}}+\Fc^{*}_{1,j}& \text{ for } 1\leq j\leq i_{0}.
\end{array}
\end{equation}
By Lemma~\ref{lm:delta-circle}, $G_{\Delta}(\Ft)$ vanishes unless $n$ divides $\bar{\Fc}_{1,i_{0}}=k^{*}_{i_{0}}$.

We again define new decorated arrays $\Delta'(\Ft)$   and $\Delta'^{\iota}(\Ft)$, whose
entries are the same as the entries of $\Delta(\Ft)$ and $\Delta^{\iota}(\Ft)$, resp., but whose decorations are 
(usually) modified.
If $k_{1}-\varepsilon_{2}k_{2}=\mu_{r}$, then the decoration rules for $\Delta'(\Ft)$ are the same 
as those for $\Delta(\Ft)$.  In this case, $\Delta'^{\iota}(\Ft)$ is decorated so as
to be non-strict. If instead $k_{1}-\varepsilon_{2}k_{2}\ne \mu_{r}$, then
in $\Delta'(\Ft)$ the entry $\bar{\Fc}_{1,i_{0}}$ is neither boxed nor circled.
The entry $\bar{\Fc}_{1,i_{0}+1}$ is circled if $d_{2r-i_{0}-1}=0$. The 
remaining entries of  $\Delta'(\Ft)$ are assigned the same decorations as in $\Delta(\Ft)$. 
In $\Delta'^{\iota}(\Ft)$, the entry $\overline{\Fc}(\Delta^{\iota})_{1,i_{0}+1}$ is neither boxed nor circled.
The entry $\bar{\Fc}(\Delta^{\iota})_{1,i_{0}+2}$ is circled if $d_{2r-i_{0}-1}=0$.
The remaining entries of $\Delta'^{\iota}(\Ft)$ are assigned the same decorations as in $\Delta^{\iota}(\Ft)$. 

\begin{example}\label{r=3 part 2}\rm
We continue Example~\ref{r=3 example} by giving the $\Delta$ and $\Delta'$ arrays, using the same notation.
We have
\begin{align*}
\Delta(\Ft)=&(2d_{1}+d_{2}+2d_{3}+d_{4},d_{2}+2d_{3}+d_{4}+d_{1},2d_{3}+d_{4}+d_{1},d_{4}+d_{1},d_{1})\\
=&(2k^{*}_{2}+a+2k^{\sharp}_{1},k^{*}_{2}+d_{4}+a+2k^{\sharp}_{1},2k^{\sharp}_{1}+k^{*}_{2},k^{*}_{2},d_{1}).
\end{align*}
The decoration rules for $\Delta(\Ft)$ and $\Delta'(\Ft)$ are:
$$
\Delta(\Ft): \begin{pmatrix}
2k^{*}_{2}+a+2k_1^{\sharp}&k^{*}_{2}+d_{4}+a+2k^{\sharp}_{1}&2k^{\sharp}_{1}+k^{*}_{2}&k^{*}_{2}&d_{1}\\
\fbox{$d_{1}$}&\fbox{$d_{2}$}\circled{$d_{1}$}&\fbox{$d_{3}$}&\circled{$d_{3}$}&\fbox{$d_{2}$}\circled{$d_{4}$}\circled{$d_{1}$}
\end{pmatrix}
$$
$$
\Delta'(\Ft): \begin{pmatrix}
2k^{*}_{2}+a+2k_1^{\sharp}&k^{*}_{2}+d_{4}+a+2k^{\sharp}_{1}&2k^{\sharp}_{1}+k^{*}_{2}&k^{*}_{2}&d_{1}\\
\fbox{$d_{1}$}&\fbox{$d_{2}$}\circled{$d_{1}$}&\fbox{$d_{3}$}\circled{$d_{3}$}&&\fbox{$d_{2}$}\circled{$d_{4}$}\circled{$d_{1}$}
\end{pmatrix}.
$$
\end{example}

\begin{lemma}\label{I:delta=delta'}
If $\Ft$ is a short pattern in Class I and $k(\Ft)$ is strict, then $G_{\Delta}(\Ft)=G_{\Delta'}(\Ft)$ and $G_{\Delta^{\iota}}(\Ft)=G_{\Delta'^{\iota}}(\Ft)$.
\end{lemma}
\begin{proof}
To show that $G_{\Delta}(\Ft)=G_{\Delta'}(\Ft)$, it is sufficient to establish the equality
\begin{equation}\label{eq:lm:fc-0}
\gamma_{\Delta}(\bar{\Fc}_{1,i_{0}+1})\gamma_{\Delta}(\bar{\Fc}_{1,i_{0}})=\gamma_{\Delta'}(\bar{\Fc}_{1,i_{0}+1})\gamma_{\Delta'}(\bar{\Fc}_{1,i_{0}}).
\end{equation}
If $d_{2r-i_{0}-1}\neq 0$, these two entries have the same decorations in $\Delta(\Ft)$ and $\Delta'(\Ft)$,
and Eqn.~\eqref{eq:lm:fc-0} holds. If $d_{2r-i_{0}-1}=0$, then $d_{i_{0}+2}<\mu_{r-i_{0}-2}+d_{i_{0}+1}-d_{2r-i_{0}-1}$ and $\bar{\Fc}_{1,i_{0}+1}$ is not boxed, and Eqn.~\eqref{eq:lm:fc-0} again holds.

The proof for the second equality of the Lemma is similar and is omitted.
\end{proof}

\begin{proposition}\label{pro:I}
Suppose that the weight $\bf k$ is in Class I and is strict.  Then
$$
\sum_{\Ft\in\FS_{\bf k}}G_{\Gamma}(\Ft)=\sum_{\Ft\in\FS_{\bf k}}G_{\Delta}(\Ft)
\text{ and }
\sum_{\Ft\in\FS_{\bf k}}G_{\Gamma^{\iota}}(\Ft)=\sum_{\Ft\in\FS_{\bf k}}G_{\Delta^{\iota}}(\Ft).
$$
\end{proposition}

\begin{proof}
We will prove this Proposition by induction on $r$. To begin the induction,
the equalities for the case $r=1$ are established in Example~\ref{ex:r=1}.

Suppose $r>1$. 
Suppose first that $k_{1}-\varepsilon_{2}k_{2}\ne \mu_{r}$.
By Lemmas~\ref{I:gamma=gamma'}, \ref{I:delta=delta'} and \ref{lm:divisibility}, we have 
$$
G_{\Gamma}(\Ft)=q^{i_{0}(k^{\sharp}_{1}+a)+(2r-2i_{0}-1)k^{*}_{i_{0}}+k^*_{i_0}}(1-q^{-1})G_{\Gamma^{\flat}}(\Ft^{*})G_{\Gamma}(\Ft^{\sharp}),
$$
and
$$
G_{\Delta}(\Ft)=q^{(2r-2i_{0}-1)k^{*}_{i_{0}}+i_{0}(a+k^{\sharp}_{1})+k^*_{i_0}}(1-q^{-1})G_{\Delta^{\flat}}(\Ft^{*})G_{\Delta}(\Ft^{\sharp}),
$$
when $n$ divides $k^{*}_{i_{0}}$, and these quantities are $0$ otherwise.  Since $\mu^*$ is totally resonant, 
by Proposition~\ref{lm:resonant}, we have 
$$
\sum_{\Ft^{*}\in\FS_{{\bf k}^*}(\mu^{*})}G_{\Gamma^{\flat}}(\Ft^{*})=
\sum_{\Ft^{*}\in\FS_{{\bf k}^*}(\mu^{*})}G_{\Delta^{\flat}}(\Ft^{*}).
$$
In addition, by induction,
$$
\sum_{\Ft^{\sharp}\in\FS_{{\bf k}^\sharp}(\mu^{\sharp})}G_{\Gamma}(\Ft^{\sharp})=
\sum_{\Ft^{\sharp}\in\FS_{{\bf k}^\sharp}(\mu^{\sharp})}G_{\Delta}(\Ft^{\sharp}).
$$ 
Referring to Lemma~\ref{split-reduce}, we have
\begin{multline*}
\sum_{\Ft\in\FS_{\bf k}(\mu)}G_{\Gamma}(\Ft)=\\
\sum_{{\bf k^{*}},{\bf k^{\sharp}}}q^{i_{0}(k^{\sharp}_{1}+a)+(2r-2i_{0}-1)k^{*}_{i_{0}}+k^*_{i_0}}(1-q^{-1})\sum_{\Ft^{*}\in\FS_{\bf k^{*}}(\mu^*)}G_{\Gamma^{\flat}}(\Ft^*)\sum_{\Ft^{\sharp}\in\FS_{\bf k^{\sharp}}(\mu^\#)}G_{\Gamma}(\Ft^{\sharp})
\end{multline*}
and 
\begin{multline*}
\sum_{\Ft\in\FS_{\bf k}(\mu)}G_{\Delta}(\Ft)=\\
\sum_{{\bf k^{*}},{\bf k^{\sharp}}}q^{(2r-2i_{0}-1)k^{*}_{i_{0}}+i_{0}(a+k^{\sharp}_{1})+k^*_{i_0}}(1-q^{-1})
\sum_{\Ft^{*}\in\FS_{\bf k^{*}}(\mu^*)}G_{\Delta^{\flat}}(\Ft^{*})
\sum_{\Ft^{\sharp}\in\FS_{\bf k^{\sharp}}(\mu^\#)}G_{\Delta}(\Ft^{\sharp})
\end{multline*}
where the outer sums over ${\bf k^{*}}$ and ${\bf k^{\sharp}}$ are as given in Lemma~\ref{split-reduce}.
Therefore, 
$$\sum_{\Ft\in\FS_{\bf k}(\mu)}G_{\Gamma}(\Ft)=\sum_{\Ft\in\FS_{\bf k}(\mu)}G_{\Delta}(\Ft),$$
as claimed.

For the second equality, by Lemma~\ref{lm:divisibility} we may suppose that $n$ divides $k_{1}-k^{*}_{i_{0}}$
since otherwise both sides are zero.  When $n$ divides $k_{1}-k^{*}_{i_{0}}$,
$$
G_{\Gamma^{\iota}}(\Ft)=q^{(2r-2i_{0})(N-k_{1}+k^{*}_{i_{0}})}(1-q^{-1})\prod^{i_{0}-1}_{j=1}\gamma_{\Gamma}(\iota(c_{1,j}))\gamma_{\Gamma}(\iota(\bar{c}_{1,j}))\cdot \gamma_{\Gamma}(\iota(c_{1,i_{0}}))\cdot G_{\Gamma}(\Ft^{\sharp}),
$$
and
\begin{multline*}
G_{\Delta^{\iota}}(\Ft)=q^{k_{1}-k^{*}_{i_{0}}}\prod^{i_{0}-1}_{j=1}\gamma_{\Delta}(\Fc(\Delta^{\iota})_{1,j})\gamma_{\Delta}(\bar{\Fc}(\Delta^{\iota})_{1,j})\cdot
 \gamma_{\Delta}(\bar{\Fc}(\Delta^{\iota})_{1,i_{0}})\\
 \times (1-q^{-1})q^{(2r-2i_{0})(N-k_{1}+k^{*}_{i_{0}})}G_{\Delta}(\Ft^{\sharp}).
\end{multline*}

If $\Ft^{*}$ is maximal, then all the entries are boxed and thus
\begin{multline*}
\prod^{i_{0}-1}_{j=1}\gamma_{\Gamma}(\iota(c_{1,j}))\gamma_{\Gamma}(\iota(\bar{c}_{1,j}))\cdot \gamma_{\Gamma}(\iota(c_{1,i_{0}}))=\\q^{k_{1}-k^{*}_{i_{0}}}\prod^{i_{0}-1}_{j=1}\gamma_{\Delta}(\Fc(\Delta^{\iota})_{1,j})\gamma_{\Delta}(\bar{\Fc}(\Delta^{\iota})_{1,j})\cdot \gamma_{\Delta}(\bar{\Fc}(\Delta^{\iota})_{1,i_{0}}).
\end{multline*}
If $\Ft^{*}$ is non-maximal, then
$$
G_{\Gamma^{\iota}}(\Ft)=q^{(2r-2i_{0})(N-k_{1}+k^{*}_{i_{0}})+(i_{0}-1)(N-k_{1})+i_{0}(N-k^{*}_{i_{0}})}(1-q^{-1})G_{\Gamma^{\flat}}(\Ft^{*})G_{\Gamma}(\Ft^{\sharp}),
$$
and
$$
G_{\Delta^{\iota}}(\Ft)=q^{(2r-2i_{0})(N-k_{1}+k^{*}_{i_{0}})+(i_{0}-1)(N-k_{1})+i_{0}(N-k^{*}_{i_{0}})}(1-q^{-1})G_{\Delta^{\flat}}(\Ft^{*})G_{\Delta}(\Ft^{\sharp}).
$$

Again referring to Lemma~\ref{split-reduce}, we have
\begin{multline*}
\sum_{\Ft\in\FS_{\bf k}(\mu)}G_{\Gamma^{\iota}}(\Ft)=
\sum_{{\bf k^{*}},{\bf k^{\sharp}}}
q^{(2r-2i_{0})(N-k_{1}+k^{*}_{i_{0}})+(i_{0}-1)(N-k_{1})+i_{0}(N-k^{*}_{i_{0}})}(1-q^{-1})\\ \times
\sum_{\Ft^{*}\in\FS_{\bf k^{*}}(\mu^*)}G_{\Gamma^{\flat}}(\Ft^{*})
\sum_{\Ft^{\sharp}\in\FS_{\bf k^{\sharp}}(\mu^\#)}G_{\Gamma}(\Ft^{\sharp})
\end{multline*}
and 
\begin{multline*}
\sum_{\Ft\in\FS_{\bf k}(\mu)}G_{\Delta^{\iota}}(\Ft)=
\sum_{{\bf k^{*}},{\bf k^{\sharp}}}q^{(2r-2i_{0})(N-k_{1}+k^{*}_{i_{0}})+(i_{0}-1)(N-k_{1})+i_{0}(N-k^{*}_{i_{0}})}(1-q^{-1})
\\ \times
\sum_{\Ft^{*}\in\FS_{\bf k^{*}}(\mu^*)}G_{\Delta^{\flat}}(\Ft^{*})
\sum_{\Ft^{\sharp}\in\FS_{\bf k^{\sharp}}(\mu^\#)}G_{\Delta}(\Ft^{\sharp})
\end{multline*}
where the outer sums over ${\bf k^{*}}$ and ${\bf k^{\sharp}}$ are described in Lemma~\ref{split-reduce}.
By induction, we obtain
 $\sum_{\Ft\in\FS_{\bf k}(\mu)}G_{\Gamma^{\iota}}(\Ft)=\sum_{\Ft\in\FS_{\bf k}(\mu)}G_{\Delta^{\iota}}(\Ft)$,
as claimed.

Finally, suppose instead that $k_{1}-\varepsilon_{2}k_{2}=\mu_{r}$.
Then a similar argument easily reduces the desired result to the $r-1$ case for $\Gamma(\Ft)$ and $\Delta(\Ft)$ and to the non-strict case for $\Gamma^{\iota}(\Ft)$ and $\Delta^{\iota}(\Ft)$. 
\end{proof}

\subsection{The Class II Case}
In this subsection, we consider the case when $\bf k$ is in Class II. The approach to proving
Statement A for such weights is the same as in Class I, but the details require modification.
For $\Ft$ in Class II, we have
\begin{equation}\label{eq:II-k}
\begin{array}{ll}
k_{i}(\Ft)=2k^{*}_{i_{0}}+a+k^{\sharp}_{1}
&\text{ for }1\leq i\leq i_{0},\\
k_{i}(\Ft)=2k^{*}_{i_{0}}+2a+k^{\sharp}_{i-i_{0}} &\text{ for } i_{0}+1\leq i<r,\\
k_{r}(\Ft)=k^{*}_{i_{0}}+a+k^{\sharp}_{r-i_{0}}.&
\end{array}
\end{equation}

The entries of $\Gamma(\Ft)$ and $\Gamma^{\iota}(\Ft)$ are given as follows.
\begin{equation}\label{eq:II-c-*-sharp}
\begin{array}{lll}
\bar{c}_{1,j}=\bar{c}^{*}_{1,j} &\iota(\bar{c}_{1,j})=N-k_{1}+\bar{c}^{*}_{1,j}&\text{ for } 1\leq j\leq i_{0},\\
\bar{c}_{1,j}=k^{*}_{i_{0}}+a+\bar{c}^{\sharp}_{1,j-i_{0}} &\iota(\bar{c}_{1,j})=N-k^{*}_{i_{0}}-k^{\sharp}_{1}+\bar{c}^{\sharp}_{1,j-i_{0}}&\text{ for } i_{0}+1\leq j<r,\\
c_{1,j}=k^{*}_{i_{0}}+a+c^{\sharp}_{1,j-i_{0}} &\iota(c_{1,j})=N-k^{*}_{i_{0}}-k^{\sharp}_{1}+c^{\sharp}_{1,j-i_{0}}&\text{ for } i_{0}+1\leq j\leq r,\\
c_{1,j}=c^{*}_{1,j}+a+k^{\sharp}_{1} &\iota(c_{1,j})=N-2k^{*}_{i_{0}}+c^{*}_{1,j}&\text{ for } 1\leq j\leq i_{0}.
\end{array}
\end{equation}

Once again, we define new decorated arrays
$\Gamma'(\Ft)$ and $\Gamma'^{\iota}(\Ft)$, whose entries
are the same as those of $\Gamma(\Ft)$ and $\Gamma^{\iota}(\Ft)$ (resp.), but with some modifications
to the decoration rules.
In $\Gamma'(\Ft)$, the entry $c_{1,j}$ for $i_{0}< j\leq 2r-i_{0}-1$ is circled if $c_{1,j}=c_{1,j-1}$,
and the rest of the decorations for $\Gamma'(\Ft)$ are the same as those of $\Gamma(\Ft)$. 
In $\Gamma'^{\iota}(\Ft)$, the decoration rules for the entries $\iota(c_{1,j})$ in $\Gamma'^{\iota}(\Ft)$ are the same as
those for $\iota(c_{1,j})$ in $\Gamma^{\iota}(\Ft)$ except  when $j=i_{0}$ or $j=2r-i_{0}$. The entry $\iota(c_{1,i_{0}})$
in $\Gamma'^{\iota}(\Ft)$ is neither boxed nor circled, and the entry $\iota(\bar{c}_{1,i_{0}})$ is boxed if $d_{i_{0}}=\mu_{r+1-i_{0}}$ and circled if $d_{i_{0}}=0$.

\begin{lemma}\label{lm:II-gamma=gamma'}
If $\Ft$ is a short pattern in Class II and $k(\Ft)$ is strict, then 
$G_{\Gamma}(\Ft)=G_{\Gamma'}(\Ft)$ and $G_{\Gamma^{\iota}}(\Ft)=G_{\Gamma'^{\iota}}(\Ft)$.
\end{lemma}
\begin{proof}

We may suppose that no entry $c_{1,j}$ for $i_{0}\leq j\leq 2r-i_{0}$ is both boxed and circled in  $\Gamma(\Ft)$,
as otherwise both sides are zero by
 Lemma~\ref{lm:strictness}. In $\Gamma(\Ft)$, let $c_{1,j}$ for $m_{1}\leq j\leq m_{2}$ be a sequence of circled entries, where $m_{1}\geq i_{0}$ and $m_{2}< 2r-i_{0}-1$,
which is maximal: the entry $c_{1,m_{2}+1}$ is not circled and either the entry
$c_{1,m_{1}-1}$ is not circled or $m_{1}=i_{0}$. 
By strictness, the entries $c_{1,j}$ for $m_{1}\leq j\leq m_{2}$ are not boxed. 
Next, let us show that the entry $c_{1,m_{2}+1}$ is not boxed. When $m_{2}<r$, if $c_{1,m_{2}+1}$ is boxed, 
then $d_{m_{2}+1}=\mu_{r-m_{2}}+d_{m_{2}}-d_{2r-m_{2}}+d_{2r-m_{2}-1}$ and $d_{m_{2}}=d'_{m_{2}}$, 
$d'_{m_{2}+1}=0$
by the circle decoration on $c_{1,m_{2}}$. 
Thus, $\Gamma^{\iota}(\Ft)$ is non-strict contradicting strictness for $\Gamma(\Ft)$ by Lemma~\ref{lm:strictness}.
Similarly when  $m_{2}\geq r$, if $\bar{c}_{1,2r-m_{2}-1}=c_{1,m_{2}+1}$ is boxed, then 
$d_{m_{2}+1}=\mu_{m_{2}-r+2}+d'_{2r-m_{2}-2}-d_{m_{2}+2}$ and $d'_{2r-m_{2}}=0$, $d_{m_{2}+1}=d'_{2r-m_{2}-1}$,
contradicting strictness. 
Hence,  $\prod_{m_{1}\leq j\leq m_{2}+1}\gamma_{\Gamma}(c_{1,j})=(1-q^{-1})q^{\sum_{m_{1}\leq j\leq m_{2}+1}c_{1,j}}$ when $n$ divides $c_{1,m_{2}+1}$ and is zero otherwise.

On the other hand, in $\Gamma'(\Ft)$, the entries $c_{1,j}$ for $m_{1}+1\leq j\leq m_{2}+1$ are circled and the entry $c_{1,m_{1}}$ is neither boxed nor circled. We obtain the same divisibility condition from $c_{1,m_1}$,
and evaluating directly,  $\prod_{m_{1}\leq j\leq m_{2}+1}\gamma_{\Gamma}(c_{1,j})= \prod_{m_{1}\leq j\leq m_{2}+1}\gamma_{\Gamma'}(c_{1,j})$.  Thus  $G_{\Gamma}(\Ft)=G_{\Gamma'}(\Ft)$ as claimed.

Next, we show that $G_{\Gamma^{\iota}}(\Ft)=G_{\Gamma'^{\iota}}(\Ft)$. It is sufficient to prove that 
\begin{equation}\label{eq:II-gamma-iota=gamma-iota'}
\gamma_{\Gamma^{\iota}}(\iota(c_{1,i_{0}}))\gamma_{\Gamma^{\iota}}(\iota(\bar{c}_{1,i_{0}}))=\gamma_{\Gamma'^{\iota}}(\iota(c_{1,i_{0}}))\gamma_{\Gamma'^{\iota}}(\iota(\bar{c}_{1,i_{0}})).
\end{equation}
If $d_{i_{0}}\neq 0$, these two entries have the same decorations in $\Gamma^{\iota}(\Ft)$ and $\Gamma'^{\iota}(\Ft)$, so Eqn.~\eqref{eq:II-gamma-iota=gamma-iota'} is true.
If $d_{i_{0}}=0$, then $\Ft^{*}$ is not maximal.
Thus by Lemma~\ref{lm:gamma-circle} $G_{\Gamma^{\iota}}(\Ft)=G_{\Gamma'^{\iota}}(\Ft)=0$ unless $n$ divides $k_{1}$ and $k^{*}_{i_{0}}$.  If this divisibility holds, then $n$ divides $\iota(c_{1,i_{0}})$ and $n$ divides $\iota(\bar{c}_{1,i_{0}})$.
Evaluating directly, we find that
$$\gamma_{\Gamma^{\iota}}(\iota(c_{1,i_{0}}))\gamma_{\Gamma^{\iota}}(\iota(\bar{c}_{1,i_{0}}))=
\gamma_{\Gamma'^{\iota}}(\iota(c_{1,i_{0}}))\gamma_{\Gamma'^{\iota}}(\iota(\bar{c}_{1,i_{0}}))
=q^{\iota(c_{1,i_{0}})+
\iota(\bar{c}_{1,i_{0}})} (1-q^{-1}),$$ 
and Eqn.~\eqref{eq:II-gamma-iota=gamma-iota'} holds.  
\end{proof}

We turn to the $\Delta$ arrays.
The entries of $\Delta(\Ft)$ and $\Delta^{\iota}(\Ft)$ are given as follows.
\begin{equation}\label{eq:II-fc-*-sharp}
\begin{array}{lll}
\bar{\Fc}_{1,j}=\bar{\Fc}^{*}_{1,j}=\bar{c}_{1,j} &\iota(\bar{\Fc}_{1,j})=N-k_{1}+\bar{\Fc}_{1,j}& \text{ for } 1\leq j<i_{0},\\
\bar{\Fc}_{1,j}=k^{*}_{i_{0}}+a+\bar{\Fc}^{\sharp}_{1,j-i_{0}} &\iota(\bar{\Fc}_{1,j})=N-k^{*}_{i_{0}}-k^{\sharp}_{1}+\bar{\Fc}^{\sharp}_{1,j-i_{0}}& \text{ for } i_{0}\leq j<r,\\
\Fc_{1,j}=k^{*}_{i_{0}}+a+\Fc^{\sharp}_{1,j-i_{0}}  &\iota(\Fc_{1,j})=N-k^{*}_{i_{0}}-k^{\sharp}_{1}+\Fc^{\sharp}_{1,j-i_{0}} &\text{ for } i_{0}< j\leq r,\\
\Fc_{1,j}=\Fc^{*}_{1,j}+a+k^{\sharp}_{1} &\iota(\Fc_{1,j})=N-2k^{*}_{i_{0}}+\Fc^{*}_{1,j}& \text{ for } 1\leq j\leq i_{0}.
\end{array}
\end{equation}

Once again we define new arrays
$\Delta'(\Ft)$ and $\Delta'^{\iota}(\Ft)$, whose entries are the same as $\Delta(\Ft)$ and $\Delta^{\iota}(\Ft)$ (resp.)
but with some modifications to the decoration rules.
In $\Delta'(\Ft)$,
the entry $\Fc_{1,i_{0}+1}$ is neither boxed nor circled; when $i_{0}>1$, the entry $\bar{\Fc}_{1,i_{0}-1}$ is boxed if $d_{i_{0}}=\mu_{r+1-i_{0}}$, and circled if $d_{i_{0}}=0$; when $i_{0}=1$, the entry $\Fc_{1,1}$ is  circled if $d_{1}=0$. The rest 
of the decorations for $\Delta'(\Ft)$ are the same as for $\Delta(\Ft)$.
In $\Delta'^{\iota}(\Ft)$, 
the entry $\Fc(\Delta^{\iota})_{1,i_{0}}$ is neither boxed nor circled, the entry $\bar{\Fc}(\Delta^{\iota})_{1,i_{0}}$ is boxed if $d_{i_{0}}=\mu_{r+1-i_{0}}$ and circled if $d_{i_{0}}=0$, and the remaining entries in $\Delta'^{\iota}(\Ft)$ are
given the same decorations as in $\Delta^{\iota}(\Ft)$.

\begin{lemma}\label{lm:II-delta=delta'}
If $\Ft$ is a short pattern in Class II and $k(\Ft)$ is strict, then $G_{\Delta}(\Ft)=G_{\Delta'}(\Ft)$ and $G_{\Delta^{\iota}}(\Ft)=G_{\Delta'^{\iota}}(\Ft)$.
\end{lemma}
\begin{proof}
In order to prove $G_{\Delta}(\Ft)=G_{\Delta'}(\Ft)$, it is sufficient to show that 
\begin{equation}\label{eq:II-delta=delta'}
\begin{array}{lll}
&\gamma_{\Delta}(\Fc_{1,1})\gamma_{\Delta}(\Fc_{1,2})=\gamma_{\Delta'}(\Fc_{1,1})\gamma_{\Delta'}(\Fc_{1,2})&
(\text{when $i_{0}=1$});\\
&\gamma_{\Delta}(\Fc_{1,i_{0}+1})\gamma_{\Delta}(\bar{\Fc}_{1,i_{0}-1})=\gamma_{\Delta'}(\Fc_{1,i_{0}+1})\gamma_{\Delta'}(\bar{\Fc}_{1,i_{0}-1})&(\text{when $ i_{0}>1$}).
\end{array}
\end{equation}
If $d_{i_{0}}\neq 0$, these two entries have the same decorations in $\Delta(\Ft)$ and $\Delta'(\Ft)$, and Eqn.~\eqref{eq:II-delta=delta'} holds. If $d_{i_{0}}=0$ and $i_{0}=1$, then $\Fc_{1,1}=\Fc_{1,2}$ and $\gamma_{\Delta}(\Fc_{1,1})\gamma_{\Delta}(\Fc_{1,2})=\gamma_{\Delta'}(\Fc_{1,1})\gamma_{\Delta'}(\Fc_{1,2})=q^{2\Fc_{1,1}}(1-q^{-1})$
when $n$ divides $\Fc_{1,1}$ and is zero otherwise, so again Eqn.~(\ref{eq:II-delta=delta'}) is true.
If $d_{i_{0}}=0$ and $i_{0}>1$, then $\Ft^{*}$ is not maximal. Thus, $G_{\Delta}(\Ft)=G_{\Delta'}(\Ft)=0$ unless $n$ divides $k_{1}$ and $k^{*}_{i_{0}}$.  However, evaluating, we find that 
 $\gamma_{\Delta}(\Fc_{1,i_{0}+1})\gamma_{\Delta}(\bar{\Fc}_{1,i_{0}-1})=q^{\Fc_{1,i_{0}+1}+\bar{\Fc}_{1,i_{0}-1}}(1-q^{-1})$ when $n$ divides $\bar{\Fc}_{1,i_{0}-1}$ and is zero otherwise, 
while $\gamma_{\Delta'}(\Fc_{1,i_{0}+1})\gamma_{\Delta'}(\bar{\Fc}_{1,i_{0}-1})$ gives the
same quantity when $n$ divides $\Fc_{1,i_{0}+1}$ and zero otherwise. 
Since $\Fc_{1,i_{0}+1}=k_{1}-k^{*}_{i_{0}}$ and $\bar{\Fc}_{1,i_{0}-1}=k^{*}_{i_{0}}$, 
Eqn.~\eqref{eq:II-delta=delta'} follows.

We turn to the equality of
 $G_{\Delta^{\iota}}(\Ft)$ and $G_{\Delta'^{\iota}}(\Ft)$. As above, we need only consider the case $d_{i_{0}}=0$. Then, 
 $\gamma_{\Delta^{\iota}}(\Fc(\Delta^\iota)_{1,i_0})\gamma_{\Delta^{\iota}}(\bar{\Fc}(\Delta^\iota)_{1,i_0})=
 q^{\iota(\Fc_{1,i_{0}+1})+\iota(\bar{\Fc}_{1,i_{0}-1})}(1-q^{-1})$ when $n$ 
 divides $\iota(\bar{\Fc}_{1,i_{0}-1})$ and is zero otherwise and 
$\gamma_{\Delta'^{\iota}}(\Fc(\Delta^{\iota})_{1,i_0})\gamma_{\Delta'^{\iota}}(\bar{\Fc}(\Delta^\iota)_{1,i_0})$
gives the same quantity when $n$ divides $\iota(\Fc_{1,i_{0}+1})$ and zero otherwise.
Since $d_{i_{0}}=0$, as above $G_{\Delta^{\iota}}(\Ft)=G_{\Delta'^{\iota}}(\Ft)=0$ unless $n$ divides
$\Fc_{1,i_{0}-1}=k^{*}_{i_{0}}$ and $\Fc_{i_{0}+1}=k_{1}-k^{*}_{i_{0}}$. In that case,
since $\iota(x)=N-k_{1}+x$, we see that $n$ divides $\iota(\bar{\Fc}_{1,i_{0}-1})$ and $\iota(\Fc_{1,i_{0}+1})$. This implies the 
desired equality.
\end{proof}

\begin{proposition}\label{pro:II}
Suppose that the weight $\bf k$ is in Class II and is strict.  Then
$$
\sum_{\Ft\in\FS_{\bf k}}G_{\Gamma}(\Ft)=\sum_{\Ft\in\FS_{\bf k}}G_{\Delta}(\Ft)
\text{ and }
\sum_{\Ft\in\FS_{\bf k}}G_{\Gamma^{\iota}}(\Ft)=\sum_{\Ft\in\FS_{\bf k}}G_{\Delta^{\iota}}(\Ft).
$$
\end{proposition}
\begin{proof}
By Lemmas~\ref{lm:II-gamma=gamma'}, \ref{lm:II-delta=delta'} and \ref{lm:divisibility}, we have 
$$
G_{\Gamma}(\Ft)=q^{i_{0}(k_{1}-k^{*}_{i_{0}})} (1-q^{-1})G_{\Gamma^{\flat}}(\Ft^{*})G_{\Gamma^{\iota}}(\Ft^{\sharp}),
$$
and
$$
G_{\Delta}(\Ft)=q^{i_{0}(k_{1}-k^{*}_{i_{0}})} (1-q^{-1})G_{\Delta^{\flat}}(\Ft^{*})G_{\Delta^{\iota}}(\Ft^{\sharp}),
$$
when $n$ divides $k_{1}-k^{*}_{i_{0}}$, while both quantities are zero otherwise.
Thus by an inductive step similar to the proof of Proposition~\ref{pro:I}
we obtain $\sum_{\Ft\in\FS_{\bf k}(\mu)}G_{\Gamma}(\Ft)=\sum_{\Ft\in\FS_{\bf k}(\mu)}G_{\Delta}(\Ft)$.

Next, we show that $\sum_{\Ft\in\FS}G_{\Gamma^{\iota}}(\Ft)=\sum_{\Ft\in\FS}G_{\Delta^{\iota}}(\Ft)$.
When $n$ divides $k^{*}_{i_{0}}$, we have
$$
G_{\Gamma^{\iota}}(\Ft)=\prod^{i_{0}-1}_{j=1}\gamma_{\Gamma}(\iota(c_{1,j}))\gamma_{\Gamma}(\iota(\bar{c}_{1,j}))\cdot \gamma_{\Gamma}(\iota(\bar{c}_{1,i_{0}}))\cdot \phi(p^{N-k^{*}_{i_{0}}})G_{\Gamma^{\iota}}(\Ft^{\sharp}),
$$
and
$$
G_{\Delta^{\iota}}(\Ft)=q^{k^{*}_{i_{0}}}\prod^{i_{0}-1}_{j=1}\gamma_{\Delta}(\Fc(\Delta^{\iota})_{1,j})\gamma_{\Delta}(\bar{\Fc}(\Delta^{\iota})_{1,j})\cdot \gamma_{\Delta}(\bar{\Fc}(\Delta^{\iota})_{1,i_{0}})\cdot \phi(p^{N-k^{*}_{i_{0}}})G_{\Delta^{\iota}}(\Ft^{\sharp}),
$$
where the function $\iota$ arising in $G_{\Gamma^{\iota}}(\Ft^{\sharp})$ and $G_{\Delta^{\iota}}(\Ft^{\sharp})$ is given by $\iota(x)=N-k^{*}_{i_{0}}-k^{\sharp}_{1}+x$.  Both quantities are zero if $n$ does not divide $k^{*}_{i_{0}}$.

If $\Ft^{*}$ is maximal and $n$ divides $k^{*}_{i_{0}}$, then
$$
\prod^{i_{0}-1}_{j=1}\gamma_{\Gamma}(\iota(c_{1,j}))\gamma_{\Gamma}(\iota(\bar{c}_{1,j}))\cdot \gamma_{\Gamma}(\iota(\bar{c}_{1,i_{0}}))=q^{k^{*}_{i_{0}}}\prod^{i_{0}-1}_{j=1}\gamma_{\Delta}(\Fc(\Delta^{\iota})_{1,j})\gamma_{\Delta}(\bar{\Fc}(\Delta^{\iota})_{1,j})\cdot \gamma_{\Delta}(\bar{\Fc}(\Delta^{\iota})_{1,i_0}).
$$

If $\Ft^{*}$ is non-maximal, then $G_{\Gamma^{\iota}}(\Ft)=G_{\Delta^{\iota}}(\Ft)=0$ unless
$n$ divides $k_{1}$.  When $n$ divides $k_{1}$, 
since  $\iota(c_{1,j})=N-2k^{*}_{i_{0}}+c^{*}_{1,j}$ for $1\leq j<i_{0}$ and
$\iota(\bar{c}_{1,j})=N-k_{1}+\bar{c}_{1,j}$ for $1\leq j\leq i_{0}$, we have
$$
\prod^{i_{0}-1}_{j=1}\gamma_{\Gamma}(\iota(c_{1,j}))=q^{(i_{0}-1)(N-2k^{*}_{i_{0}})}\prod^{i_{0}-1}_{j=1}\gamma_{\Gamma}(c^{*}_{1,j})
\text{ and }
\prod^{i_{0}}_{j=1}\gamma_{\Gamma}(\iota(\bar{c}_{1,j}))=q^{i_{0}(N-k_{1})}\gamma_{\Gamma}(\bar{c}^{*}_{1,j}).
$$ 
Similarly,
since $\Fc(\Delta^{\iota})_{1,j}=N-2k^{*}_{i_{0}}+\Fc^{*}_{1,j+1}$ for $1\leq j<i_{0}$ and $\bar{\Fc}(\Delta^{\iota})_{1,j}=N-k_{1}+\bar{\Fc}_{1,j-1}$ for $1\leq j\leq i_{0}$,  we obtain
\begin{equation}\label{eq:delta-t*-bar-II}
\prod^{i_{0}-1}_{j=1}\gamma_{\Delta}(\Fc(\Delta^{\iota})_{1,j})=q^{(i_{0}-1)(N-2k^{*}_{i_{0}})}\prod^{i_{0}}_{j=2}\gamma_{\Delta}(\Fc^{*}_{1,j})
\end{equation}
and 
\begin{equation}\label{eq:delta-t*-II}
\prod^{i_{0}}_{j=1}\gamma_{\Delta}(\bar{\Fc}(\Delta^{\iota})_{1,j})=q^{(i_{0}-1)(N-k_{1})}\prod^{i_{0}-1}_{j=1}\gamma_{\Delta}(\bar{\Fc}^{*}_{1,j})\cdot q^{N-k_{1}-2k^{*}_{0}}\gamma_{\Delta}(\Fc^{*}_{1,1}).
\end{equation}
By~\eqref{eq:delta-t*-bar-II} and \eqref{eq:delta-t*-II},
\begin{align*}
q^{k^{*}_{i_{0}}}\prod^{i_{0}-1}_{j=1}\gamma_{\Delta}(\Fc(\Delta^{\iota})_{1,j})\gamma_{\Delta}(\bar{\Fc}(\Delta^{\iota})_{1,j})\cdot
& \gamma_{\Delta}(\bar{\Fc}(\Delta^{\iota})_{1,i_0})\\
=&q^{i_{0}(N-k_{1})+(i_{0}-1)(N-2k^{*}_{i_{0}})}\cdot q^{-k^{*}_{i_{0}}}\prod^{2i_{0}-1}_{j=1}\gamma_{\Delta}(\Fc^{*}_{1,j})\\
=&q^{i_{0}(N-k_{1})+(i_{0}-1)(N-2k^{*}_{i_{0}})}G_{\Delta}(\Ft^{*}).
\end{align*}

Hence,
$$
G_{\Gamma^{\iota}}(\Ft)=q^{(i_{0}-1)(N-2k^{*}_{i_{0}})+i_{0}(N-k_{1})+(N-k^{*}_{i_{0}})}(1-q^{-1})G_{\Gamma}(\Ft^{*})\, G_{\Gamma^{\iota}}(\Ft^{\sharp})
$$
and
$$
G_{\Delta^{\iota}}(\Ft)=q^{(i_{0}-1)(N-2k^{*}_{i_0})+i_{0}(N-k_{1})+(N-k^{*}_{i_{0}})}(1-q^{-1})G_{\Delta}(\Ft^{*})\,G_{\Delta^{\iota}}(\Ft^{\sharp}).
$$
Since $G_{\Delta}(\Ft^{*})=q^{(i_{0}-1)k^{*}_{i_{0}}}G_{\Delta^{\flat}}(\Ft^{*})$ and $G_{\Gamma}(\Ft^{*})=q^{(i_{0}-1)k^{*}_{i_{0}}}G_{\Gamma^{\flat}}(\Ft^{*})$,
we obtain
$$
G_{\Gamma^{\iota}}(\Ft)=q^{i_{0}(2N-k_{1}-k^{*}_{i_{0}})}(1-q^{-1})\,G_{\Gamma^{\flat}}(\Ft^{*})\,
G_{\Gamma^{\iota}}(\Ft^{\sharp}),
$$
and
$$
G_{\Delta^{\iota}}(\Ft)=q^{i_{0}(2N-k_{1}-k^{*}_{i_{0}})}(1-q^{-1})\,G_{\Delta^{\flat}}(\Ft^{*})\,G_{\Delta^{\iota}}(\Ft^{\sharp}).
$$

Using the above equalities and arguing via induction as in the Class I case, 
we have $\sum_{\Ft\in\FS}G_{\Gamma^{\iota}}(\Ft)=\sum_{\Ft\in\FS}G_{\Delta^{\iota}}(\Ft)$. 
\end{proof}
Combining Propositions \ref{pro:resonant}, \ref{pro:I} and \ref{pro:II}, we conclude that Statement~A is true.
This completes the proof of Theorem~\ref{thm:main}.

\qed

\section{A Crystal Graph Description For Even Degree Covers}\label{sec:even}

In this Section, we will establish an inductive formula for the
contributions at powers of $p$ to the Whittaker coefficients in the case of even degree covers.
The formula is based on sums over certain BZL-patterns as in Section~\ref{sect62}, but with different decoration rules.
This then gives an inductive crystal graph description of the coefficients in the even degree cover case.

Recall that $\varepsilon_{i}=1$ for $1\leq i\leq r-1$ and $\varepsilon_{r}=2$.
Given an $r$-tuple $\bf k$ of non-negative integers, 
we associate a graph as follows: Each $k_{i}$ is a vertex; $k_{i}$ and $k_{j}$ are connected if and 
only if $j=i+1$ and $\varepsilon_{i}k_{i}=\varepsilon_{j}k_{j}$. 
Let $\CE=(k_{i_{1}},k_{i_{1}+1},\dots,k_{i_{2}})$  be a subsequence of ${\bf k}$, 
consisting of all vertices of a connected component of the graph. We call such $\CE$ a connected component of ${\bf k}$.
Define $\ell_{\CE}=i_{1}$ and $r_{\CE}=i_{2}$.
Separating into connected components gives a disjoint partition of ${\bf k}$,
\begin{equation}
{\bf k}=(\CE_{1},\CE_{2},\cdots,\CE_{h}),
\end{equation}
ordered by $\ell_{\CE_{i+1}}=r_{\CE_{i}}+1$ for $1\leq i<h$.

For each connected component $\CE$, $k_{i}$
satisfies the properties that 
$\varepsilon_{i}k_{i}=\varepsilon_{j}k_{j}$ for all $\ell_{\CE}\leq i,j\leq r_{\CE}$; either $\varepsilon_{\ell_{\CE}-1}k_{\ell_{\CE}-1}\ne \varepsilon_{\ell}k_{\ell_{1}}$ or $\ell_{\CE}=1$; and either $\varepsilon_{r_{\CE}}k_{r_{\CE}}\ne \varepsilon_{r_{\CE}+1}k_{r_{\CE}+1}$ or $r_{\CE}=r$. 
If $\ell_{\CE}=1$, let $a_{\CE}=0$, and if $r_{\CE}=r$, let $b_{\CE}=0$.  Otherwise let
$a_{\CE}=|\varepsilon_{\ell_{\CE}-1}k_{\ell_{\CE}-1}-\varepsilon_{\ell_{\CE}}k_{\ell_{\CE}}|$ and $b_{\CE}=|\varepsilon_{r_{\CE}}k_{r_{\CE}}-\varepsilon_{r_{\CE}+1}k_{r_{\CE}+1}|$.

Let $\mu$, as in Sect.~\ref{sect61}, be the $r$-tuple of positive integers obtained from $\bf m$.  Then $\mu$
gives the highest weight of the crystal graph whose BZL patterns we shall employ.
If $r_{\CE}\neq \ell_{\CE}$, define $\mu(\CE)=(\mu(\CE)_{1},\mu(\CE)_{2},\dots,\mu(\CE)_{r_{\CE}-\ell_{\CE}+1})$ by
specifying that $\mu(\CE)_{i}=\mu_{r-r_{\CE}+i}$ for all $1<i< r_{\CE}-\ell_{\CE}+1$ and
that
$$
\mu(\CE)_1=\begin{cases}\mu_{r+1-r_{\CE}}-b_{\CE}&\text{if 
$\varepsilon_{r_{\CE}}k_{r_{\CE}}>\varepsilon_{r_{\CE}+1}k_{r_{\CE}+1}$}\\
\mu_{r+1-r_{\CE}}&\text{if $\varepsilon_{r_{\CE}}k_{r_{\CE}}<\varepsilon_{r_{\CE}+1}k_{r_{\CE}+1}$}
\end{cases}
$$
$$\mu(\CE)_{r_{\CE}-\ell_{\CE}+1}=\begin{cases}\mu_{r+1-\ell_{\CE}}&\text{if 
$\varepsilon_{\ell_{\CE}-1}k_{\ell_{\CE}-1}>\varepsilon_{\ell_{\CE}}k_{\ell_{\CE}}$}\\
\mu_{r+1-\ell_{\CE}}-a_{\CE}&\text{if  
$\varepsilon_{\ell_{\CE}-1}k_{\ell_{\CE}-1}<\varepsilon_{\ell_{\CE}}k_{\ell_{\CE}}$.}
\end{cases}
$$
When $r_{\CE}=\ell_{\CE}$, define
$$
\mu(\CE)_{1}=\begin{cases}
\mu_{r+1-r_{\CE}}-b_{\CE} &\text{if 
$\varepsilon_{r_{\CE}}k_{r_{\CE}}>\varepsilon_{r_{\CE}+1}k_{r_{\CE}+1}$ and $\varepsilon_{\ell_{\CE}-1}k_{\ell_{\CE}-1}>\varepsilon_{\ell_{\CE}}k_{\ell_{\CE}}$}\\
\mu_{r+1-r_{\CE}}-b_{\CE}-a_{\CE} &\text{if 
$\varepsilon_{r_{\CE}}k_{r_{\CE}}>\varepsilon_{r_{\CE}+1}k_{r_{\CE}+1}$ and $\varepsilon_{\ell_{\CE}-1}k_{\ell_{\CE}-1}<\varepsilon_{\ell_{\CE}}k_{\ell_{\CE}}$}\\
\mu_{r+1-r_{\CE}} &\text{if $\varepsilon_{r_{\CE}}k_{r_{\CE}}<\varepsilon_{r_{\CE}+1}k_{r_{\CE}+1}$ and $\varepsilon_{\ell_{\CE}-1}k_{\ell_{\CE}-1}>\varepsilon_{\ell_{\CE}}k_{\ell_{\CE}}$}\\
\mu_{r+1-r_{\CE}}-a_{\CE} &\text{if $\varepsilon_{r_{\CE}}k_{r_{\CE}}<\varepsilon_{r_{\CE}+1}k_{r_{\CE}+1}$ and $\varepsilon_{\ell_{\CE}-1}k_{\ell_{\CE}-1}<\varepsilon_{\ell_{\CE}}k_{\ell_{\CE}}$.}   
\end{cases}
$$

The set $BZL_{1}(\mu(\CE))$ consists of short patterns of length $r_{\CE}-\ell_{\CE}+1$
satisfying (\ref{ineq:BZL-d}) with $\mu$ replaced by $\mu(\CE)$.  Given a weight vector ${\bf v}$ in $\Z^{r_{\CE}-\ell_{\CE}+1}_{\geq 0}$, 
the subset $\FS_{\bf v}(\mu(\CE))$ of $BZL_{1}(\mu(\CE))$ consists of short patterns of weight $\bf v$.

Set 
$$
\Xi_{\bf k}=\cpair{{\bf x}=(x_{1},x_{2},\cdots,x_{h})\in\Z^{h}_{\geq 0}\left.\right|~
 \sum^{h}_{i=1}x_{i}+\sum^{h-1}_{i=1}\frac{b_{\CE_{i}}-\varepsilon_{r_{\CE_{i}}}k_{r_{\CE_{i}}}+\varepsilon_{r_{\CE_{i}}+1}k_{r_{\CE_{i}}+1}}{2}=k_{r}}.
$$ 
For ${\bf x}\in \Xi_{\bf k}$ and $1\leq i\leq h$, let ${\bf k}(\CE_{i},{\bf x})=(2x_{i},2x_{i},\dots,2x_{i},x_{i})$ be a totally resonant weight vector in $\Z^{r_{\CE_{i}}-\ell_{\CE_{i}}+1}_{\geq 0}$. Define a map 
$$
\Psi_{\bf k}\colon \FS_{\bf k}\to \cup_{{\bf x}\in\Xi_{\bf k}}\FS_{{\bf k}(\CE_{1},{\bf x})}(\mu(\CE_{1}))\times
\FS_{{\bf k}(\CE_{2},{\bf x})}(\mu(\CE_{2}))\times\cdots\times\FS_{{\bf k}(\CE_{h},{\bf x})}(\mu(\CE_{h})),
$$
by $\Psi_{\bf k}(\Ft)=(\Ft(\CE_{1}),\Ft(\CE_{2}),\dots,\Ft(\CE_{h}))$, where if $\Ft=(d_1,\dots,d_{2r-1})$ then
\begin{equation}
\Ft(\CE_{i})=
(d_{\ell_{\CE_{i}}},d_{\ell_{\CE_{i}}+1},\dots,d_{r_{\CE_{i}}-1}, d'_{r_{\CE_{i}}},d_{2r-r_{\CE_{i}}+1},\dots,d_{2r-\ell_{\CE_{i}}-1},d_{2r-\ell_{\CE_{i}}}).
\end{equation}
For example, if $\ell_{\CE}=r_{\CE}=1$, then $\Ft(\CE)=(d'_{1})$.
We will also write  $\Ft(\CE_{i})=(d(\CE_{i})_{1},d(\CE_{i})_{2},\dots,d(\CE_{i})_{2(r_{\CE_{i}}-\ell_{\CE_{i}})+1})$;
note that $\Ft(\CE_i)$ is a totally resonant short pattern in $BZL_{1}(\mu(\CE_{i}))$.
\begin{lemma}
The map $\Psi_{\bf k}$ is a bijection from $\FS_{\bf k}$  to
$$
\bigcup_{{\bf x}\in\Xi_{\bf k}}\FS_{{\bf k}(\CE_{1},{\bf x})}(\mu(\CE_{1}))\times\FS_{{\bf k}(\CE_{2},{\bf x})}(\mu(\CE_{2}))\times
\cdots\times\FS_{{\bf k}(\CE_{h},{\bf x})}(\mu(\CE_{h})).
$$
\end{lemma}
\begin{proof}
One may verify that the map $\Psi_{\bf k}$ is bijective directly from the definitions.
\end{proof}

For each $\CE$, define an ordered subset of $\Gamma(\Ft)$ of length $2(r_\CE-\ell_\CE)+1$
$$
\Gamma_{\Ft}(\CE)=(c(\CE)_{1,1},c(\CE)_{1,2},\dots,c(\CE)_{1,r_{\CE}-\ell_{\CE}+1},
\bar{c}(\CE)_{1,r_{\CE}-\ell_{\CE}},\dots,\bar{c}(\CE)_{1,2},\bar{c}(\CE)_{1,1})
$$
with 
$c(\CE)_{1,i}=c_{1,\ell_{\CE}-1+i}$ for $1\leq i\leq r_{\CE}-\ell_{\CE}$, $\bar{c}(\CE)_{1,i}=\bar{c}_{1,\ell_{\CE}-1+i}$ for $1\leq i\leq r_{\CE}-\ell_{\CE}$, and 
$$
c(\CE)_{1,r_{\CE}-\ell_{\CE}+1}=\begin{cases}
c_{1,r_{\CE}}&\text{ if }\varepsilon_{r_{\CE}}k_{r_{\CE}}>\varepsilon_{r_{\CE}+1}k_{r_{\CE}+1} \text{ and } r_{\CE}\ne r,\\
\bar{c}_{1,r_{\CE}}&\text{ if } \varepsilon_{r_{\CE}}k_{r_{\CE}}<\varepsilon_{r_{\CE}+1}k_{r_{\CE}+1}\text{ and } r_{\CE}\ne r,\\
c_{1,r} & \text{ if }  r_{\CE}=r.
\end{cases}
$$
By convention, we write $c(\CE)_{1,2(r_{\CE}-\ell_{\CE}+1)-i}=\bar{c}(\CE)_{1,i}$ for $1\leq i\leq r_{\CE}-\ell_{\CE}+1$.

We shall decorate $\Gamma_{\Ft}(\CE)$ using the following circling and boxing rules.
If $\mu_{r+1-r_{\CE}}=\varepsilon_{r_{\CE}}k_{r_{\CE}}-\varepsilon_{r_{\CE}+1}k_{r_{\CE}+1}$
then the entry $c_{1,r_{\CE}}$ is boxed and $\bar{c}_{1,r_{\CE}}$ is circled.  
Note that in this case,  if {\bf k} is strict then $r_{\CE}=\ell_{\CE}<r$, 
$d_{r_{\CE}}=\mu_{r+1-r_{\CE}}$ and $d_{2r-r_{\CE}}=0$. 
If $\mu_{r+1-r_{\CE}}\ne \varepsilon_{r_{\CE}}k_{r_{\CE}}-\varepsilon_{r_{\CE}+1}k_{r_{\CE}+1}$, 
we decorate $\Gamma_{\Ft}(\CE)$ as follows.
For $1\leq i\leq  r_{\CE}-\ell_{\CE}$,
the entry $c(\CE)_{1,i}$ is circled if $d(\CE)_{i+1}=0$,  and the entry
$\bar{c}(\CE)_{1,i}$ is circled if $d(\CE)_{i}=0$.
The entry $c(\CE)_{1,r_{\CE}-\ell_{\CE}+1}$ is circled if $d(\CE)_{1}=0$. 
For $1\leq i\leq r_{\CE}-\ell_{\CE}$,
the entries $c(\CE)_{1,i}$ and $\bar{c}(\CE)_{1,i}$ are both boxed if $d(\CE)_{i+1}=\mu(\CE)_{r_{\CE}-\ell_{\CE}+1-i}$, 
and the entry $c(\CE)_{1,r_{\CE}-\ell_{\CE}+1}$ is boxed if $d(\CE)_{1}=\mu(\CE)_{r_{\CE}-\ell_{\CE}+1}$.

If the graph attached to $\bf k$ earlier in this Section 
has $h$ components, then the union of the subsets  $\Gamma_{\Ft}(\CE_i)$, $1\leq i\leq h$,
omits $h-1$ entries from $\Gamma(\Ft)$.  Accordingly,
for $\CE=\CE_i$ with $r_{\CE}<r$, define
\begin{equation}\label{eq:c-CE}
c_{\CE}=\begin{cases}
\bar{c}_{1,r_{\CE}}&\text{ if }\varepsilon_{r_{\CE}}k_{r_{\CE}}>\varepsilon_{r_{\CE}+1}k_{r_{\CE}+1},\\
c_{1,r_{\CE}}&\text{ if } \varepsilon_{r_{\CE}}k_{r_{\CE}}<\varepsilon_{r_{\CE}+1}k_{r_{\CE}+1}.
\end{cases}
\end{equation}
No entry $c_{\CE_{i}}$ is either boxed or circled. 
Define $G_{\Gamma_{\Ft}(\CE)}$ be
$$
\begin{cases}
\prod^{2(r_{\CE}-\ell_{\CE})-1}_{i=1}\gamma_{\Gamma}(c(\CE)_{1,i})\cdot q^{c_{\CE}}(1-q^{-1}) &\text{ if $\mu_{r+1-r_{\CE}}\ne \varepsilon_{r_{\CE}}k_{r_{\CE}}-\varepsilon_{r_{\CE}+1}k_{r_{\CE}+1}$ and $r_{\CE}<r$,}\\
\prod^{2(r_{\CE}-\ell_{\CE})-1}_{i=1}\gamma_{\Gamma}(c(\CE)_{1,i}) &\text{ if $r_{\CE}=r$,}\\
\gamma_{\Gamma}(c_{1,r_{\CE}})\gamma_{\Gamma}(\bar{c}_{1,r_{\CE}})& \text{ if $\mu_{r+1-r_{\CE}}=\varepsilon_{r_{\CE}}k_{r_{\CE}}-\varepsilon_{r_{\CE}+1}k_{r_{\CE}+1}$.}
\end{cases}
$$
Here the factors $\gamma_{\Gamma}(\cdot)$ are computed using the decoration rule above.
We emphasize that this decoration rule is not the same as that defined in Section~\ref{sect62}. 
The contributions for the entries $c_{\CE_{i}}$ are handled separately as shown.

Using this, we construct a new weight function 
$G_\Psi(\Ft)$ by defining
\begin{equation}\label{eq:Psi}
G_{\Psi}(\Ft)=\prod^{h}_{i=1}G_{\Gamma_{\Ft}(\CE_{i})}
\end{equation}
when $k(\Ft)$ is strict and $n$ divides $c_{\CE_{i}}$ for $1\leq i<h$, and $G_{\Psi}(\Ft)=0$ otherwise. 

Now we introduce a new inductive formula for the $p$-power contributions to the Whittaker coefficients.
As we shall see, this formula is valid uniformly for covers of all degrees. If $r=1$, define 
$H_{\daleth}(p^{k};p^{m})=H(p^{k};p^{m})$. If $r>1$, define $H_{\daleth}(p^{k};p^{m})$ by the inductive formula
$$
H_{\daleth}(p^{\bf k};p^{\bf m})=\sum_{\substack{{\bf k}', {\bf k}''\\ {\bf k}'+(0,{\bf k}'')={\bf k}}}\,\,
\sum_{\substack{\Ft\in BZL_{1}(\mu)\\ k(\Ft)={\bf k}'}}G_{\Psi}(\Ft)\,
H_{\daleth}(p^{{\bf k}''};p^{{\bf m}'}),
$$
where the outer sum is over tuples ${\bf k'}=(k_1',\dots,k_r')$ and ${\bf k''}=(k''_1,\dots,k_{r-1}'')$ of non-negative integers 
such that ${\bf k}'+(0,{\bf k}'')={\bf k}$ and where 
$${\bf m}'=(m_{2}+k'_{1}+k'_{3}-2k'_{2},\dots,m_{r-1}+k'_{r-2}+2k'_{r}-2k'_{r-1},m_{r}+k'_{r-1}-2k'_{r}).$$
In addition, if ${\bf m}$ is not in $\Z^{r}_{\geq 0}$, we define $H_{\daleth}(p^{\bf k};p^{\bf m})=0$. 

\begin{theorem}\label{theorem-general}  Let $\bf k$ and $\bf m$ be $r$-tuples of non-negative integers.
$$
H_{\daleth}(p^{\bf k};p^{\bf m})=H(p^{\bf k};p^{\bf m}).
$$
\end{theorem}
\begin{proof}
In the definition of $H_{\daleth}(p^{\bf k};p^{\bf m})$ and in the expression for
$H(p^{\bf k};p^{\bf m})$ given in Corollary~\ref{thm:H}, if ${\bf k'}$ is non-strict, then $G_{\Delta}(\Ft)$ and   $H_{\daleth}(p^{{\bf k}''};p^{{\bf m}'})$ vanish, so such $\bf k'$ do not contribute to the coefficients. 
Hence, it is sufficient to prove that for a short pattern $\Ft$ of strict weight ${\bf k'}=k(\Ft)$, the equality
\begin{equation}\label{eq:Psi=Delta}
G_{\Psi}(\Ft)=G_{\Delta}(\Ft)
\end{equation}
holds.

If $h=1$, the graph of ${\bf k'}$ is connected and the decorated array $\Gamma_{\Ft}(\CE_{1})$ is the same the decorated array $\Delta'(\Ft)$ defined in Section~\ref{sec:resonant}. In this case, Eqn.~\eqref{eq:Psi=Delta} is proved in Lemma~\ref{lm:Delta=Delta'}. 

Suppose that $h>1$.  Similarly to the earlier constructions, we move from $\Gamma(\Ft)$ to $\Delta(\Ft)$.
And similarly to earlier in this section, we introduce a new decoration rule for $\Delta(\Ft)$.  To this end,
for each $\CE$, define an ordered subset of $\Delta(\Ft)$ of length $2(r_\CE-\ell_\CE)-1$
$$
\Delta_{\Ft}(\CE)=(\Fc(\CE)_{1,1},\Fc(\CE)_{1,2},\dots,\Fc(\CE)_{1,r_{\CE}-\ell_{\CE}},\Fc(\CE)_{1,r_{\CE}-\ell_{\CE}+1},\bar{\Fc}(\CE)_{1,r_{\CE}-\ell_{\CE}},\dots,\bar{c}(\CE)_{1,2},\bar{c}(\CE)_{1,1}).
$$
Here the entries $\Fc(\CE)_{1,i}$ are given as follows:
\begin{enumerate}
\item If $\ell_{\CE}=1$, or if $\ell_\CE>1$ and $d_{\ell_{\CE}-1}>d_{2r-\ell_{\CE}+1}$, then
$$
\Delta_{\Ft}(\CE)=(\Fc_{1,\ell_{\CE}},\Fc_{1,\ell_{\CE}+1},\dots,\Fc_{1,r_{\CE}},\bar{\Fc}_{1,r_{\CE}-1},\dots,\bar{\Fc}_{1,\ell_{\CE}+1},\bar{\Fc}_{1,\ell_{\CE}});
$$
\item If $\ell_\CE>1$ and $d_{\ell_{\CE}-1}<d_{2r-\ell_{\CE}+1}$, then
$$
\Delta_{\Ft}(\CE)=(\Fc_{1,\ell_{\CE}+1},\Fc_{1,\ell_{\CE}+2},\dots,\Fc_{1,r_{\CE}},\bar{\Fc}_{1,r_{\CE}-1},\dots,\bar{\Fc}_{1,\ell_{\CE}},\bar{\Fc}_{1,\ell_{\CE}-1}).
$$
\end{enumerate}
For example, if $\ell_{\CE}=r_{\CE}=1$, then $\Delta_{\Ft}(\CE)=(\Fc_{1,1})$. 

We now introduce a new decoration rule for $\Delta_{\Ft}(\CE)$. 
If $r_{\CE}=\ell_{\CE}$ and $\mu_{r+1-r_{\CE}}=\varepsilon_{r_{\CE}}k_{r_{\CE}}-\varepsilon_{r_{\CE}+1}k_{r_{\CE}+1}$, 
then the entry $\Fc_{1,r_{\CE}}$ is boxed and the entry
$\bar{\Fc}_{1,r_{\CE}}$ is circled. 
If instead $\CE$ satisfies $r_{\CE}\ne \ell_{\CE}$ or $\mu_{r+1-r_{\CE}}\ne \varepsilon_{r_{\CE}}k_{r_{\CE}}-\varepsilon_{r_{\CE}+1}k_{r_{\CE}+1}$, the rule is as follows.
For $\ell_{\CE}\leq i\leq r_{\CE}-1$,
the entry $\Fc_{1,i+1}$ is circled if $d_{i+1}=0$ ; for $\ell_{\CE}\leq i\leq r_{\CE}-1$, the entry $\bar{\Fc}_{1,i}$ is circled if 
$d_{i}=0$.  This specifies the circling criteria for all but one entry
of $\Delta_{\Ft}(\CE)$; the remaining entry, either $\Fc_{1,\ell_{\CE}}$ or $\bar{\Fc}_{1,\ell_{\CE}-1}$, 
is circled  if $d_{\ell}=0$. The boxing rules for $\Delta_{\Ft}(\CE)$ 
are those already assigned as entries in $\Delta(\Ft)$.  Note that we have only changed the circling rules in giving this new decoration.

Similarly to $c_{\CE}$ defined in~\eqref{eq:c-CE}, we also define, for $\ell_{\CE}>1$,
$$
\Fc_{\CE}=\begin{cases}
\bar{\Fc}_{1,\ell_{\CE}-1}&\text{ if } d_{\ell_{\CE}-1}>d_{2r-\ell_{\CE}+1},\\
\Fc_{1,\ell_{\CE}}&\text{ if } d_{\ell_{\CE}-1}<d_{2r-\ell_{\CE}+1}.
\end{cases}
$$
Since $r_{\CE_{i}}=\ell_{i+1}-1$ for all $1\leq i<h$ using  the definitions of $\Fc_{1,j}$ and $c_{1,j}$
 it is easy to check that $c_{\CE_{i}}=\Fc_{\CE_{i+1}}$.  The entries $\Fc_{\CE}$ 
are neither boxed nor circled.
Using this new decoration rule, define the weight function $G_{\Delta_{\Ft}(\CE)}$
$$
\begin{cases}
\prod^{2(r_{\CE}-\ell_{\CE})+1}_{j=1}\gamma_{\Delta}(\Fc(\CE)_{1,j})
\cdot q^{\Fc_{\CE}}(1-q^{-1})&\text{ if $\mu_{r+1-r_{\CE}}\ne \varepsilon_{r_{\CE}}k_{r_{\CE}}-\varepsilon_{r_{\CE}+1}k_{r_{\CE}+1}$ and $r_{\CE}<r$,}\\
\prod^{2(r_{\CE}-\ell_{\CE})+1}_{j=1}\gamma_{\Delta}(\Fc(\CE)_{1,j})&\text{ if $r_{\CE}=r$,}\\
\gamma_{\Delta}(\Fc_{1,r_{\CE}})\gamma_{\Gamma}(\bar{\Fc}_{1,r_{\CE}})& \text{ if $\mu_{r+1-r_{\CE}}=\varepsilon_{r_{\CE}}k_{r_{\CE}}-\varepsilon_{r_{\CE}+1}k_{r_{\CE}+1}$.}
\end{cases}
$$

 Applying an argument similar to the proofs of Lemmas~\ref{I:delta=delta'} and \ref{lm:II-delta=delta'},
 one sees that  if $k(\Ft)$ is strict then
\begin{equation}\label{eq:Psi-Delta}
G_{\Delta}(\Ft)=q^{-\sum^{2r-1}_{j=r}d_{j}}\prod^{h}_{i=1}G_{\Delta_{\Ft}(\CE_{i})}
\end{equation}
when $n$ divides $c_{\CE_{i}}$ for $1\leq i< h$, and $G_{\Delta}(\Ft)=0$ otherwise. 

If $\mu_{r+1-r_{\CE_{i}}}=\varepsilon_{r_{\CE_{i}}}k_{r_{\CE_{i}}}-\varepsilon_{r_{\CE_{i}}+1}k_{r_{\CE_{i}}+1}$ for some $i$, then $d_{r_{\CE_{i}}}=\mu_{r+1-r_{\CE_{i}}}$, $d_{2r-r_{\CE_{i}}}=0$ and $\Fc_{1,r_{\CE_{i}}}=c_{1,r_{\CE_{i}}}$, $\bar{\Fc}_{1,r_{\CE_{i}}}=\bar{c}_{1,r_{\CE_{i}}}$. Thus $G_{\Gamma_{\Ft(\CE_{i})}}=G_{\Delta_{\Ft(\CE_{i})}}$ and they vanish when $n$ does not divide $\bar{c}_{1,r_{\CE_{i}}}$.  

Recall that $c_{\CE_{i}}=\Fc_{\CE_{i+1}}$ for $1\leq i<h$.
By \eqref{eq:Psi} and \eqref{eq:Psi-Delta}, $G_{\Delta}(\Ft)$ and $G_{\Psi}(\Ft)$ vanish unless $n$ divides $c_{\CE_{i}}$ for all $1\leq  i< h$. Thus, we only need to consider the case that $n$ divides $c_{\CE_{i}}$ for all $1<i\leq h$.
And in that case, we must establish a relation between the $G_{\Delta_{\Ft}(\CE_{i})}$ 
and the $G_{\Gamma_{\Ft}(\CE_{i})}$.

Since $d_{i}=d_{2r-i}$ for all $\ell_{\CE}\leq i\leq r_{\CE}-1$, we have $\bar{\Fc}_{1,i}=\bar{c}_{1,i}$ and 
$\Fc_{1,i+1}=c_{1,i}$ for $\ell_{\CE}\leq i\leq r_{\CE}-1$. 
Also, for such $i$
the entries $\bar{\Fc}_{1,i}$ and $\bar{c}_{1,i}$, resp.\ $\Fc_{1,i+1}$ and $c_{1,i}$, have the 
same decorations in the decorated arrays $\Gamma_{\Ft}(\CE)$ and $\Delta_{\Ft}(\CE)$.
Hence, there is only one entry left to compare in  $\Gamma_{\Ft}(\CE)$ and $\Delta_{\Ft}(\CE)$.  We 
denote this entry by $\Fc^{-}_{\CE}$,  $c^{-}_{\CE}$, resp. Specifically, 
$$
\Fc^{-}_{\CE}=\begin{cases}
\bar{\Fc}_{1,\ell_{\CE}-1}&\text{ if } d_{\ell_{\CE}-1}<d_{2r-\ell_{\CE}+1},\\
\Fc_{1,\ell_{\CE}}&\text{ if  $d_{\ell_{\CE}-1}>d_{2r-\ell_{\CE}+1}$ or $\ell_\CE=1$,}
\end{cases}
\text{ and }
c^{-}_{\CE}=c(\CE)_{1,r_{\CE}-\ell_{\CE}+1}.
$$
In addition, by definition, for $1<i<h$
$$
\Fc_{\CE_{i}}+\Fc^{-}_{\CE_{i}}=\Fc_{1,\ell_{\CE_{i}}}+\bar{\Fc}_{\ell_{\CE_{i}}-1}=k_{\ell_{\CE_{i}}}(\Ft)=\bar{c}_{1,r_{\CE_{i}}}+\bar{c}_{1,r_{\CE_{i}}}=c_{\CE_{i}}+c^{-}_{\CE_{i}}=k_{r_{\CE_{i}}}(\Ft).
$$
Since we are in the case that $n$ divides all $c_{\CE_{i}}$, 
$n$ divides $\Fc^{-}_{\CE_{i}}-c^{-}_{\CE_{i}}=c_{\CE_{i}}-\Fc_{\CE_{i}}$ for each $\CE_{i}$. Since the decorations of $\Fc^{-}_{\CE_{i}}$ and $c^{-}_{\CE_{i}}$ are  the same, we deduce that
$$
G_{\Delta_{\Ft}(\CE_{i})}=q^{c_{\CE_{i}}-\Fc_{\CE_{i}}}\,G_{\Gamma_{\Ft}(\CE_{i})}.
$$
Since $c_{\CE_{i-1}}=\Fc_{\CE_{i}}$,
\begin{equation}\label{eq:delta-t-i}
\prod^{h-1}_{i=2}G_{\Delta_{\Ft}(\CE_{i})}=
\prod^{h-1}_{i=2}q^{c_{\CE_{i}}-c_{\CE_{i-1}}}G_{\Gamma_{\Ft}(\CE_{i})}=q^{c_{\CE_{h-1}}-c_{\CE_{1}}}
\prod^{h-1}_{i=2}G_{\Gamma_{\Ft}(\CE_{i})}.
\end{equation}
Moreover,
\begin{equation}\label{eq:delta-r-1}
G_{\Delta_{\Ft}(\CE_{1})}=q^{c_{\CE_{1}}}G_{\Gamma_{\Gamma_{\Ft}(\CE_{1})}}
\text{ and }
G_{\Delta_{\Ft}(\CE_{h})}=q^{\sum^{2r-1}_{j=r}d_{j}-c_{\CE_{h-1}}}G_{\Gamma_{\Gamma_{\Ft}(\CE_{h})}}.
\end{equation}
Multiplying~\eqref{eq:delta-t-i} and \eqref{eq:delta-r-1} together, we obtain
$$
q^{-\sum^{2r-1}_{j=r}d_{j}}\prod^{h}_{i=1} G_{\Delta_{\Ft}(\CE_{i})}=\prod^{h}_{i=1}G_{\Gamma_{\Ft}(\CE_{i})}.
$$
Then using \eqref{eq:Psi} and \eqref{eq:Psi-Delta}, the equality $G_{\Psi}(\Ft)=G_{\Delta}(\Ft)$ follows. 
\end{proof}

\section{The Stable Case}\label{sec9}

In this Section, we establish a conjecture of Brubaker, Bump and Friedberg \cite{BBF1, BBF2} for the prime power
coefficients $H(p^{\bf k};p^\ell)$
when $n$ is sufficiently large, the so-called `stable case.'  The formula is Lie-theoretic, and expresses the
non-zero coefficients as products of Gauss sums, with each non-zero coefficient 
attached via a bijection to an element of the Weyl group.
To do so, we make use of the
work of Beineke, Brubaker and Frechette \cite{BeBrF2}, who showed that in the stable case such contributions
could also be matched with stable strict generalized Gelfand-Tsetlin patterns (see \cite{BeBrF2}, Section 4).  
We recast their results in terms of decorated BZL-patterns and observe
 that their description matches the one given in Section~\ref{sec:even} for $n$ sufficiently large (either odd
or even).

Fix $\mu\in \Z^r_{\ge1}$ as above.  The condition that $n$ be sufficiently large depends on $\mu$, and is as follows.

{\bf Stability Assumption.} The degree $n$ of the cover satisfies
$$
n\geq \begin{cases}
\mu_{1}+2\sum^{r}_{i=2}\mu_{i} & \text{ if $n$ is odd,}\\
2\mu_{1}+4\sum^{r}_{i=2}\mu_{i} & \text{ if $n$ is even.}\\
\end{cases}
$$
Analogously to the definition of a stable 
Gelfand-Tsetlin pattern in \cite{BeBrF2}, Definition 10, we define a stable BZL-pattern.
\begin{definition}
A decorated BZL-pattern $\Gamma\in BZL(\mu)$ is called stable if  $G(\Gamma)\ne 0$ for some $n$ satisfying the Stability
Assumption.
\end{definition}

Given a BZL-pattern $\Gamma=\Gamma(c_{i,j})\in BZL(\mu)$, denote the entries of the $i$-th row
$$
(c_{i,i},c_{i,i+1},\dots,c_{i,r-1},c_{i,r},\bar{c}_{i,r-1},\dots,\bar{c}_{i,i+1},\bar{c}_{i,i}).
$$
Similar to the definitions for the first row $(c_{1,j})$, for the $i$-th row $(c_{i,j})$
define a short pattern $\Ft_{i}=(d_{i,i},d_{i,i+1},\dots,d_{i,2r-i})$ and the weight vectors $k_{i,j}$ for $\Ft_{i}$ where $i\leq j\leq r$.
Let $\mu_{1,j}=\mu_{j}$. Define $\mu_{i,j}$ inductively by 
$$\mu_{i,r+1-j}=\mu_{i-1,j}+k_{i-1,j-1}+\varepsilon_{j+1}k_{i-1,j+1}-2k_{i-1,j}$$ for $i>1$ and $i\leq j\leq r$, where $k_{i,r+1}=0$. Let ${\bf \mu_{i}}$ be the vector $(\mu_{i,1},\mu_{i,2},\dots,\mu_{i,r+1-i})$.
Then $\Ft_{i}$ is in $BZL_{1}({\bf \mu_{i}})$.  Define $G_{\Psi}(\Gamma)=\prod^{r}_{i=1}G_{\Psi}(\Ft_{i})$.

One may check the following two facts from the definitions.  
First, if a decorated BZL-pattern $\Gamma\in BZL(\mu)$
is not stable then $G_\Psi(\Gamma)=0$ for all $n$ satisfying the Stability Assumption.  
This follows since some non-zero entry must be undecorated but the divisibility
condition in (\ref{dec:Gamma}) can not hold.  Second, if a decorated
BZL-pattern $\Gamma\in BZL(\mu)$
is stable then $G_{\Psi}(\Gamma)=G(\Gamma).$  Indeed, this holds since the two sides are
identically equal by definition as the decorations are the same.

Let $\Phi(C)$ be the set of all roots of the dual group $\Sp_{2r}(\C)$ and 
 $W$ denote the Weyl group of $\Phi(C)$.  Then Proposition~12 of \cite{BeBrF2} and the
 following discussion may be recast in the language of 
stable BZL-patterns as follows.

\begin{lemma} \label{lm:stable}  Suppose that $\Gamma\in BZL(\mu)$ is a stable decorated BZL-pattern.
Then there exists a unique element $w$ in $W$ such that 
$$
\mu-w(\mu)=(2k_{r}(\Gamma)-k_{r-1}(\Gamma),k_{r-1}(\Gamma)-k_{r-2}(\Gamma),\dots,k_{1}(\Gamma)).
$$
Each element $w\in W$ is so-obtained.
\end{lemma}

Let $\Phi^{+}(C)$ (resp.\ $\Phi^{-}(C)$) denote the sets of positive (resp.\ negative) roots in $\Phi(C)$.
For each $w\in W$, let 
$$\Phi_{w}=\{\alpha\in\Phi^{+}(C)\mid w(\alpha)\in \Phi^{-}(C)\}.$$  Recall that given $\mu$, we may
realize it as $\mu=\lambda+\rho$ as in Section~\ref{sect62} above, with $\lambda=\sum l_i\epsilon_i$.
For $\alpha\in \Phi^+(C)$, let $d_\lambda(\alpha)=\frac{2<\mu,\alpha>}{<\alpha,\alpha>}$
where $\apair{\cdot,\cdot}$ is the standard Euclidean inner product.  (The notation $d_\lambda(\alpha)$ is the same as that of 
\cite{BBF2}.)
Then in view of the remarks above
the following Corollary follows
from Theorem~\ref{theorem-general} above together with Theorem~1 of \cite{BeBrF2}.

\begin{corollary}\label{lm:G-stable}  Suppose that the Stability Assumption holds.
Let $w$ be the Weyl element associated to $\Gamma$ as in Lemma~\ref{lm:stable}. Then
$$H(p^ {\kappa(\Gamma)};p^{\ell})=
G(\Gamma)=\prod_{\alpha\in \Phi_{w}}g_{\|\alpha\|^{2}}(p^{d_{\lam}(\alpha)-1},p^{d_{\lam}(\alpha)}).
$$
Each non-zero $p$-power coefficient is obtained from a Weyl group element $w$.
\end{corollary}

This confirms Conjecture 1.4 of \cite{BBF1} and the conjecture described in Section 1.1 
of \cite{BBF2} for root systems of this type.

\noindent
\textsc{\small{Department of Mathematics, Boston College, Chestnut Hill, MA
02467-3806}}

\end{document}